%% file: fully_discrete_carleman.tex
\documentclass[a4paper, 11pt, english]{article}
\usepackage[utf8]{inputenc}
\usepackage[T1]{fontenc}
\usepackage{graphicx}
\usepackage{stmaryrd}
\usepackage[a4paper]{geometry}
\usepackage{amsmath,amsfonts,amssymb,amsthm,epsfig,epstopdf,url,array}
\usepackage{rotating}
\usepackage[colorlinks=true,citecolor=red,linkcolor=blue,pdfpagetransition=Blinds]{hyperref}
\usepackage{cleveref}
\usepackage{nameref}
\usepackage[inline]{enumitem}
\usepackage{comment}
\Crefname{paragraph}{Section}{Sections}
\usepackage{fancyhdr}
\usepackage{fullpage}
\usepackage{mathrsfs}

\crefname{lem}{lemma}{lemmas}
\crefname{theo}{theorem}{theorems}

\providecommand{\keywords}[1]{\noindent {\textit{Keywords:}} #1}

\usepackage{color}
\newcommand{\R}{{\mathbb{R}}}

\renewcommand{\O}{{\mathcal{O}}}
\theoremstyle{plain} 
\newtheorem{prop}{Proposition}[section] 
\newtheorem{theo}[prop]{Theorem}

\newtheorem{lem}[prop]{Lemma}

\theoremstyle{definition}
\newtheorem{defi}[prop]{Definition}
\newtheorem{rmk}[prop]{Remark}

\newtheorem{claim}[prop]{Claim}
\newtheorem{ass}[prop]{Assumption}

\newcommand{\norme}[1]{\left\lVert#1\right\rVert}
\newcommand\dt{\triangle t}
\newcommand\sdt{\scriptscriptstyle \triangle t}
\newcommand\ssh{\scriptscriptstyle h}
\newcommand\mT{\mathcal P}
\newcommand\dmT{\mathcal D}
\newcommand\smT{{\scriptscriptstyle\mT}}
\newcommand\sdmT{{\scriptscriptstyle\dmT}}
\newcommand{\sh}{\mathtt s}
\newcommand{\shpx}[1]{(\mathtt s^+#1)}
\newcommand{\shmx}[1]{(\mathtt s^-#1)}
\newcommand{\difh}{\partial_h}
\newcommand{\delh}{\Delta_h}
\newcommand{\mh}{m_h}
\newcommand{\av}[1]{\tilde{#1}}
\newcommand{\avl}[1]{\widetilde{#1}}
\newcommand{\shb}{\bar{\mathtt s}}
\newcommand{\shbpx}[1]{({\bar{\mathtt s}}^+#1)}
\newcommand{\shbmx}[1]{({\bar{\mathtt s}}^-#1)}
\newcommand{\difhb}{\ov{\partial_h}}
\newcommand{\mbh}{\ov{m_h}}
\newcommand{\ov}[1]{\overline{#1}}

\newcommand{\sht}{\mathtt t}
\newcommand{\taup}[1]{(\mathtt{t}^+#1)}

\newcommand{\taum}[1]{(\mathtt{t}^-#1)}

\newcommand{\taubp}[1]{(\bar{\mathtt t}^+#1)}
\newcommand{\taubm}[1]{(\bar{\mathtt t}^-#1)}
\newcommand{\tbp}[1]{\bar{\mathtt t}^+#1}
\newcommand{\tbm}[1]{\bar{\mathtt t}^-#1}
\newcommand{\tcp}[1]{{\normalfont\textsf t}^+#1}
\newcommand{\tcm}[1]{{\normalfont\textsf t}^-#1}
\newcommand\print{\int}
\newcommand\Dtbar{\overline D_t}
\newcommand\Dt{D_t}
\newcommand\Dtc{\normalfont\textsf{D}_t}
\newcommand\ddbint{\,\dint\!\!\!\int}
\newcommand\dbint{\int\!\!\!\int}
\newcommand{\shc}{\normalfont\textsf{s}_{ h}}
\newcommand\dhc{\normalfont\textsf{d}_{ h}} 
\newcommand\dhctwo{\normalfont\textsf{d}_{ h2}} 
\newcommand\mhc{\normalfont\textsf{m}_{ h}} 
\newcommand\mhctwo{\normalfont\textsf{m}_{ h2}} 
\newcommand\avc[1]{\hat{#1}}
\newcommand\avcl[1]{\widehat{#1}}
\def\d{\textnormal{d}}
\newcommand\D{\displaystyle}
\newcommand{\abs}[1]{\left\lvert#1\right\rvert}

\def\dt{{\triangle t}}
\def\dx{{h}}

\makeatletter
\newcommand{\inlineitem}[1][]{%
\ifnum\enit@type=\tw@
    {\descriptionlabel{#1}}
  \hspace{\labelsep}%
\else
  \ifnum\enit@type=\z@
       \refstepcounter{\@listctr}\fi
    \quad\@itemlabel\hspace{\labelsep}%
\fi}
\makeatother
\input{macros_fully.tex}

\usetikzlibrary{matrix,external}
\usetikzlibrary{shapes,backgrounds}
\usetikzlibrary{shapes.misc}
\tikzset{cross/.style={cross out, draw=blue, fill=none, minimum size=2*(#1-\pgflinewidth), inner sep=0pt, outer sep=0pt}, cross/.default={4pt}}

\makeatletter
\let\original@addcontentsline\addcontentsline
\newcommand{\dummy@addcontentsline}[3]{}
\newcommand{\DeactivateToc}{\let\addcontentsline\dummy@addcontentsline}
\newcommand{\ActivateToc}{\let\addcontentsline\original@addcontentsline}
\makeatother

\begin{document}

\title{\bf Carleman estimates and controllability results for fully-discrete approximations of 1-D parabolic equations}

\author{Pedro Gonz\'alez Casanova\thanks{Instituto de Matem\'aticas, Universidad Nacional Aut\'onoma de M\'exico, Circuito Exterior, C.U., C.P. 04510 CDMX, Mexico.} \and V\'ictor Hern\'andez-Santamar\'ia${}^{*,}$\thanks{Corresponding author.  E-mail: \texttt{victor.santamaria@im.unam.mx}}}

\maketitle

\begin{abstract}

In this paper, we prove a Carleman estimate for fully-discrete approximations of parabolic operators in which the discrete parameters $h$ and $\triangle t$ are connected to the large Carleman parameter. We use this estimate to obtain relaxed observability inequalities which yield, by duality, controllability results for fully-discrete linear and semilinear parabolic equations.

\end{abstract}

\footnotesize
\tableofcontents
\normalsize

\section{Introduction}
For $L>0$, let us consider the open interval $\Omega=(0,L)$. Let $\omega$ be a nonempty subset of $\Omega$. We consider the linear control system given by
\begin{equation}\label{eq:linear_heat}
\begin{cases}
y_t-y_{xx}=\mathbf{1}_\omega v &\quad\text{in }(0,T)\times \Omega, \\
y(t,0)=y(t,L)=0 &\quad\text{n } (0,T) \\
y(0,x)=g(x) &\quad\text{in } \Omega.
\end{cases}
\end{equation}
In \eqref{eq:linear_heat}, $y=y(t,x)$ is the state, $v=v(t,x)$ is the control function acting on the system on $\omega$, and $g\in L^2(\Omega)$ is a given initial data. Here, $\mathbf{1}_\omega$ stands for the indicator function of the set $\omega$.

It is well-known that for any $g\in L^2(\Omega)$ and $v\in L^2(\omega\times(0,T))$, system \eqref{eq:linear_heat} has a unique weak solution such that $y\in C([0,T;L^2(\Omega)])\cap L^2(0,T;H_0^1(\Omega))$. This regularity motivates the following definition.

\begin{defi}
System \eqref{eq:linear_heat} is said to be null-controllable at time $T$ if for any $g\in L^2(\Omega)$, there exists a control $v\in L^2(\omega\times(0,T))$ such that the corresponding solution satisfies
\begin{equation*}
y(T,\cdot)=0 \quad\text{in } \Omega.
\end{equation*}
\end{defi}

It is by now well-known that \eqref{eq:linear_heat} is indeed null-controllable for any $T>0$ and any nonempty subset $\omega\subset \Omega$. In fact, this property holds in any dimension. The problem was addressed independently in the 90's in the seminal works by Lebeau \& Robbiano \cite{LR95} and Fursikov \& Imanuvilov \cite{FI96}. By duality, the null-controllability of \eqref{eq:linear_heat} is equivalent to the observability of the adjoint state. In more detail, for each $q_T\in L^2(\Omega)$, consider 
\begin{equation}\label{eq:adj_heat}
\begin{cases}
-q_t-q_{xx}=0 &\quad\text{in }(0,T)\times \Omega, \\
q(t,0)=q(t,L)=0 &\quad\text{n } (0,T) \\
q(T,x)=q_T(x) &\quad\text{in } \Omega.
\end{cases}
\end{equation}
Then, system \eqref{eq:linear_heat} is null-controllable if and only if there exists $C_{obs}>0$ such that the following observability inequality holds
\begin{equation}\label{eq:obs_intro_cont}
|q(0)|_{L^2(\Omega)}\leq C_{obs}\left(\dbint_{\omega\times(0,T)}|q|^2dxdt\right)^{1/2}.
\end{equation}

In this paper, our main interest is to study some controllability and observability properties for fully-discrete approximations of systems \eqref{eq:linear_heat} and \eqref{eq:adj_heat}, respectively. As it is pointed out in \cite{Zua05}, it is known that controllability/observability and numerical discretization do not commute well, however we expect to retain some properties. 

\subsection{Discrete setting}\label{sec:discrete_setting}

Hereinafter, we shall use the notation $\inter{a,b}=[a,b]\cap \mathbb N$ for any real numbers $a<b$. 

For given $N,M\in\mathbb N^*$, we set the space- and time-discretization parameters $\dt=T/M$ and $\dx=L/(N+1)$, respectively. We consider the pairs $(t_n,x_i)$ with $t_n=n\dt$, $n\in\inter{0,M}$, and $x_i=ih$, $i\in\inter{0,N+1}$. The numerical approximation of a function $f=f(t,x)$ at a grid point $(t_n,x_i)$ will be denoted as 
\begin{equation}\label{eq:notation_fully}
f_i^n:=f(t_n,x_i).
\end{equation}

We consider the following fully-discrete system
\begin{equation}\label{eq:fully_discr_heat}
\begin{cases}
\D \frac{y^{n+1}-y^{n}}{\dt} - \Delta_h y^{n+1}=\mathbf{1}_{\omega_h}v_h^{n+1} &  n\in\inter{0,M-1}, \\
y^{n+1}_0=y^{n+1}_{N+1}=0 & n\in \inter{0,M-1}, \\
y^0=g &i\in\inter{1,N},
\end{cases}
\end{equation}
where $y^{n}=\left(\begin{array}{c}y_1^n \\ \vdots \\ y_{N}^n\end{array}\right)$ is the state, $v^{n}=\left(\begin{array}{c}v_1^{n} \\ \vdots \\ v_{N}^n\end{array}\right)$ is the control, $g=\left(\begin{array}{c}g_1 \\ \vdots \\ g_N\end{array}\right)$ stands for the initial data $g(x)$ sampled at points $x_i$ and 
\begin{equation*}
\delh= \frac{1}{h^2}\left(
\begin{array}{ccccc}
-2 & 1 &  &  &  \\
1 & -2 & 1 &  &  \\
 & \ddots & \ddots & \ddots & \\
 &  & 1 & -2 & 1 \\
 &  &  & 1 & -2 \\
\end{array}\right).
\end{equation*}
In \eqref{eq:fully_discr_heat}, $\mathbf{1}_{\omega_h}$ stands for an approximation of the continuous indicator function $\mathbf{1}_\omega$. A natural choice is for instance the diagonal $N\times N$ matrix with entries 
\begin{equation*}
(\mathbf{1}_{\omega_h})_{i,i}=
\begin{cases}
1 &\text{if } x_i\in\omega, \\
0 &\text{if } x_i\notin\omega. 
\end{cases}
\end{equation*}
For simplicity, hereinafter, we simply denote by $\mathbf{1}_\omega$ this approximation of the indicator function.

Notice that the fully-discrete system  \eqref{eq:fully_discr_heat} is the result of applying a standard centered finite difference method for the space variable and an implicit Euler scheme for the time variable to system \eqref{eq:linear_heat}. Obviously, there are many other ways of discretize system \eqref{eq:linear_heat} but for the sake of simplicity we have chosen this method. 

As in the continuous 	case, we can introduce a notion of controllability for the fully-discrete scheme. More precisely, system \eqref{eq:fully_discr_heat} is said to be null controllable if for any initial datum $g\in \R^{N}$ there exists a sequence $\{v^{n+1}\}_{n\in\inter{0,M-1}}$ such that the corresponding solution satisfies 
\begin{equation}\label{eq:contr_intro}
y^{M}=0.
\end{equation}

Controllability results for discretized systems can be divided into two main categories: space-discrete and time-discrete results. Below we give a general panorama of the results available in the literature. 

\smallskip
\textbf{Space-discrete setting}. The controllability of semi-discrete (in space) approximations of parabolic systems has deserved a lot of attention in the recent past, see, for instance, \cite{LoZ98,LT06,BHL10,BHLR10ar,BLR14,NT14,ABM18,BHSdT19,AB20}.

It comes out that when addressing controllability problems for these kind of systems, the classical notion of null-controllability \eqref{eq:contr_intro} might be too strong since it may happen that the discrete problem is not uniformly controllable with respect to the discretization parameter. Actually, as shown in \cite{Zua05},  there are even counter-examples in 2-D for which a time-continuous variant of \eqref{eq:fully_discr_heat} is not even approximately controllable for given $h$. To handle this problem, it was proposed in \cite{LT06,Boy13} and other related works to relax the controllability requirements and consider the so-called $\phi(h)$-controllability. This notion roughly consists in constructing uniformly bounded controls such that the norm of the space-discrete solution $y_h(T) \sim \sqrt{\phi(h)}$ for some function $h\mapsto \phi(h)$ tending to 0 as $h\to 0$ and amounts to prove a \textit{relaxed} version of inequality {\eqref{eq:obs_intro_cont}} (cf. \cref{eq:obs_intro}). 

In this direction, it was shown in \cite{BHL10,BLR14,ABM18} (with three different approaches: Lebeau-Robbiano method \cite{LR95}, Carleman estimates \cite{FI96}, and moment's method \cite{FR74}, respectively) that, for a finite-difference scheme in space and a continuous time variable, the uniform $\phi(h)$-controllability property holds for functions $\phi(h)$ that do not tend to zero faster than some exponential $h\mapsto e^{-C/h^{\alpha}}$ (we refer to \cite{LT06} for a discussion of Galerkin approximations). 

\smallskip
\textbf{Time-discrete setting}. In the case where only time-discretization is used (i.e. the space variable remains continuous), the controllability results available in the literature are less compared to the previous setting. Most probably, this comes from the fact that, as pointed out by \cite{zhe08}, system \eqref{eq:fully_discr_heat} is not even approximately controllable for any given $\dt>0$, except for the trivial case $\omega=\Omega$. At the light of this negative result and following the spirit of the space-discrete case, a natural question that arises is whether the controllability constraint \eqref{eq:contr_intro} can be relaxed. In this direction, in \cite{zhe08}, the controllability of \eqref{eq:fully_discr_heat} is addressed by employing a time-discrete Lebeau-Robbiano strategy and controlling (uniformly with respect to $\dt$) the projections of solutions over a suitable class of low frequency Fourier components. In \cite{EV10}, the authors prove in a quite general framework that any controllable parabolic equation is null-controllable after time discretization by applying an adequate filtering of high frequencies. Finally, in \cite{BHS20}, the authors establish a Carleman-type estimate for time-discrete approximations of the parabolic operator $-\partial_t-\Delta$, allowing them to obtain a $\phi(\dt)$-controllability result where a small target is reached, that is,
\begin{equation}\label{eq:phi_dt_contr}
|y^M|_{L^2(\Omega)}\leq C\sqrt{\phi(\dt)}|g|_{L^2(\Omega)},
\end{equation}
where $C>0$ is uniform with respect to $\dt$ and $\dt\mapsto\phi(\dt)$ is a suitable function decaying exponentially as $\dt\to 0$. 

\smallskip

Regarding the controllability of fully-discrete approximations of parabolic systems, as far the author's knowledge, there are only two works addressing this problem. As we have mentioned, in \cite{EV10}, the authors prove that for any controllable parabolic equation, the discretization in time preserves some controllability properties (in the sense of \eqref{eq:phi_dt_contr}) after some filtering process. In addition, the authors prove that a similar result holds if a suitable discretization in space is performed. On the other hand, in the work \cite{BHLR11}, the authors extend the semi-discrete Lebeau-Robbiano method used in \cite{BHL10,BHLR10ar} and look for conditions between the discretization parameters $h$ and $\dt$ to prove controllability results for fully-discrete approximations of parabolic control problems. 

We remark that the above approaches rely on spectral analysis techniques and the results are thus limited to linear autonomous control systems. For this reason, in this paper we shall look to directly prove Carleman-type estimates for fully-discrete parabolic operators and employ them to prove some controllability results. The main goal and novelty of our approach will allow us to include in the analysis more general time-dependent coefficients, semi-linear systems or even coupled equations which are generally out of reach for the spectral techniques. 

In what follows, we shall introduce some notation that allow us to represent system \eqref{eq:fully_discr_heat} in a more compact way and also will help us to use later a formalism (differentiation, integration by parts, and so on) as close as possible to the continuous case. In particular, this will help us to carry out most of the computations in a very intuitive manner and clarify the exposition of our main results. First, we set the framework of the discretization in the space variable and then the one for dealing with the discretization in time. Afterwards, we make some comments about the combination of the settings. 

\subsubsection{Discretization-in-space}\label{sec:not_space}

From the discretization points $x_i$ introduced above, we define $\mesh:=\left\{x_i:i\in\inter{1,M}\right\}$. We refer to this points as the primal mesh (in space). As expected, we define the boundary values as $\partial\mesh=\{x_0,x_{N+1}\}$. Additionally, we introduce the points $x_{i+\frac{1}{2}}:=(x_{n+1}-x_n)/2$ for $i\in\inter{0,N}$ (see \Cref{fig:discr_space}). In what follows, the set of points $\dmesh:=\{x_{i+\frac{1}{2}}:i\in\inter{0,M}\}$ will be referred as the dual mesh (in space).

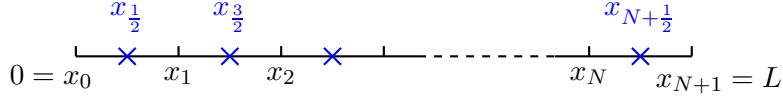
\begin{figure}[t!]
\raggedright\hspace{1.75 cm}
\begin{tikzpicture}[scale=0.9]
\draw[-,thick] (-4,0) -- (1,0);
\draw[dashed,thick] (1, 0) -- (3, 0);
\draw[-,thick] (3,0) -- (5.0,0);
\draw [thick] (-4,0) -- (-4,0.15);
\draw[thick,color=black] (-4.36,0.025) node[below]{$0=x_0$} ;
\draw [thick] (-2.5,0) node[below]{$x_1$} -- (-2.5,0.15);
\draw [thick] (-1.0,0) node[below]{$x_2$} -- (-1.0,0.15);
\draw [thick] (0.5,0) node[below]{} -- (0.5,0.15);
\draw [thick] (3.5,0) node[below]{$x_{N}$} -- (3.5,0.15);
\draw [thick] (5,0) -- (5,0.15);
\draw[thick,color=black] (5.4,0.005) node[below]{$x_{N+1}=L$} ;
\draw[thick] (-3.25,0) node[cross]{} -- (-3.25,0);
\draw[thick,color=blue] (-3.25,0.2) node[above]{$x_{\frac{1}{2}}$} ;
\draw[thick] (-1.75,0) node[cross]{} -- (-1.75,0);
\draw[thick,color=blue] (-1.75,0.2) node[above]{$x_{\frac{3}{2}}$} ;
\draw[thick] (-0.25,0) node[cross]{} -- (-0.25,0);
\draw[thick,color=blue] (4.25,0.2) node[above]{$x_{N+\frac{1}{2}}$} ;
\draw[thick] (4.25,0) node[cross]{} -- (4.25,0);
%
\put(175,10){$\mesh\cup\partial\mesh=(x_i)_{i\in\inter{0,N+1}}$}
\put(175,-10){$\textcolor{blue}{\dmesh=(x_{i+\frac{1}{2}})_{i\in \inter{0, M}}}$}
\end{tikzpicture}
\caption{Discretization of the space variable and its notation.}
\label{fig:discr_space}
\end{figure}

We denote by $\mathbb R^\smesh$ and $\R^{\sdmesh}$ the set of discrete functions defined on $\mesh$ and $\dmesh$, respectively. If $u^\smesh \in\R^{\smesh}$ (resp. $u^{\sdmesh}\in \R^{\sdmesh}$), we denote by $u_i$ (resp. $u_{i+\frac{1}{2}}$) its value corresponding to $x_i$ (resp. $x_{i+\frac{1}{2}}$). For $u^\smesh\in \R^\smesh$, we define the discrete integral 
\begin{equation}\label{eq:int_discr_primal}
\int_{\Omega}u^\smesh:=\sum_{i=1}^{N}h u_i
\end{equation}
and, analogously, for $u^{\sdmesh}\in \R^{\sdmesh}$ we define
\begin{equation}\label{eq:int_discr_dual}
\int_{\Omega}u^\sdmesh :=\sum_{i=0}^{N}h u_{i+\frac{1}{2}}.
\end{equation}

\begin{rmk}
Notice that the discrete integrals \eqref{eq:int_discr_primal} and \eqref{eq:int_discr_dual} are defined with the same symbol. Later, we will see that most of the time, from the context, we can infer which integral is being used. For this reason, to ease the notation, we simply write $u$ to denote functions $u^\smesh$ or $u^\sdmesh$.
\end{rmk}

For some $u\in\R^\smesh$, we shall need to associate boundary conditions $u^{{\scriptscriptstyle \partial}{\smesh}}=\{u_0,u_{N+1}\}$. The set of such discrete functions will be denoted by $\R^{\sfullmesh}$. Homogeneous Dirichlet boundary conditions consist in the choice $u_0=u_{N+1}=0$ and for short with write $u^{\scriptscriptstyle\partial\mesh}=0$ or $u_{|_{\partial\Omega}}=0$. 

With the notation above, we define the following $L^2$-inner product on $\R^\smesh$ (resp. $\R^{\sdmesh}$)
\begin{equation}\label{eq:def_inner_L2}
\begin{split}
&(u,v)_{L^2(\Omega)}:=\int_\Omega u\,v=\sum_{i=1}^{N}h\,u_i\, v_i \\
&\quad \left(\textnormal{resp.}\quad (u,v)_{L^2(\Omega)}:=\int_{\Omega}u v=\sum_{i=0}^{N}h\,u_{i+\frac{1}{2}}\, v_{i+\frac{1}{2}}\right).
\end{split}
\end{equation}
The associated norms will be denoted by $|u|_{L^2(\Omega)}$.  Analogously, we define the $L^\infty$-norm on $\R^{\smesh}$ (resp. $\R^{\sdmesh}$) as
\begin{align}\label{eq:def_linfty_space} 
&|u|_{L^\infty(\Omega)}:=\sup_{i\in\inter{1,N}}|u_i| \\ \label{eq:def_linfty_space_dual}
&\quad \left(\textnormal{resp.} \quad |u|_{L^\infty(\Omega)}:=\sup_{i\in\inter{0,N}}|u_{i+\frac{1}{2}}|\right).
\end{align}
Often times, we shall use functions restricted to (or defined on) subdomains, e.g., $\omega\subset \Omega$, where $\omega$ is a nonempty open set. Similar definitions and notations to \eqref{eq:def_inner_L2}--\eqref{eq:def_linfty_space_dual} will be employed for such functions. For instance, we define the discrete $L^2(\omega)$-norm on $\R^{\mesh}$ (resp. $\R^{\sdmesh}$) by
\begin{equation}\label{eq:norm_restricted}
\begin{split}
&|u|_{L^2(\omega)}:=\left(\sum_{i\in\inter{1,N},  \ x_i\in\omega } h |u_i|^2\right)^{1/2} \\
&\quad  \left(\textnormal{resp.}\quad  \abs{u}_{L^2(\omega)}:=\left(\sum_{i\in\inter{0,N}, \ x_{i+\frac12}\in\omega } h |u_{i+\frac12}|^2\right)^{1/2} \quad \right).
\end{split}
\end{equation}

In order to manipulate the discrete functions, we define the following translation operator for indices
\begin{equation*}
\shpx{u}_{i+\frac{1}{2}}:=u_{i+1}, \quad \shmx{u}_{i+\frac{1}{2}}:= u_i, \quad i\in\inter{0,N}.
\end{equation*}
A first-order difference operator $\difh$ and an averaging operator $\mh$ are then given by
\begin{align}
(\difh u)_{i+\frac{1}{2}}&:= \frac{u_{i+1}-u_{i}}{\dx}=\frac{1}{\dx}(\sh^+u - \sh^{-}u)_{i+\frac{1}{2}}, \\ \label{eq:average_dual}
(\mh u)_{i+\frac{1}{2}}&=\av{u}_{i+\frac{1}{2}}:=\frac{1}{2}(\sh^{+}u+\sh^{-}u)_{i+\frac{1}{2}}.
\end{align}
Both map $\R^{\sfullmesh}$ into $\R^{\sdmesh}$.

Likewise, we define on the dual mesh translation operators $\shb_h^{\pm}$ as follows
\begin{equation}
\shbpx{u}_{i}:=u_{i+\frac{1}{2}}, \quad \shbmx{u}_{i}:= u_{i-\frac{1}{2}}, \quad i\in\inter{1,N}.
\end{equation}
Then, a difference operator $\difhb$ and $\mbh$ (both mapping $\R^{\sdmesh}$ into $\R^\smesh$) are naturally given by
\begin{align}
(\difhb u)_{i}&:= \frac{u_{i+\frac{1}{2}}-u_{i-{\frac{1}{2}}}}{\dx}=\frac{1}{\dx}(\shb^+u - \shb^{-}u)_{i}, \\ \label{eq:average_primal}
(\mbh u)_{i}&=\ov{u}_{i}:=\frac{1}{2}(\shb^{+}u+\shb^{-}u)_{i}.
\end{align}
Observe that there is no need of boundary conditions at this point. Also notice, what with the above definitions, we can express $(\Delta_h u)_i$ as $\left(\difhb\difh u\right)_i$.

A continuous function $\psi$ defined on $\ov{\Omega}$ can be sampled on the primal mesh, that is $\psi^{\smesh}=\{\psi(x_i):i\in\inter{1,N}\}$, which we identify to
\begin{equation*}
\psi^{\smesh}=\sum_{i=1}^{N}\mathbf{1}_{[x_{i-\frac{1}{2}},x_{i+\frac{1}{2}}]}\psi_i, \quad \psi_{i}=\psi(x_i), \ i\in\inter{1,N}.
\end{equation*}
We also set 
\begin{align*}
&\psi^{\partial\smesh}=\{\psi(x_0),\psi_(x_{N+1})\}=\{\psi(0),\psi(L)\}, \\
&\psi^{\sfullmesh}=\left\{\psi(x_i):i\in\inter{0,N+1}\right\}.
\end{align*}
The function $\psi$ can be sampled also on the dual mesh, more precisely, 
\begin{equation*}
\psi^{\sdmesh}=\left\{\psi(x_{i+\frac{1}{2}}) : i\in\inter{0,N}\right\}
\end{equation*}
which we identify to 
\begin{equation*}
\psi^{\sdmesh}=\sum_{i=0}^{N}\mathbf{1}_{[x_i,x_{i+1}]}\psi_{i+\frac{1}{2}}, \quad \psi_{i+\frac{1}{2}}=\psi(x_{i+\frac{1}{2}}), \ i\in\inter{0,N}.
\end{equation*}
In the remainder of this document, we simply use the symbol $\psi$ for both the continuous function and its sampling on the primal or dual meshes. As we have mentioned before, from the context, one is able to identify the appropriate sampling. For instance, for a function $u$ defined on the primal mesh $\mesh$, in an expression like $\difhb(\rho\difh)$ where $\rho: \ov{\Omega}\to \R$ is a continuous given function, it is clear that $\rho$ is being sampled on the dual mesh $\dmesh$ since $\difh$ is defined on this mesh and the operator $\difhb$ acts on functions defined on this mesh as well. 

All of these definitions, together with the integrals \eqref{eq:int_discr_primal} and \eqref{eq:int_discr_dual}, allow us to obtain a series of results for handling in a quite natural way the application of the derivatives $\difh$ and $\difhb$ to continuous or discrete functions. For convenience, we have summarized in \Cref{app:discrete_things} the main tools and estimates for discrete functions (in space). As an example, for functions $u\in\R^{\sfullmesh}$ and $v\in\R^{\sdmesh}$, we have the following useful formula
\begin{equation*}
\int_{\Omega}u(\difhb v)=-\int_{\Omega}(\difh v) u + u_{N+1}g_{N+\frac{1}{2}}-u_{0}v_{\frac{1}{2}},
\end{equation*}
which resembles classical integration-by-parts formula. 

\subsubsection{Discretization-in-time}\label{sec:not_time}

Here, we devote to introduce some notations and definitions to handle effectively the discretization of the time variable. We recall that for given $M\in\mathbb N^*$, we have set $\dt=T/M$ and consider points $t_n=n\dt$, $n\in\inter{0,M}$. We also introduce the points $t_{n+\frac{1}{2}}:=(t_{n+1}-t_n)/2$ for $n\in\inter{0,M-1}$ (see \Cref{fig:discr_T}). 

\begin{figure}[t]
\raggedright\hspace{1.75 cm}
\begin{tikzpicture}[scale=0.9]
\draw[-,thick] (-4,0) -- (1,0);
\draw[dashed,thick] (1, 0) -- (3, 0);
\draw[-,thick] (3,0) -- (5.75,0);
\draw [thick] (-4,0) -- (-4,0.15);
\draw[thick,color=black] (-4.36,0.025) node[below]{$0=t_0$} ;
\draw [thick] (-2.5,0) node[below]{$t_1$} -- (-2.5,0.15);
\draw [thick] (-1.0,0) node[below]{$t_2$} -- (-1.0,0.15);
\draw [thick] (0.5,0) node[below]{} -- (0.5,0.15);
\draw [thick] (3.5,0) node[below]{$t_{M-1}$} -- (3.5,0.15);
\draw [thick] (5,0) -- (5,0.15);
\draw[thick,color=black] (5.4,0.005) node[below]{$t_M=T$} ;
\draw[thick] (-3.25,0) node[cross]{} -- (-3.25,0);
\draw[thick,color=blue] (-3.25,0.2) node[above]{$t_{\frac{1}{2}}$} ;
\draw[thick] (-1.75,0) node[cross]{} -- (-1.75,0);
\draw[thick,color=blue] (-1.75,0.2) node[above]{$t_{\frac{3}{2}}$} ;
\draw[thick] (-0.25,0) node[cross]{} -- (-0.25,0);
\draw[thick,color=blue] (4.25,0.2) node[above]{$t_{M-\frac{1}{2}}$} ;
\draw[thick] (4.25,0) node[cross]{} -- (4.25,0);
\draw[thick,color=blue] (5.75,0.2) node[above]{$t_{M+\frac{1}{2}}$} ;
\draw[thick] (5.75,0) node[cross]{} -- (5.75,0);
\put(175,10){$\overline{\mT}=(t_n)_{n\in\inter{0,M}}$}
\put(175,-10){$\textcolor{blue}{\overline{\dmT}=(t_{n+\frac{1}{2}})_{n\in \inter{0, M}}}$}
\end{tikzpicture}
\caption{Discretization of the time variable and its notation.}
\label{fig:discr_T}
\end{figure}
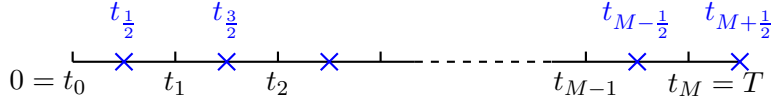

From these, we will denote by $\mT:=\left\{t_n : n\in\inter{1,M}\right\}$ the (primal) set of points excluding the first one and we write $\ov{\mT}:=\mT\cup\{t_0\}$. Analogous to the space variable, to handle the approximation of time derivatives, we will work with the (dual) points $t_{n+\frac{1}{2}}$. Its collection will be defined as $\dmT:=\{t_{n+\frac{1}{2}}:n\in\inter{0,M-1}\}$. It will be convenient to consider also an extra point $\{t_{M+\frac{1}{2}}\}$ which lies outside the time interval $[0,T]$ (see \Cref{fig:discr_T}) and to write $\ov{\dmT}:= \mT\cup\{T_{M+\frac{1}{2}}\}$. Observe that both $\mT$ and $\dmT$ have a total number of $M$ points. 

We denote by $\R^{\smT}$ and $\R^{\sdmT}$ the sets of real-valued discrete functions defined on $\mT$ and $\dmT$. If $u^{\smT}\in\R^{\smT}$ (resp. $u^\sdmT\in\R^{\sdmT}$), we denote by $u^n$ (resp. $u^{n+\frac{1}{2}}$) its value corresponding to $t_n$ (resp. $t_{n+\frac{1}{2}}$). For $u\in \R^{\smT}$ we define the time-discrete integral
\begin{equation}\label{eq:int_time_primal}
\int_{0}^{T}u^{\smT}:=\sum_{n=1}^{M}\dt\, u^n,
\end{equation}
and for $u^{\sdmT}\in \R^{\sdmT}$ we define 
\begin{equation}\label{eq:int_time_dual}
\dint_{0}^{T}u^{\sdmT}:=\sum_{n=0}^{M-1}\dt\, u^{n+\frac{1}{2}}.
\end{equation}
\begin{rmk}
For the time-discrete case, we have decided to use different symbols to define the integrals. For this reason, in what follows we shall write $u$ indistinctly to refer to functions $u^\smT$ or $u^{\sdmT}$. 
\end{rmk}

Let $\{X,|\cdot|_{X}\}$ be a real Banach space. Obviously, $X$ can be a finite or infinite dimensional space. We denote by $X^{\smT}$ and $X^{\sdmT}$ the sets of vector-valued functions defined on $\mT$ and $\dmT$, respectively. Using definitions \eqref{eq:int_time_primal} and \eqref{eq:int_time_dual}, we denote by $L_{\smT}^p(0,T;X)$ (resp. $L^p_{\sdmT}(0,T;X)$), $1\leq p<\infty$, the space $X^\smT$ (resp. $X^{\sdmT}$) endowed with the norm
\begin{equation*}
\norme{u}_{L^p_{\smT}(0,T;X)}:=\left(\int_{0}^{T}|u|^p_{X}\right)^{1/p} \quad \left(\textnormal{resp.} \quad \norme{u}_{L^p_{\sdmT}(0,T;X)}:=\left(\dint_{0}^{T}|u|^p_{X}\right)^{1/p}\right).
\end{equation*}

We also define the space $L^\infty_{\smT}(0,T;X)$ (resp. $L^\infty_{\sdmT}(0,T;X)$) by means of the norm
\begin{equation}\label{eq:def_Linfty_time}
\|u\|_{L^\infty_{\smT}(0,T;X)}:=\sup_{n\in\inter{1,M}} |u^n|_{X} \quad \left(\text{resp.}\quad  \|u\|_{L^\infty_{\sdmT}(0,T;X)}:=\sup_{n\in\inter{0,M-1}} |u^{n+\frac12}|_{X}  \right).
\end{equation}

In the case where $p=2$ and $X$ is replaced by a Hilbert space $\{H,(\cdot,\cdot)_{H}\}$, $H^{\smT}$ (resp. $H^{\sdmT}$) becomes a Hilbert space for the norm induced by the inner product
\begin{equation}\label{eq:inner_prod}
\int_{0}^{T}\left(u,v\right)_{H}:=\sum_{n=1}^{M}\dt\,(u^n,v^n)_{H}  \quad \left(\text{resp.}\quad \dint_{0}^{T}\left(u,v\right)_{H}:=\sum_{n=0}^{M-1}\dt\,(u^{n+\frac12},v^{n+\frac12})_{H}\right).
\end{equation}
%
%

To manipulate time-discrete functions, we define translation operators $\sht^+:X^\smT \to X^{\sdmT}$ and $\sht^-:X^{\overline\smT} \to X^{\sdmT}$ as follows:
\begin{equation*}
\taup{u}^{n+\frac12}:=u^{n+1}, \quad \taum{u}^{n+\frac 12}:=u^n, \quad n\in \inter{0,M-1}.
\end{equation*}
With this, we can define a difference operator $\Dt$ as the map from $X^{\overline\smT}$ into $X^\sdmT$ given by
\begin{equation*}
\begin{split}
&(\Dt u)^{n+\frac12}:= \frac{u^{n+1}-u^n}{\dt}=\left(\frac{1}{\dt}\left(\sht^+ - \sht^-\right)u\right)^{n+\frac12}, \quad n\in \inter{0,M-1}. \\
\end{split}
\end{equation*}

In the same manner, we can define the translation operators  $\bar{\sht}^+:X^{\overline\sdmT} \to X^{\smT}$ and $\bar{\sht}^-:X^{\sdmT} \to X^{\smT}$ as follows:
\begin{equation}\label{def_taus_dual}
\taubp u^{n}:=u^{n+\frac 12}, \quad \taubm u^{n}=u^{n-\frac 12}, \qquad n\in \inter{1,M}, \\
\end{equation}
as well as a difference operator $\Dtbar$ (mapping $X^{\overline{\sdmT}}$ into $X^{\smT}$) given by 
\begin{equation*}
(\Dtbar u)^{n}:= \frac{u^{n+\frac 12}-u^{n-\frac12}}{\dt}=\left(\frac{1}{\dt}\left(\bar{\mathtt t}^+-\bar{\mathtt t}^{-}\right)u\right)^{n}, \quad n\in \inter{1,M}.
\end{equation*}

As in the previous section, these definitions allow us to readily handle the evaluation of continuous functions on the primal and dual meshes $\mT$ and $\dmT$. Also, the notation here allow us to present an integration-by-parts (in time) formula in a natural way. For instance, in the case of \eqref{eq:inner_prod}, for functions $u\in H^{\ov{\smT}}$ and $v\in H^{\ov{\sdmT}}$we have
\begin{equation*}
\dint_0^{T}(\Dt u,v)_{H}=-(u^0,v^{\frac{1}{2}})_{H}+(u^{M},v^{M+\frac{1}{2}})_{H}-\int_{0}^{T}(\Dtbar v,u)_{H}.
\end{equation*}
A summary with other useful formulas and estimates are presented in \Cref{app:discrete_things}.

\subsubsection{Combining the effects of time- and space-discretization}\label{sec:not_full}
With the notation introduced in the time-discrete setting, we can combine easily the effects of the discretization in space and define the functional spaces we shall work with. Let us set $H=\R^{\smesh}$ (resp. $H=\R^{\sdmesh}$) endowed with the $L^2$ inner product \eqref{eq:def_inner_L2}. Observe that $(\R^\smesh)^\smT=\R^{\smesh\times\smT}$ (resp. $(\R^\sdmesh)^\smT=\R^{\sdmesh\times\smT}$). Then, for fully discrete functions $u\in \R^{\smesh\times\smT}$ (resp. $u\in \R^{\sdmesh\times\smT}$), we can combine \eqref{eq:def_inner_L2} and \eqref{eq:inner_prod} to define the $L^2$-inner product
\begin{align}\label{eq:inner_mp}
&\iint_{Q} u v :=\int_{0}^{T}(u,u)_{L^2(\Omega)}=\sum_{n=1}^{M}\dt \sum_{i=1}^{N}h\, u_{i}^{n}\,v_i^{n} \\ \label{eq:inner_mbarp}
&\quad \left(\textnormal{resp.}\quad  \dbint_{Q} u v :=\int_{0}^{T}(u,u)_{L^2(\Omega)}=\sum_{n=1}^{M}\dt \sum_{i=0}^{N}h\, u_{i+\frac{1}{2}}^{n}\,v_{i+\frac{1}{2}}^{n}  \quad \right).
\end{align}
As before, from the context, we will deduce which integral is being used for the space variable. The same idea holds for functions $u\in \R^{\smesh\times \sdmT}$ (resp. $u\in \R^{\sdmesh\times \sdmT}$). In this case, we have
\begin{align}\label{eq:inner_md}
&\ddbint_{Q} u v :=\dint_{0}^{T}(u,u)_{L^2(\Omega)}=\sum_{n=0}^{M-1}\dt \sum_{i=1}^{N}h\, u_{i}^{n+\frac{1}{2}}\,v_i^{n+\frac{1}{2}} \\ \label{eq:inner_mbard}
&\quad \left(\textnormal{resp.}\quad  \ddbint_{Q} u v :=\dint_{0}^{T}(u,u)_{L^2(\Omega)}=\sum_{n=0}^{M-1}\dt \sum_{i=0}^{N}h\, u_{i+\frac{1}{2}}^{n+\frac{1}{2}}\,v_{i+\frac{1}{2}}^{n+\frac{1}{2}}  \quad \right).
\end{align}

 For short and in accordance with the notation used in the continuous case, we will denote the Hilbert space $L^2_{\smT}(0,T;\R^{\smesh})$ (resp. $L^2_{\smT}(0,T;\R^{\sdmesh})$) induced by the inner product \eqref{eq:inner_mp} (resp. \eqref{eq:inner_mbarp}) as $L^2_{\smT}(Q)$. We introduce the analogous notation $L^2_{\sdmT}(Q)$ for functions belonging to the Hilbert spaces $L^2_{\sdmT}(0,T;\R^{\smesh})$ and $L^2_{\sdmT}(0,T;\R^{\sdmesh})$ with norms induced by \eqref{eq:inner_md}--\eqref{eq:inner_mbard}. 

As we have mentioned, often times we shall work with functions restricted or defined on subdomains. In this case, we also adopt a standard notation used in the continuous case. For instance, using \eqref{eq:norm_restricted}, we define
\begin{equation*}
\dbint_{\omega\times(0,T)} u v := \int_{0}^{T}(u,v)_{L^2(\omega)}=\sum_{n=1}^{M}\dt\left(\sum_{i\in\inter{1,N},\ x_i\in\omega } h\, u_i^n v_i^n\right).
\end{equation*}
We use analogous notations by changing the integral symbol $\int$ for $\dint$. The corresponding spaces induced by such inner products will be denoted by $L^2_{\smT}(\omega\times(0,T))$ and $L^2_{\sdmT}(\omega\times(0,T))$.

For our purposes, another important functional space to consider is $L^\infty_{\smT}(Q)$. This space can be naturally obtained by considering \eqref{eq:def_Linfty_time} and $X=\R^\smesh$ (resp. $X=\R^{\sdmesh}$) endowed with the norm \eqref{eq:def_linfty_space} (resp. \eqref{eq:def_linfty_space_dual}), this is, 
\begin{align*}
&\norme{u}_{L^\infty_{\smT}(Q)}:= \norme{u}_{L^\infty_{\smT}(0,T;\R^{\smesh})}=\sup_{n\in\inter{1,M}}\sup_{i\in\inter{1,N}} |u_i^{n}| \\
&\quad \left(\norme{u}_{L^\infty_{\smT}(Q)}:= \norme{u}_{L^\infty_{\smT}(0,T;\R^{\sdmesh})}=\sup_{n\in\inter{1,M}}\sup_{i\in\inter{0,N}} |u_{i+\frac12}^{n}| \right).
\end{align*}
Similar ideas can be used to construct and denote the space $L^\infty_{\sdmT}(Q)$.

\subsection{Statement of the main results}

\subsubsection{Carleman estimate}
To state one of our main result, we introduce several weight functions that will be used in the remainder of this paper. We begin by considering a function $\psi$ fulfilling the following assumption.
\begin{ass}\label{ass:phi}
Let $\mathcal B_0$ be a nonempty open set of $\Omega$. Let $\tilde{\Omega}$ be a smooth open and connected neighborhood of $\ov{\Omega}$ in $\R^N$. The function $x\mapsto \phi(x)$ is in $C^k(\ov{\tilde \Omega}; \R)$ with $k$ sufficiently large, and satisfies for some $c>0$
\begin{align*}
&\psi>0\quad\text{in } \tilde{\Omega}, \quad |\nabla \psi|\geq c \quad\text{in } \tilde{\Omega}\setminus\mathcal B_0, \\
&\text{and}\quad \partial_{n_x}\psi(x)\leq -c<0, \quad \text{for $x\in V_{\partial\Omega}$},
\end{align*}
where $V_{\partial\Omega}$ is a sufficiently small neighborhood of $\partial \Omega$ in $\tilde{\Omega}$, in which the outward unit normal $n_x$ is extended from $\partial \Omega$. 
\end{ass}

The construction of such function in general smooth domains is classical (see e.g. \cite[Lemma 1.1]{FI96}). In our particular one dimensional setting, we can take a point $x_0\in \mathcal B_0$ and consider $\psi(x)=C-(x-x_0)^2$ for $C>0$ large enough.

Let $K>\|\psi\|_{C(\ov{\Omega})}$ and consider a parameter $\lambda>0$. We set
\begin{equation}\label{eq:def_phi}
\varphi(x)=e^{\lambda\psi(x)}-e^{\lambda K}<0, \quad \phi(x)=e^{\lambda\psi(x)},
\end{equation}
and
\begin{equation}\label{eq:def_theta}
\theta(t)=\frac{1}{(t+\delta T)(T+\delta T-t)}
\end{equation}
for some $0<\delta<1/2$. The parameter $\delta$ is introduced to avoid singularities at time $t=0$ and $t=T$. 

We state our first result, a uniform Carleman estimate for the fully-discrete  backward parabolic operator formally defined on the dual grid (in time) as follows
\begin{equation}\label{eq:discr_operator_q}
(L_{\sdmT} q)^{n}:= -(\Dtbar q)^n-\delh\taubm{q}^n, \quad n\in\inter{1,M},
\end{equation}
for any $q\in (\R^{\smesh})^{\sdmT}$. The result is the following.

\begin{theo}\label{thm:fully_discrete_carleman}
Let $\mathcal B_0$ be a nonempty set of $\Omega$, $\psi$ be a function verifying \Cref{ass:phi} and define $\varphi$ according to \eqref{eq:def_phi}. Let $\mathcal B$ be another open subset of $\Omega$ such that $\mathcal B_0\subset\subset \mathcal B$. For the parameter $\lambda\geq 1$ sufficiently large, there exist $C>0$, $\tau_0\geq 1$, $h_0>0$, $\epsilon_0> 0$, depending on $\mathcal B$, $\mathcal B_0$, $T$ and $\lambda$ such that
\begin{align}\notag 
&\tau^{-1}\dbint_{Q}\tbm(e^{2\tau\theta\varphi}\theta^{-1})\left(|\Dtbar q|^2+|\delh\taubm{q}|^2\right)+\tau\dbint_{Q}\tbm(e^{2\tau\theta\varphi}\theta)|\difh\taubm{q}|^2 \\ \notag
&\quad +\tau\dbint_{Q}\tbm(e^{2\tau\theta\varphi}\theta)|\ov{\difh\taubm{q}}|^2+\tau^3\dbint_{Q} \tbm(e^{2\tau\theta\varphi}\theta^3)\taubm{q}^2 \\ \notag
&\leq C\left(\dbint_{Q}\tbm(e^{2\tau\theta\varphi})|L_{\sdmT}q|^2+\tau^3\dbint_{\mathcal{B}\times(0,T)}\tbm(e^{2\tau\theta\varphi}\theta^3)\taubm{q}^2\right) \\ \label{eq:car_fully_discrete}
&\quad   +C h^{-2}\left(\int_{\Omega}\abs{(e^{\tau\theta\varphi}q)^{\frac{1}{2}}}^2+\int_{\Omega}\abs{(e^{\tau\theta\varphi}q)^{M+\frac{1}{2}}}^2\right) 
\end{align}
for all $\tau\geq \tau_0(T+T^2)$, $0<h\leq h_0$, $\dt>0$ and $0<\delta\leq 1/2$ satisfying the conditions 
\begin{equation}\label{eq:cond_delta}
\frac{\tau^4\dt}{\delta ^4 T^6} \leq \epsilon_0 \quad\textnormal{and}\quad \frac{\tau h}{\delta T^2}\leq \epsilon_0,
\end{equation}
and $q$ is any fully-discrete function in $(\R^{\sfullmesh})^{\ov{\sdmT}}$ with $(q_{|\partial\Omega})^{n-\frac{1}{2}}=0$, $n\in\inter{1,M}$.
\end{theo}

To prove our fully-discrete Carleman estimate, we follow as close as possible the well-known methodology introduced in \cite{FI96} for the continuous setting. During the proof, we will use some tools and estimates that have been proved in other related works in the space- and time-discrete settings (see \cite{BLR14} and \cite{BHS20}, respectively), but some  results are new and specific to the fully-discrete case. We clarify this during the proof. 

Even though the parameter $\delta$ does not appear explicitly in \eqref{eq:car_fully_discrete} since it is hidden in the definition of \eqref{eq:def_theta}, this parameter plays a key role during the proof. We use it to avoid singularities at time $t=0$ and $t=T$ (which correspond to the case $\delta=0$ and are systematically used in the continous setting, see \cite{FI96,FCG06}, but which are somehow inconvenient at the discrete level). Here, by taking $\delta>0$, we are able to obtain time- and space-discrete derivatives of the Carleman weights and set a suitable change of variable, which is the starting point of the proof. 

\begin{rmk} The following remarks are in order. 
\begin{itemize}
\item We can readily recognize in \eqref{eq:car_fully_discrete} the structure of the continuous Carleman inequality (cf. \cite[Lemma 1.3]{FCG06}) but, as usual, the last two terms are specific to the discrete case. These terms appear naturally during the proof and cannot be avoided. Otherwise, we would have obtained a non-relaxed observability inequality leading to a uniform controllability result which is not true. Also notice that some conditions connecting the Carleman parameter $\tau$ and the discretization parameters $h$ and $\dt$ appear in \eqref{eq:cond_delta}. These basically state that, unlike the continuous case, we cannot take $\tau$ as large as we want without paying the price of decreasing the size of the discretization parameters. As we will see later, this is not a limitation for proving controllability results for fully-discrete systems. 
\item Even if the last two terms of \eqref{eq:car_fully_discrete} have a factor $h^{-2}$ in front, we can make them exponentially small by connecting parameter $\delta$ and $h$. Actually, for controllability purposes, we can take $\delta$ of order $h^{1/\vartheta}$ for a parameter $\vartheta\in\inter{1,4}$ to prove the smallness of these terms.   Note, however, that by doing so, a natural condition relating $h$ and $\dt$ should appear. As we will see later, we have to take $\dt$ of order $h^{4/\vartheta}$ to prove a controllability result for our system. A condition like this is expected and similar ones have appeared in other works addressing the controllability of fully-discrete parabolic systems, see \cite{BHLR11} and \cite{Boy13}. 
\end{itemize}
\end{rmk}

\subsubsection{Controllability results}
For a potential $a\in L^\infty_{\smT}(Q)$, we consider the sequence $y=\{y^n\}_{n\in\inter{0,M}}$ verifying the recursive formula
\begin{equation}\label{eq:heat_potential}
\begin{cases}
\D \frac{y^{n+1}-y^n}{\dt}-\delh y^{n+1}+a^{n+1}y^{n+1}=0 &n\in\inter{0,M-1}, \\
y^{n+1}_{|\partial\Omega}=0 &n\in\inter{0,M-1}, \\
y^0=g.
\end{cases}
\end{equation}
With the notation introduced previously, we can compactly rewrite \eqref{eq:heat_potential} as 
\begin{equation}\label{eq:heat_potential_compact}
\begin{cases}
(\Dt y)^{n+\frac{1}{2}}-\delh\taup{y}^{n+\frac{1}{2}}+\taup{ay}^{n+\frac{1}{2}}=0 &n\in\inter{0,M-1}, \\
y^0=g.
\end{cases}
\end{equation}
For convenience, we shall not make explicit the boundary conditions in such compact formulas since we will devote our analysis only to systems with homogeneous Dirichlet boundary conditions. 

Observe that the equation verified by $y$ is written in the dual (in time) grid $\dmT$. This motivates us to look for discrete controls defined on this grid and consider controlled systems of the form
\begin{equation}\label{eq:sys_fully_control}
\begin{cases}
(\Dt y)^{n+\frac{1}{2}}-\delh\taup{y}^{n+\frac{1}{2}}+\taup{ay}^{n+\frac{1}{2}}=\mathbf{1}_{\omega} v^{n+\frac{1}{2}}&n\in\inter{0,M-1}, \\
y^0=g.
\end{cases}
\end{equation}
Of course, we could have used the notation $\taup{v}^{n+\frac{1}{2}}$ in agreement with the control appearing in system \eqref{eq:fully_discr_heat} but, as we will see below, this is a more comfortable and natural approach.  

Following the so-called Hilbert Uniqueness Method (see \cite{GLH08}), it is possible to build a control function by minimizing a quadratic functional defined for the solutions to the adjoint of \eqref{eq:heat_potential}. The strategy is as follows. 

The adjoint equation to \eqref{eq:heat_potential} is given by recursive formula
\begin{equation}\label{eq:adj_heat_potential}
\begin{cases}
\D \frac{q^{n-\frac{1}{2}}-q^{n+\frac{1}{2}}}{\dt}-\delh q^{n-\frac{1}{2}}+a^{n}q^{n-\frac{1}{2}}=0 &n\in\inter{1,M}, \\
q_{|\partial\Omega}^{n-\frac{1}{2}}=0 &n\in\inter{1,M}, \\
q^{M+\frac{1}{2}}=q_T,
\end{cases}
\end{equation}
where $q_T\in\R^{\smesh}$ is a given initial datum. With our notation, we can rewrite \eqref{eq:adj_heat_potential} in the compact form
\begin{equation}\label{eq:adj_heat_compact}
\begin{cases}
-(\Dtbar q)^{n}-\delh\taubm{q}^{n}+a^n\taubm{q}^{n}=0 &n\in\inter{1,M}, \\
q^{M+\frac{1}{2}}=q_T.
\end{cases}
\end{equation}

Using the Carleman inequality \eqref{eq:car_fully_discrete}, we can prove a relaxed observability inequality of the form 
\begin{equation}\label{eq:obs_intro}
|q^{\frac{1}{2}}|_{L^2(\Omega)}\leq C_{obs} \left(\ddbint_{\omega\times(0,T)}|q|^2+e^{-C_2/h^{1/\vartheta}}|q_T|^2_{L^2(\Omega)}\right)^{\frac{1}{2}},
\end{equation}
where $\vartheta\in\inter{1,4}$ and the positive constants $C_2$ and $C_{obs}$ only depend on $T$, $\omega$ and $\norme{a}_{L^\infty_{\smT}(Q)}$. 

Comparing with the results obtained in the time- and space-discrete cases in \cite{BLR14} and \cite{BHS20}, a new parameter $\vartheta$ appears here due to the fully-discrete nature of our problem. Indeed, as we will see in \Cref{prop:obs_ineq}, inequality \eqref{eq:obs_intro} holds for parameters $\dt$ and $h$ chosen sufficiently small and such that $\dt\leq T^{-2} h^{4/\vartheta}$. In this way, as $h$ tends to $0$, $\dt\to 0$ as well and we recover the well-known result in the continuous case. See \Cref{sec:further_h_dt} for a further discussion.

With \eqref{eq:obs_intro} at hand, we can prove that the fully-discrete quadratic functional
\begin{equation*}
J_{\ssh,\sdt}(q_T)=\frac{1}{2}\ddbint_{\omega\times(0,T)}|q|^2+\frac{\phi(\dx)}{2}|q_T|_{L^2(\Omega)}+(g,q^{\frac{1}{2}})_{L^2(\Omega)}
\end{equation*}
has a unique minimizer $\widehat{q_T}$. By taking $v=\mathbf{1}_{\omega}\widehat{q}$, where $\widehat{q}$ is the solution to \eqref{eq:adj_heat_compact} with initial datum $\widehat{q_T}$, we can prove that $\norme{v}_{L^2_{\sdmT}(\omega\times(0,T))}\leq C$ uniformly with respect to $h$ and $\dt$ and that
\begin{equation}\label{eq:target_intro}
|y^{M}|_{L^2(\Omega)}\leq C \sqrt{\phi(h)}|g|_{L^2(\Omega)},
\end{equation}
where $h\mapsto \phi(h)$ is any given function of the discretization parameter such that
\begin{equation}\label{eq:inf_intro}
\liminf_{h\to 0}\frac{\phi(h)}{e^{-C_2/h^{1/\vartheta}}}>0
\end{equation}
and 
\begin{equation}\label{eq:con_intro}
\dt \leq T^{-2} h^{4/\vartheta}.
\end{equation}

This indicates that in fact we reach a small target $y^{M}$ whose size goes to zero as $h\to 0$ at the prescribed rate $\sqrt{\phi(h)}$ with controls that remain uniformly bounded with respect to $h$ and $\dt$. 

Our main controllability reads as follows.

\begin{theo}\label{thm:main_contr}
Let $T>0$, $\vartheta\in\inter{1,4}$ and assume that $h>0$ is sufficiently small. Then, for any $g\in\R^{\smesh}$, any function $\phi(h)$ verifying \eqref{eq:inf_intro}, and any $\dt>0$ verifying \eqref{eq:con_intro}, there exists a fully-discrete control $v\in L^2_{\sdmT}(\omega\times(0,T))$ such that
\begin{equation*}
\norme{v}_{L^2_{\sdmT}(\omega\times(0,T))}\leq C|g|_{L^2(\Omega)}
\end{equation*}
and such that the corresponding controlled solution $y$ to \eqref{eq:heat_potential} satisfies \eqref{eq:target_intro}, for a positive constant $C$ only depending on $\phi$, $T$, $\omega$ and $\norme{a}_{L^\infty_{\smT}(Q)}$.
\end{theo}

\begin{rmk}
In practice, the rate $\sqrt{\phi(h)}$ can be chosen in agreement with the accuracy of the discretization scheme, while the parameter $\vartheta$ gives us freedom at the moment of setting the relation between $h$ and $\dt$ (see \cref{eq:con_intro}), which typically amounts to choose $\dt \approx h$. We refer to \cite[Section 4]{Boy13} and \cite[Section 4]{BHLR11} for a more detailed discussion.
\end{rmk}

One of the main advantages of proving a fully-discrete Carleman estimate is that we can drop the spectral analysis techniques used in previous works (usually restricted to linear problems) and we can prove controllability results for semilinear systems. In fact, we can extend the previous theorem and study the controllability of the fully-discrete system 
\begin{equation}\label{eq:sys_semi}
\begin{cases}
\D \frac{y^{n+1}-y^{n}}{\dt}-\delh y^{n+1}+f(y^{n+1})=\mathbf{1}_{\omega} v^{n+\frac{1}{2}} &n\in\inter{0,M_0-1} \\
y^{n+1}_{|\partial \Omega}=0 & n\in \inter{0,M_0-1}, \\
y^0=g,
\end{cases}
\end{equation}
where $f\in C^1(\R)$ is a globally Lipschitz function with $f(0)=0$. The result reads as follows. 
\begin{theo}\label{thm:semilinear}
Let $T>0$, $\vartheta\in\inter{1,4}$ and assume that $h>0$ is sufficiently small. Then, for any $g\in\R^{\smesh}$, any function $\phi(h)$ verifying \eqref{eq:inf_intro}, and any $\dt>0$ verifying \eqref{eq:con_intro}, there exists a uniformly bounded fully-discrete control $v\in L^2_{\sdmT}(0,T;\R^\smesh)$ such that the associated controlled solution $y$ to \eqref{eq:sys_semi} verifies $|y^M|_{L^2(\Omega)}\leq C \sqrt{\phi(h)}|g|_{L^2(\Omega)}$, with $C>0$ only depending on $\phi$, $T$, $\omega$ and $f$.
\end{theo}

The proof of \Cref{thm:semilinear} follows other well-known controllability results for semilinear systems (see, for
instance, see e.g. \cite{FPZ95,FCZ00}). First, we prove the existence of a $\phi(h)$-null control for a linearized version of \eqref{eq:sys_semi} and then, after a careful analysis on the dependence of the constants appearing in \eqref{eq:obs_intro}, we can use a standard fixed point argument to deduce the result for the nonlinear case.

\subsection{Outline}

The rest of the paper is organized as follows. In \Cref{sec:fully_carleman} we present the proof of \Cref{thm:fully_discrete_carleman}. As in the continuous case, the proof roughly consists in writing a new equation after conjugation with the Carleman weight, splitting the resulting equation into two parts and then estimating the $L^2$ product between those two parts. To ease the reading, we have divided the proof in several parts and indicating clearly the procedure developed in each of them. 

\Cref{sec:control_results} is devoted to the application of the fully-discrete Carleman estimate to obtain controllability results. We divide it in two parts: in the first one, using estimate \eqref{eq:car_fully_discrete}, we derive a relaxed observability inequality, where we will pay special attention to the connection between the parameters $h$, $\dt$ and $\delta$. Then, this result is used to obtain the $\phi(h)$-controllability results in \Cref{thm:main_contr,thm:semilinear}. 

Finally, we devote \Cref{sec:further} to present additional results and remarks regarding fully-discrete Carleman estimates as well as the possible applications for handling less traditional control problems.

\section{Fully-discrete Carleman estimate}\label{sec:fully_carleman}
In this section, we present the proof of \Cref{thm:fully_discrete_carleman}. For the sake of presentation, we have divided the proof in various steps and lemmas. The ideas presented here follow as close as possible the proofs presented in the classical continuous setting (see e.g. \cite{FI96,FCG06}).

As in other related works, we will keep track of the dependence of all constants with respect to $\lambda$, $\tau$ and $T$. Also, in accordance with the discrete nature of our problem, we will pay special attention of the dependence with respect to the discrete parameters $h$, $\dt$ and $\delta$. 

\subsection{The change of variable}

To abridge the notation, we introduce the following instrumental functions
\begin{align*}
s(t)=\tau\theta(t), \quad \tau>0, \quad t\in(-\delta T,T+\delta T), \\
r(t,x)=e^{s(t)\varphi(x)}, \quad \rho(t,x)=(r(t,x))^{-1}, \quad x\in \ov{\tilde\Omega}, \quad t\in(-\delta T,T+\delta T).
\end{align*}
For $q\in(\R^\smesh)^{\ov{\sdmT}}$ we introduce the change of variables
\begin{equation*}
z^{n+\frac{1}{2}}=r^{n+\frac{1}{2}} q^{n+\frac{1}{2}}, \quad n\in\inter{0,M},
\end{equation*}
with $r^{n+\frac{1}{2}}$ defined as
\begin{equation*}
r^{n+\frac{1}{2}}=\left(\begin{array}{c}r^{n+\frac{1}{2}}_1 \\ \vdots \\ r_N^{n+\frac{1}{2}}\end{array}\right).
\end{equation*}
where we recall notation \eqref{eq:notation_fully}.

As in \cite{BLR14}, the enlarged neighborhood $\tilde \Omega$ of $\Omega$ in \Cref{ass:phi} allows us to apply multiple discrete operations such as $\difh$, $\mh$ on the weight functions. In particular, from the construction of $\psi$, we have
\begin{equation}\label{eq:deriv_weight_boundary}
(r \ov{\difh \rho})_0^{n+\frac{1}{2}}\leq 0, \quad (r \ov{\difh \rho})_{N+1}^{n+\frac{1}{2}}\geq  0, \quad n\in\inter{0,M-1}.
\end{equation}

Hereinafter, we will simplify the notation in such kind of formulas by omitting the superscript $n$ and simply write $z=r q$ which implicitly means that the continuous function $r$ is evaluated on the same time-grid (primal or dual) and at the same time point as the one attached to the discrete variable $q$. 

Our first task is to obtain an expression of the equation verified by $z$. From the change of variables proposed, we have that $q=\rho z$ and using \Cref{lem:chain_rule_space} we readily see
\begin{align}\notag 
\delh[\rho z]&=\difhb{\difh}[\rho z]= \difhb(\difh \rho)\ov{\av{z}}+\difhb\difh(z)\ov{\av{\rho}}+2\ov{\difh{z}}\,\ov{\difh{\rho}} \\ \label{eq:iden_lap}
&=\delh \rho \ov{\av{z}}+\delh{z}\ov{\av{\rho}}+2\ov{\difh z}\,\ov{\difh{\rho}}.
\end{align}
This clearly resembles the classical Leibniz formula in the continuous case. Then, using the translation operator \eqref{def_taus_dual} we can write the following equality on the primal grid $\mT$
\begin{align}\label{eq:iden_change_var_space} 
\tbm(\delh{[\rho z]})&=\tbm(\delh{\rho}\, \ov{\av{z}})+\tbm(\delh{{z}}\, {\ov{\av{\rho}}})+2\tbm(\ov{\difh {z}}\,\ov{\difh{{\rho}}}).
\end{align}

On the other hand, using \eqref{deriv_prod}, we have
\begin{equation}\label{eq:iden_change_var_time}
\Dtbar(\rho z)=\taubm{\rho}\,\Dtbar z + \Dtbar{\rho}\,\taubp{z}.
\end{equation}

Thus, multiplying \eqref{eq:iden_change_var_space}, \eqref{eq:iden_change_var_time} by $\taubm{r}$, adding the resulting expressions and using identity \eqref{eq:discr_operator_q}, we obtain
\begin{align}
\Dtbar{z}+\taubm{r}\Dtbar{\rho}\taubp{z}+\tbm(r \delh\rho \ov{\av{z}})+\tbm(r \delh z \ov{\av{\rho}})+2 \tbm(r\ov{\difh{\rho}}\,\ov{\difh{z}}) = \taubm{r}(L_{\sdmT}q).
\end{align}
Using \Cref{lem:deriv_lemma_time} in the second term of the above expression yields
\begin{align*}
&\Dtbar{z}-\tau\taubm{\theta^\prime}\varphi\taubp{z}+\dt\left(\frac{\tau}{\delta^3T^4}+\frac{\tau^2}{\delta^4T^6}\mathcal{O}_{\lambda}(1)\right)\taubp{z}+\tbm(r \delh\rho \ov{\av{z}}) \\
&\quad +\tbm(r \delh z \ov{\av{\rho}})+2 \tbm(r\ov{\difh{\rho}}\,\ov{\difh{z}}) = \taubm{r}(L_{\sdmT}q).
\end{align*}
whence, using that $\dt \Dtbar{z}=\taubp{z}-\taubm{z}$, we get
\begin{align}\notag 
&\Dtbar{z}-\tau\varphi \,\tbm(\theta^\prime z)+\tbm(r \delh\rho \ov{\av{z}}) +\tbm(r \delh z \ov{\av{\rho}})+2 \tbm(r\ov{\difh{\rho}}\,\ov{\difh{z}})  \\ \label{eq:iden_eq_z}
&= \taubm{r}(L_{\sdmT}q) - \dt\left(\frac{\tau}{\delta^3T^4}+\frac{\tau^2}{\delta^4T^6}\right)\mathcal{O}_{\lambda}(1)\taubp{z}+\tau\dt\taubm{\theta^\prime}\varphi \Dtbar{z}.
\end{align}

In the spirit of the continuous case, we rearrange \eqref{eq:iden_eq_z} and write it in the form
\begin{equation}\label{eq:new_iden_z}
Az+Bz=R,
\end{equation}
where $Az=A_1z+A_2z+A_3z$, $Bz=B_1z+B_2z+B_3z$ with
\begin{align}\label{eq:def_Az}
&A_1 z=\tbm(r\ov{\av{\rho}}\delh z), \quad A_2 z=\tbm(r\delh \rho \ov{\av{z}}), \quad A_3z=-\tau\varphi \tbm(\theta^\prime z), \\ \label{eq:def_Bz}
&B_1z=2\tbm(r\ov{\difh\rho}\,\ov{\difh{z}}), \quad B_2z=-2\tbm(s\partial_{xx}\phi z), \quad B_3z=\Dtbar{z},
\end{align}
and 
\begin{equation}
R=\taubm{r}(L_{\sdmT}q) - \dt\left(\frac{\tau}{\delta^3T^4}+\frac{\tau^2}{\delta^4T^6}\right)\mathcal{O}_{\lambda}(1)\taubp{z}+\tau\dt\taubm{\theta^\prime}\varphi \Dtbar{z} - 2\tbm(s\partial_{xx} \phi z).
\end{equation}
With the notation introduced in \Cref{sec:discrete_setting}, we take the $L^2$-norm in \eqref{eq:new_iden_z} and obtain
\begin{equation}\label{eq:iden_carleman}
\norme{Az}_{L^2_{\smT}(Q)}^2+\norme{Bz}_{L^2_{\smT}(Q)}^2+2(Az,Bz)_{L^2_{\smT}(Q)} =\norme{R}_{L^2_{\smT}(Q)}^2.
\end{equation}

As in other works devoted to prove Carleman estimates, the rest of the proof consists on estimating the term $(Az,Bz)_{L^2_{\smT}(Q)}$. For clarity, we have divided it in several steps. Developing the inner product $(Az,Bz)_{L^2_{\smT}(Q)}$, we set $I_{ij}:=(A_iz,B_iz)_{L^2_{\sdmT}(Q)}$. 

\subsection{Estimates involving only space-discrete computations}

In this step, we provide estimates for the terms $I_{ij}$ with $i,j=1,2$. Such terms can be estimated by only operating on the discrete variable in space. Indeed, the discrete time variable plays a very minor role and the proofs can be easily adapted from similar estimates already presented in  \cite{BLR14} for the space-discrete case. However, the computations shown in that work are made for a more general framework and with a heavier notation. For this reason, we give short and concise proofs of these estimates on \Cref{app:proofs}.

The first estimate reads as follows.
\begin{lem}[Estimate of $I_{11}$] \label{lem:est_I11}
For $\dt \tau (T^3\delta^2)^{-1}\leq 1$ and $\tau h (\delta T^2)^{-1}\leq 1$, we have
\begin{equation*}
I_{11}\geq -\tau \lambda^{2}\dbint_{Q} \taubm{\theta} \phi |\partial_x \psi|^2 |\difh{\taubm{z}}|^2+ Y^{h}_{11}-X_{11}^h,
\end{equation*}
with 
\begin{equation*}
Y_{11}^h=\print_{0}^{T} (1+\taubm{sh}^2\mathcal O_{\lambda}(1)) \left[ \tbm(r \ov{\difh{\rho}})_{N+1} \left|\difh{\taubm{z}}\right|^2_{N+\frac{1}{2}} - \tbm(r\ov{\difh \rho})_0 \left|\difh{\taubm{z}}\right|_{\frac{1}{2}}^2 \right]
\end{equation*}
and
\begin{equation*}
X_{11}^{h}=\dbint_{Q} \taubm{\nu_{11}} |\difh{\taubm{z}}|^2, \quad \nu_{11}=s\lambda\mathcal O(1)+s(sh)^2\mathcal O_{\lambda}(1).
\end{equation*}
\end{lem}

For the term $I_{12}$, we have the following.

\begin{lem}[Estimate of $I_{12}$]\label{lem:est_I12}
For $\dt \tau (T^3\delta^2)^{-1}\leq 1$ and $\tau h (\delta T^2)^{-1}\leq 1$, we have
\begin{equation*}
I_{12}\geq 2\tau \lambda^2 \dbint_{Q} \taubm{\theta}\phi |\partial_x\psi|^2 |\difh\taubm{z}|^2-X_{12}^h,
\end{equation*}
with
\begin{equation*}
X_{12}=\dbint_{Q}\taubm{\nu_{12}}|\difh\taubm{z}|^2+\dbint_{Q}\taubm{\mu_{12}}\taubm{z}^2,
\end{equation*}
where $\mu_{12}=s^2\mathcal O_{\lambda}(1)$ and $\nu_{12}=s\lambda\phi \mathcal O(1)+[1+s(sh)]\mathcal O_{\lambda}(1)$.
\end{lem}

The term $I_{21}$ is estimated in the following result.
\begin{lem}[Estimate of $I_{21}$]\label{lem:est_I21}
For $\dt \tau (T^3\delta^2)^{-1}\leq 1$ and $\tau h (\delta T^2)^{-1}\leq 1$, we have
\begin{align*}
I_{21}\geq 3\tau^3\lambda^4\dbint_{Q}\taubm{\theta}^3 \phi^3 |\partial_{x}\psi|^4\taubm{z}^2+Y_{21}^h-X_{21}^h,
\end{align*}
with
\begin{align*}
Y_{21}^h&=\print_0^{T}\taubm{sh}^2\mathcal O_\lambda(1)\left\{\tbm\left(r\ov{\difh\rho}\right)_0 \left|\difh \taubm{z}\right|_{N+\frac{1}{2}}^2+\tbm\left(r\ov{\difh\rho}\right)_{N+1} \left|\difh \taubm{z}\right|_{N+\frac{1}{2}}^2\right\}
\end{align*}
and
\begin{equation*}
X_{21}^{h}=\dbint_{Q} \taubm{\mu_{21}} \taubm{z}^2 + \dbint_{Q} \taubm{\nu_{21}}|\difh \taubm{z}|^2,
\end{equation*}
where $\mu_{21}=(s\lambda\phi)^3\O(1)+s^2\O_{\lambda}(1)+s^3(sh)^2\O_{\lambda}(1)$ and $\nu_{21}=s(sh)^2\O_{\lambda}(1)$.
\end{lem}

Finally, $I_{22}$ can be estimate as follows.

\begin{lem}[Estimate of $I_{22}$]\label{lem:est_I22}
For $\dt \tau (T^3\delta^2)^{-1}\leq 1$ and $\tau h (\delta T^2)^{-1}\leq 1$, we have
\begin{align*}
I_{21}\geq -2\tau^3\lambda^4\dbint_{Q}\taubm{\theta}^3 \phi^3 |\partial_{x}\psi|^4\taubm{z}^2-X_{22}^h,
\end{align*}
with
\begin{equation*}
X_{22}^{h}=\dbint_{Q} \taubm{\mu_{22}} \taubm{z}^2 + \dbint_{Q} \taubm{\nu_{21}}|\difh \taubm{z}|^2,
\end{equation*}
where $\mu_{22}=(s\lambda\phi)^3\O(1)+s^2\O_{\lambda}(1)+s^3(sh)\O_{\lambda}(1)$ and $\nu_{22}=s(sh)^2\O_{\lambda}(1)$.
\end{lem}

\subsection{Estimates involving time-discrete computations}

In this step, we will estimate the terms $I_{31}, I_{32}$, and $I_{33}$. These terms require time-discrete computations combined with some space-discrete operations but can be done without major issues. In fact, the effects of both schemes can be manipulated in a separated way and do not require any major change as compared to the works \cite{BLR14} or \cite{BHS20}.

We have the following result. 

\begin{lem}[Estimate of $I_{31}$]\label{lem:est_I31} For $\dt \tau (T^3\delta^2)^{-1}\leq 1$ and $\tau h(\delta T^2)^{-1}\leq 1$, we have
\begin{equation*}
I_{31}\geq - X_{31}^{{\scriptscriptstyle h,}\sdt},
\end{equation*}
with
\begin{equation*}
X_{31}^{{\scriptscriptstyle h,}\sdt}=\dbint_{Q}\taubm{\mu_{31}}\taubm{z}^2+\dbint_{Q}\taubm{\nu_{31}}|\difh\taubm{z}|^2,
\end{equation*}
where $\mu_{31}=T s^2\theta \mathcal O_{\lambda}(1)$ and $\nu_{31}=T(sh)^2\theta \mathcal O_{\lambda}(1)$.
\end{lem}
The proof of this result can be found in \Cref{app:proof_time_easy}.

For the term $I_{32}$, we have the following.

\begin{lem}[Estimate of $I_{32}$]\label{lem:est_I32} For $\dt \tau (T^3\delta^2)^{-1}\leq 1$ and $\tau h (\delta T^2)^{-1}\leq 1$, we have
\begin{equation*}
I_{33}\geq - X_{32}^{\sdt}
\end{equation*}
with
\begin{equation*}
X_{32}^{\sdt}=\dbint_{Q}\taubm{\mu_{32}}\taubm{z}^2,
\end{equation*}
where $\mu_{32}=T s^2\theta \mathcal O_{\lambda}(1)$.
\end{lem}

\begin{proof}
The proof follows from a direct computation and taking into accout that $|\theta^\prime|\leq CT\theta^2$, $\|\varphi\|_{C(\ov{\Omega})}=\O_{\lambda}(1)$ and $\norme{\partial_{xx}\phi}_{\infty}=\O_{\lambda}(1)$.
\end{proof}

\begin{lem}[Estimate of $I_{33}$]\label{lem:est_I33} For $\dt \tau (T^3\delta^2)^{-1}\leq 1$ and $\tau h (\delta T^2)^{-1}\leq 1$, we have
\begin{equation*}
I_{33}\geq - X_{33}^{\sdt}- W_{33}^{\sdt},
\end{equation*}
with
\begin{equation*}
X_{33}^{\sdt}=\dbint_{Q}\taubm{\mu_{33}}\taubm{z}^2
\end{equation*}
and
\begin{equation}\label{eq:def_W33}
W_{33}^{\sdt}=\dbint_{Q}\taubm{\gamma_{33}}(\Dtbar z)^2,
\end{equation}
where $\mu_{33}=(\tau T^2\theta^3+\frac{\tau \dt}{\delta^4T^5})\mathcal O_{\lambda}(1)$ and $\gamma_{33}=\dt \tau T \theta^2 \mathcal O_{\lambda}(1)$
\end{lem}
The proof of this result can be found in \Cref{app:proof_time_easy}.

\subsection{A new fully-discrete estimate}

Here, we present a result that arises specifically in the fully-discrete case. Indeed, the combined effect of time- and space-discrete computations is needed and the proof requires of some new estimates shown in \Cref{app:discrete_things}. The result reads as follows.

\begin{lem}[Estimate of $I_{13}$+$I_{23}$] \label{lem:est_I13} For $\dt \tau^2 (\delta ^4T^6)^{-1}\leq \epsilon_1(\lambda)$ and $\tau h(\delta T^2)^{-1}\leq \epsilon_1(\lambda)$, there exist constants $c_\lambda,c_\lambda^\prime>0$ uniform with respect to $\dt$ and $h$ such that
\begin{align}\notag 
I_{13}+I_{23}&\geq  c^\prime_{\lambda}\dt \dbint_{Q} \left(\Dtbar(\difh z)\right)^2 - c_{\lambda} \int_{\Omega} \left| (\difh z) ^{M+\frac{1}{2}}\right|^2- X_{+}^{{\scriptscriptstyle h,}\sdt} -W_{+}^{{{\scriptscriptstyle h,}\sdt}} \\ \notag
& \quad - \int_{\Omega} (s^{M+\frac{1}{2}})^2\mathcal O_{\lambda}(1)|z^{M+\frac{1}{2}}|^2-\int_{\Omega} (s^{\frac{1}{2}})^2\mathcal O_{\lambda}(1)|z^{\frac{1}{2}}|^2 \\ \label{eq:est_I13_I23}
& \quad - \int_{\Omega} (s^{M+\frac{1}{2}}h)^2\mathcal O_{\lambda}(1)|(\difh z)^{M+\frac{1}{2}}|^2-\int_{\Omega} (s^{\frac{1}{2}}h)^2\mathcal O_{\lambda}(1)|(\difh z)^{\frac{1}{2}}|^2,
\end{align}
with 
\begin{equation*}
X_{+}^{{{\scriptscriptstyle h,}\sdt}}=\dbint_{Q} \taubm{\mu_+}\taubm{z}^2 + \dbint_{Q}\taubm{\nu_{+}} |\difh \taubm{z}|^2
\end{equation*}
and
\begin{equation}\label{eq:def_W+}
W_{+}^{{{\scriptscriptstyle h,}\sdt}}= \dbint_{Q}\taubm{\gamma_{+}} (\Dtbar z)^{2},
\end{equation}
where 
\begin{align*}
\nu_{+}&:= \left\{T \theta (sh)^2 +s(sh)^2+\left(\frac{\tau \dt}{\delta ^3 T^4}\right)\left(\frac{\tau h}{\delta^{} T^2}\right) + \left(\frac{\tau^2 \dt}{\delta^4 T^6}\right)+\left(\frac{\tau \dt}{\delta ^3 T^4}\right)\left(\frac{\tau h}{\delta^{} T^2}\right)^3\right\}\mathcal O_{\lambda}(1),  \\
\mu_{+}&:=\left\{T s^2\theta + \left(\frac{\tau^2 \dt}{\delta^4 T^6}\right)+\left(\frac{\tau \dt}{\delta ^3 T^4}\right)\left(\frac{\tau h}{\delta^{} T^2}\right)^3\right\}\mathcal O_{\lambda}(1), \\
\gamma_{+}&:= \left\{s^{-1}(sh)^2+\dt s^2 \right\}\O_{\lambda}(1).
\end{align*}
\end{lem}

The proof of this result can be found in \Cref{sec:new_estimate}. Although very similar, this result is new as compared to the works \cite{BLR14} and \cite{BHS20}. Here, we decided to estimate $I_{13}$ and $I_{23}$ together since the first term in the right-hand side of \eqref{eq:est_I13_I23} (which is positive) arises in the estimation of $I_{13}$ and allows to control similar terms obtained while estimating $I_{23}$. This comes at the price of imposing new size constraints on the parameters $\tau, h$ and $\dt$.

\begin{rmk}
 We remark that in the continuous case the estimate of $I_{13}+I_{23}$ reduces to merely the first term of $\mu_+$. In particular, $I_{23}=0$.
 \end{rmk}

\subsection{Towards the Carleman estimate} 

In this step, we start to build our fully-discrete Carleman inequality. First, we present the following estimate for the right-hand side. 

\begin{lem}[Estimate of the norm of $R$]\label{lem:est_rhs}
For $\dt \tau (T^3\delta^2)^{-1}\leq 1$ and $\tau h (\delta T^2)^{-1}\leq 1$, there exists a constant $C_{\lambda}>0$ uniform with respect to $h$ and $\dt$ such that 
\begin{align*}
\norme{R}_{L^2_{\smT}(Q)}^2 &\leq C_{\lambda} \left(\dbint_{Q}\taubm{r}^2|L_{\sdmT} q|^2 + X_{g}^{\sdt} + W_g^{\sdt} \right) \\
&\quad + C_{\lambda}\left[\left(\frac{\dt \tau}{\delta^3T^4}\right)^2+\left(\frac{\dt \tau^2}{\delta^4 T^6}\right)^2\right]\left(\int_{\Omega}|z^{M+\frac{1}{2}}|^2\right),
\end{align*}
with 
\begin{equation*}
X_{g}^{\sdt}=\dbint_{Q}\taubm{\mu_g} \taubm{z}^2
\end{equation*}
and
\begin{equation*}
W_g^{\sdt}=\dbint_{Q} \taubm{\gamma_g}(\Dtbar z)^2,
\end{equation*}
where $\mu_{g}= s^2 + \left[\left(\frac{\dt \tau}{\delta^3T^4}\right)^2+\left(\frac{\dt \tau^2}{\delta^4 T^6}\right)^2\right]$ and $\gamma_g=T^2(\tau\dt)^2\theta^4+\frac{\tau^2(\dt)^4}{\delta^6 T^8}$.
\end{lem}

The proof of this result follows from successive applications of triangle and Young inequalities and can be carried out exactly as in \cite[Proof of Lemma 2.2]{BHS20}, thus we omit the proof. 

Using the estimates obtained in Lemmas \ref{lem:est_I11} through \ref{lem:est_rhs} in identity \eqref{eq:iden_carleman}, we have that for $0<\dt \tau^2 (\delta ^4T^6)^{-1}\leq \epsilon_1(\lambda)$ and $0<\tau h(\delta T^2)^{-1}\leq \epsilon_1(\lambda)$ there exists a positive constant $C_{\lambda}$ uniform with respect to $\tau$ and $\dt$ such that 
\begin{align}\notag 
&\norme{Az}_{L_{\smT}^2(Q)}^2+\norme{Bz}_{L_{\smT}^2(Q)}^2+  2\tau\lambda^2 \dbint_{Q} \taubm{\theta}\phi |\partial_x\psi|^2 |\difh\taubm{z}|^2 \\ \notag
&\qquad  + 2 \tau^3\lambda^4 \dbint_{Q} \taubm{\theta}^3\phi^3 |\partial_x\psi|^4 \taubm{z}^2 + 2Y \\ \notag
&\quad \leq C_{\lambda}\left(\dbint_{Q}\taubm{r}^2 |L_{\sdmT}q|^2+\int_{\Omega}(s^{M+\frac{1}{2}})^2 (z^{M+\frac{1}{2}})^2+\int_{\Omega}(s^{\frac{1}{2}})^2 (z^{\frac{1}{2}})^2+\int_{\Omega}|(\difh z)^{M+\frac{1}{2}}|^2\right) \\ \label{eq:car_psi}
&\qquad + C_{\lambda} \left(\int_{\Omega} (s^{M+\frac{1}{2}}h)^2|(\difh z)^{M+\frac{1}{2}}|^2+\int_{\Omega} (s^{\frac{1}{2}}h)^2|(\difh z)^{\frac{1}{2}}|^2\right)+2X+2W,
\end{align}
where $Y:=Y_{11}^{h}+Y_{21}^{h}$ and 
\begin{align} \label{eq:def_W}
W&:=W_{33}^{\sdt}+W_{+}^{{\scriptscriptstyle h},\sdt}+W_{g}^{\sdt}, \\
X&:=\sum_{i,j=1}^{2}X_{ij}^h+ X_{31}^{{\scriptscriptstyle h},\sdt}+X_{32}^{\sdt}+X_{33}^{\sdt}+X_{+}^{{\scriptscriptstyle h},\sdt}+X_{g}^{\sdt}.
\end{align}

As in other discrete Carleman works, the term $Y$ can be dropped. In fact, we have
\begin{lem}
For all $\lambda>0$, there exists $0<\epsilon_2(\lambda)<\epsilon_1(\lambda)$ such that for $0<\tau h (\delta T^2)^{-1}\leq \epsilon_{2}(\lambda)$, we have $Y\geq 0$.
\end{lem}
\begin{proof}
The proof of this result is straightforward. Shifting the time integral, we notice that $Y$ can be written as
\begin{align*}
Y&=\dint_{0}^{T}(r\ov{\difh \rho})_{N+1}|\difh z|_{N+\frac{1}{2}}^2-\dint_{0}^{T}(r\ov{\difh \rho})_{0}|\difh z|_{\frac{1}{2}}^2\\
&\quad +\mathcal O_{\lambda}(1)\dint_{0}^{T} (sh)^2\left[(r\ov{\difh \rho})_{N+1}|\difh z|_{N+\frac{1}{2}}^2+ (r\ov{\difh \rho})_{0}|\difh z|_{\frac{1}{2}}^2 \right].
\end{align*}
Thanks to \eqref{eq:deriv_weight_boundary} the first two terms of the above equation are nonnegative. Then, using that $(sh)\leq \tau h(\delta T)^{-1}$ the result follows by taking $\epsilon_{2}(\lambda)$ small enough. 
\end{proof}
Using the above result and recalling that $|\partial_x\psi|\geq C>0$ in $\Omega\setminus \mathcal B_0$, we obtain from \eqref{eq:car_psi}
\begin{align}\notag 
&\norme{Az}_{L_{\smT}^2(Q)}^2+\norme{Bz}_{L_{\smT}^2(Q)}^2+  2\tau\lambda^2 \dbint_{Q} \taubm{\theta}\phi |\difh\taubm{z}|^2 + 2 \tau^3\lambda^4 \dbint_{Q} \taubm{\theta}^3\phi^3 \taubm{z}^2  \\ \notag
& \leq C_{\lambda}\left(\dbint_{Q}\taubm{r}^2 |L_{\sdmT}q|^2 + \tau\lambda^2 \dbint_{\mathcal B_0\times(0,T)} \taubm{\theta}\phi |\difh\taubm{z}|^2 + \tau^3\lambda^4 \dbint_{\mathcal B_0\times(0,T)} \taubm{\theta}^3\phi^3 \taubm{z}^2\right) \\ \notag
&\quad +C_{\lambda}\left(\int_{\Omega}(s^{M+\frac{1}{2}})^2 (z^{M+\frac{1}{2}})^2+\int_{\Omega}(s^{\frac{1}{2}})^2 (z^{\frac{1}{2}})^2+\int_{\Omega}|(\difh z)^{M+\frac{1}{2}}|^2\right) \\ \label{eq:car_B0_init}
&\quad + C_{\lambda} \left(\int_{\Omega} (s^{M+\frac{1}{2}}h)^2|(\difh z)^{M+\frac{1}{2}}|^2+\int_{\Omega} (s^{\frac{1}{2}}h)^2|(\difh z)^{\frac{1}{2}}|^2\right)+2X+2W.
\end{align}

We will use now the third term in the left-hand side of the above expression to generate a positive term containing $|\ov{\difh\taubm{z}}|^2$ and some other terms. The precise result is as follows. 
\begin{lem}\label{lem:gradient_adjoint}
 Let $h_1=h_1(\lambda)$ be sufficiently small. Then, for $0<h\leq h_{1}(\lambda)$ we have
\begin{align} \label{est:grad_primal} 
\tau \lambda^2\dbint_{Q} \taubm{\theta}\phi |\difh\taubm{z}|^2 \geq \tau \lambda^2\dbint_{Q} \taubm{\theta}\phi |\ov{\difh\taubm{z}}|^2+ H - \tilde{X}-J,
\end{align}
with 
\begin{align*}
H&= \frac{h^2  \tau \lambda^2}{4} \dbint_{Q} \taubm{\theta}\phi |\delh\taubm{z}|^2, \\
\tilde{X}&=h^2\dbint_{Q}\taubm{s}\mathcal O_{\lambda}(1)|\difh\taubm z|^2+\dbint_{Q}\taubm{sh}\mathcal O_{\lambda}(1)|\ov{\difh\taubm z}|^2,
\end{align*}
and
\begin{align*}
J=h^4\dbint_{Q}\taubm{s}\mathcal O_{\lambda}(1)|\delh\taubm{z}|^2.
\end{align*}
\end{lem}

The proof follows the steps of \cite[Lemma 3.11]{BLR14}, but with some simplifications due to the 1-D nature of our problem. For completeness, we present a brief proof in \Cref{sec:aux_lemmas}

Once we reach this point, let us choose $\lambda_1\geq 1$ sufficiently large and let us fix $\lambda=\lambda_1$ for the rest of the proof. Notice that by doing this, some of the lower order terms in the remainders $X_{ij}$, $i,j=1,2$, can be absorbed by its counterparts in the left-hand side of \eqref{eq:car_B0_init}. Indeed, the terms of order $\mathcal O(1)$ can be absorbed as soon as $\lambda_1$ is large enough. 

Let us take $\epsilon_3(\lambda)=\min\{\epsilon_1(\lambda_1),\epsilon_2(\lambda_1)\}$ and $0<h\leq h_1(\lambda_1)$, then from \Cref{lem:gradient_adjoint}, \eqref{eq:car_B0_init} and the discussion above, we obtain for all $0<\dt \tau^2 (\delta ^4T^6)^{-1}\leq \epsilon_3(\lambda)$ and $0<\tau h(\delta T^2)^{-1}\leq \epsilon_3(\lambda)$ that
\begin{align}\notag 
&\norme{Az}_{L_{\smT}^2(Q)}^2+\norme{Bz}_{L_{\smT}^2(Q)}^2+ \tau \dbint_{Q} \taubm{\theta} |\ov{\difh\taubm{z}}|^2 \\ \notag
&\quad +\tau \dbint_{Q} \taubm{\theta} |\difh\taubm{z}|^2 + \tau^3 \dbint_{Q} \taubm{\theta}^3 \taubm{z}^2 + \underline{H}  \\ \notag
& \leq C_{\lambda_1}\left(\dbint_{Q}\taubm{r}^2 |L_{\sdmT}q|^2 + \tau \dbint_{\mathcal B_0\times(0,T)} \taubm{\theta} |\difh\taubm{z}|^2 + \tau^3 \dbint_{{\mathcal B_0}\times(0,T)} \taubm{\theta}^3 \taubm{z}^2\right) \\ \notag
&\quad +C_{\lambda_1}\left(\int_{\Omega}(s^{M+\frac{1}{2}})^2 (z^{M+\frac{1}{2}})^2+\int_{\Omega}(s^{\frac{1}{2}})^2 (z^{\frac{1}{2}})^2+\int_{\Omega}|(\difh z)^{M+\frac{1}{2}}|^2\right) \\ \label{eq:car_B0_st2}
&\quad + C_{\lambda_1} \left(\int_{\Omega} (s^{M+\frac{1}{2}}h)^2|(\difh z)^{M+\frac{1}{2}}|^2+\int_{\Omega} (s^{\frac{1}{2}}h)^2|(\difh z)^{\frac{1}{2}}|^2\right)+\underline{X}+W+J,
\end{align}
where 
\begin{align}\notag
\underline{H}&:= h^2\dbint_{Q}\taubm{s}|\delh\taubm{z}|^2, \quad \underline{X}:= \underline{X_1}+\underline{X_2},
\end{align}
and 
\begin{align*}
\underline{X_1}&=\dbint_{Q}\taubm{\mu_1}\taubm{z}^2+\dbint_{Q}\taubm{{\nu_{1,b}}}|\ov{\difh \taubm{z}}|^2+\dbint_{Q}\taubm{\nu_1}|\difh\taubm z|^2, \\
\underline{X_2}&= \left\{\left(\frac{\tau \dt}{\delta^4T^5}\right) + \left(\frac{\tau^2 \dt}{\delta^4 T^6}\right) + \left(\frac{\dt \tau}{\delta^3T^4}\right)^2+\left(\frac{\dt \tau^2}{\delta^4 T^6}\right)^2 + \left(\frac{\tau \dt}{\delta ^3 T^4}\right)\left(\frac{\tau h}{\delta^{} T^2}\right)^3\right\} \dbint_{Q}\taubm{z} \\
&\quad + \left\{\left(\frac{\tau^2 \dt}{\delta^4 T^6}\right)+ \left(\frac{\tau \dt}{\delta ^3 T^4}\right)\left(\frac{\tau h}{\delta^{} T^2}\right)+\left(\frac{\tau \dt}{\delta ^3 T^4}\right)\left(\frac{\tau h}{\delta^{} T^2}\right)^3 \right\}\dbint_{Q}|\difh\taubm{z}|^2,
\end{align*}
with $\mu_1:=\left[s^2+Ts^2\theta+T^2s\theta^2+s^3(sh)^2\right]\mathcal O_{\lambda_1}(1)$, $\nu_{1,b}:=s(sh)\mathcal O_{\lambda_1}(1)$, $\nu_{1}:=T(sh)^2\theta+s(sh)^2$, and where we recall that $W$ is defined in \eqref{eq:def_W}.

\begin{rmk}\label{rmk:sep_X1_X2}
Here, we have separated the terms in $\underline{X_1}$ and $\underline{X_2}$ based on the following criteria. Notice that in the definition of $\mu_1$, the first three terms do not depend on $\dt$ or $h$. Hence, using the parameter $\tau$ we can absorb them in the left-hand side of \eqref{eq:car_B0_st2} as in a classical Carleman estimate. On the other hand, notice that all of the other terms in $\underline{X_1}$ have a good power of $s$ as compared to the corresponding ones in left-hand side of \eqref{eq:car_B0_st2} but they are multiplied by a factor of $(sh)$. By taking  $(sh)$ small enough, we will absorb them into the left-hand side. Finally, notice that all the terms in $\underline{X_2}$ have some power of $\dt$. In a subsequent step, we will see that we can obtain a general condition for taking $\dt$ small enough and control them by the similar terms in the left-hand side. Similar ideas will be used to absorb the term $W$.
\end{rmk}

Let us clean up a little bit more inequality \eqref{eq:car_B0_st2} by imposing some conditions on the parameter $h$ and the product $(sh)$. First, notice that the new term $\underline{H}$ can control the remainder term $J$ by considering some $0<h_0\leq h_1(\lambda_1)$ small enough. Indeed, for $0<h\leq h_0$, we can drop both $J$ and $\underline H$ in \eqref{eq:car_B0_st2}. 

As anticipated in \Cref{rmk:sep_X1_X2}, to absorb the term $\underline{X_1}$, let us choose some $0<\epsilon_4\leq \epsilon_3(\lambda_1)$ and some $\tau_1\geq 1$ sufficiently large. Thus, for $\tau\geq \tau_1(T+T^2)$ and
\begin{equation}\label{eq:cond_inter}
\frac{\tau h}{\delta T^2}\leq \epsilon_4 \quad\text{and}\quad \frac{\dt \tau^2}{\delta ^4T^6}\leq \epsilon_3,
\end{equation}
we have
\begin{align*}\notag 
&\norme{Az}_{L_{\smT}^2(Q)}^2+\norme{Bz}_{L_{\smT}^2(Q)}^2+ \tau \dbint_{Q} \taubm{\theta} |\ov{\difh\taubm{z}}|^2 \\
&\quad +\tau \dbint_{Q} \taubm{\theta} |\difh\taubm{z}|^2 + \tau^3 \dbint_{Q} \taubm{\theta}^3 \taubm{z}^2   \\ \notag
& \leq C_{\lambda_1}\left(\dbint_{Q}\taubm{r}^2 |L_{\sdmT}q|^2 + \tau \dbint_{{\mathcal B_0}\times(0,T)} \taubm{\theta} |\difh\taubm{z}|^2 + \tau^3 \dbint_{{\mathcal B_0}\times(0,T)} \taubm{\theta}^3 \taubm{z}^2\right) \\ \notag
&\quad +C_{\lambda_1}\left(\int_{\Omega}(s^{M+\frac{1}{2}})^2 (z^{M+\frac{1}{2}})^2+\int_{\Omega}(s^{\frac{1}{2}})^2 (z^{\frac{1}{2}})^2+\int_{\Omega}|(\difh z)^{M+\frac{1}{2}}|^2\right) \\
&\quad + C_{\lambda_1} \left(\int_{\Omega} (s^{M+\frac{1}{2}}h)^2|(\difh z)^{M+\frac{1}{2}}|^2+\int_{\Omega} (s^{\frac{1}{2}}h)^2|(\difh z)^{\frac{1}{2}}|^2\right)+\underline{X_2}+W.
\end{align*}

To conclude this step, notice that the first term of $W_{+}^{h,\sdt}$ (see Eq. \eqref{eq:def_W+}) already have the good power $s^{-1}$ and factor $(sh)^2$. Hence, we can absorb this term using \eqref{eq:cond_inter} and obtain
\begin{align}\notag 
&\norme{Az}_{L_{\smT}^2(Q)}^2+\norme{Bz}_{L_{\smT}^2(Q)}^2+ \tau \dbint_{Q} \taubm{\theta} |\ov{\difh\taubm{z}}|^2 \\ \notag
&\quad +\tau \dbint_{Q} \taubm{\theta} |\difh\taubm{z}|^2 + \tau^3 \dbint_{Q} \taubm{\theta}^3 \taubm{z}^2   \\ \notag
& \leq C_{\lambda_1}\left(\dbint_{Q}\taubm{r}^2 |L_{\sdmT}q|^2 + \tau \dbint_{{\mathcal B_0}\times(0,T)} \taubm{\theta} |\difh\taubm{z}|^2 + \tau^3 \dbint_{{\mathcal B_0}\times(0,T)} \taubm{\theta}^3 \taubm{z}^2\right) \\ \notag
&\quad +C_{\lambda_1}\left(\int_{\Omega}(s^{M+\frac{1}{2}})^2 (z^{M+\frac{1}{2}})^2+\int_{\Omega}(s^{\frac{1}{2}})^2 (z^{\frac{1}{2}})^2+\int_{\Omega}|(\difh z)^{M+\frac{1}{2}}|^2\right) \\ \label{eq:car_B0_st3}
&\quad + C_{\lambda_1} \left(\int_{\Omega} (s^{M+\frac{1}{2}}h)^2|(\difh z)^{M+\frac{1}{2}}|^2+\int_{\Omega} (s^{\frac{1}{2}}h)^2|(\difh z)^{\frac{1}{2}}|^2\right)+\underline{X_2}+\underline{W},
\end{align}
where $\underline{W}$ stands for
\begin{equation}\label{def:under_W}
\underline{W}:= \dbint_{Q} \taubm{\gamma_1}(\Dtbar z)^2,
\end{equation}
with $\gamma_1:=\dt\left(\tau T {\theta}^2 + \frac{\tau \dt}{\delta^3 T^4}\right)+\dt s^2 +\left(T^2(\tau\dt)^2\theta^4+\frac{\tau^2(\dt)^4}{\delta^6 T^8}\right)$.

\subsection{Adding a term of $\Dtbar$ and $\delh$ in the left-hand side and absorbing the remaining terms}
Using the equation verified by $Az$ (see \cref{eq:def_Az}) and since $r\ov{\av{\rho}}=1+(sh)^2\mathcal O_{\lambda}(1)$, we have
\begin{align*}
\delh\taubm{z}&= Az+\mathcal O_{\lambda}(1)(sh)^2 \delh\taubm{z}+\mathcal O_{\lambda}(1)\taubm{s}^2\left(\taubm{z}+\frac{h^2}{4}\delh\taubm{z}\right) \\
&\quad +\mathcal O_{\lambda}(1)\tau T \taubm{\theta}^2\taubm{z},
\end{align*}
where we have also used that $r\delh \rho=s^2\mathcal O_{\lambda}(1)$, \Cref{lem:double_average} and the estimate $|\theta^\prime|\leq CT\theta^2$ for all $t\in[0,T]$. Multiplying the above equation by $\taubm{s}^{-1/2}$ and taking the $L^2_{\smT}(Q)$-norm in both sides yield
\begin{align*}
\dbint_{Q}\taubm{s}^{-1}|\delh\taubm{z}|^2 &\leq C_{\lambda_1}\left(\dbint_{Q}\taubm{s}^{-1}|Az|^2+\dbint_{Q}\tbm\left(s^{-1}[sh]^4\right)|\delh\taubm{z}|^2\right) \\
&\quad + C_{\lambda_1}\left(\dbint_{Q} \taubm{s}^{3}\taubm{z}^2+\dbint_{Q}\tau T^2\taubm{\theta}^3\taubm{z}^2\right).
\end{align*}
Notice the the term containing $\delh$ in the right-hand side of the above inequality has the good power $s^{-1}$ and the factor $(sh)^4$. Thus, by recalling condition \eqref{eq:cond_inter} we can absorb it into the right-hand side. Furthermore, increasing if necessary the value of $\tau_1$ such that $\tau_1\geq 1$ and $s(t)\geq 1$ for any $t$, we get
\begin{align}\label{est:delh_add}
\dbint_{Q}\taubm{s}^{-1}|\delh\taubm{z}|^2 &\leq C_{\lambda_1}\left(\norme{Az}_{L^2_{\smT}(Q)}^2+\dbint_{Q} \taubm{s}^{3}\taubm{z}^2\right).
\end{align}

In a similar way, using the equation verified by $Bz$ (see \cref{eq:def_Bz}), it is not difficult to see that 
\begin{equation*}
\dbint_{Q}\taubm{s^{-1}}(\Dtbar z)^2 \leq C_{\lambda_1}\left(\dbint_{Q}\taubm{s^{-1}} |Bz|^2+\dbint_{Q}\taubm{s}|\ov{\difh\taubm{z}}|^2+\dbint_{Q}\taubm{s}\taubm{z}^2\right),
\end{equation*}
where we have used that $r\ov{\difh\rho}=s\mathcal O_{\lambda}(1)$ (see Proposition \ref{eq:r_mh_dr_rho}) and $\partial_{xx}\phi=\mathcal O_{\lambda}(1)$. Using again that $s(t)\geq 1$ for any $t$, we have
\begin{equation}\label{est:Dtbar_add}
\dbint_{Q}\taubm{s^{-1}}(\Dtbar z)^2 \leq C_{\lambda_1}\left(\norme{Bz}_{L^2_{\smT}(Q)}^2+\dbint_{Q}\taubm{s}|\ov{\difh\taubm{z}}|^2+\dbint_{Q}\taubm{s}^3\taubm{z}^2\right).
\end{equation}

By combining estimates \eqref{eq:car_B0_st3}, \eqref{est:delh_add}, and \eqref{est:Dtbar_add}, we get
\begin{align}\notag 
&\dbint_{Q}\taubm{s}^{-1}\left[(\Dtbar z)^2+|\delh \taubm{z}|^2\right] + \dbint_{Q} \taubm{s} |\ov{\difh\taubm{z}}|^2 \\ \notag
&\quad + \dbint_{Q} \taubm{s} |\difh\taubm{z}|^2 + \dbint_{Q} \taubm{s}^3 \taubm{z}^2   \\ \notag
& \leq C_{\lambda_1}\left(\dbint_{Q}\taubm{r}^2 |L_{\sdmT}q|^2 +  \dbint_{{\mathcal B_0}\times(0,T)} \taubm{s} |\difh\taubm{z}|^2 +  \dbint_{{\mathcal B_0}\times(0,T)} \taubm{s}^3 \taubm{z}^2\right) \\ \notag
&\quad +C_{\lambda_1}\left(\int_{\Omega}(s^{M+\frac{1}{2}})^2 (z^{M+\frac{1}{2}})^2+\int_{\Omega}(s^{\frac{1}{2}})^2 (z^{\frac{1}{2}})^2+\int_{\Omega}|(\difh z)^{M+\frac{1}{2}}|^2\right) \\ \label{eq:car_B0_st4}
&\quad + C_{\lambda_1} \left(\int_{\Omega} (s^{M+\frac{1}{2}}h)^2|(\difh z)^{M+\frac{1}{2}}|^2+\int_{\Omega} (s^{\frac{1}{2}}h)^2|(\difh z)^{\frac{1}{2}}|^2\right)+\underline{X_2}+\underline{W}.
\end{align}

With this new inequality at hand, the next result gives us conditions on the parameter $\dt$ such that the terms $\underline{X_2}$ and $\underline{W}$ can be absorbed into the left-hand side. The result is as follows.
\begin{lem}\label{lem:dt_absorb}
For any $\tau\geq 1$, there exists $\epsilon_5=\epsilon_5(\lambda_1)$ such that for
\begin{equation*}
0<\frac{\tau^4\dt}{\delta^4 T^6}\leq \epsilon_5
\end{equation*}
the following estimate holds
\begin{equation*}
\underline{X_2}+\underline{W}\leq \epsilon_5\left(\tau^3\dbint_{Q} \taubm{\theta}^3\taubm{z}^2+\tau^{-1}\dbint_{Q}\taubm{\theta}^{-1}(\Dtbar z)^2\right).
\end{equation*}
\end{lem}
The proof of this result is similar to that in \cite[Lemma 2.3]{BHS20}. For completeness, we sketch it briefly in \Cref{sec:aux_lemmas}.

Using \Cref{lem:dt_absorb}, we take $\epsilon_5=1/2C_{\lambda_1}$, where $C_{\lambda_1}>0$ is the constant appearing in \eqref{eq:car_B0_st4} and set $\epsilon_6(\lambda)=\min\{\epsilon_3(\lambda_1),\epsilon_4(\lambda_1),\epsilon_5(\lambda_1)\}$. Whence, for $\tau\geq \tau_1(T+T^2)$, $h\leq h_0$, and 
\begin{equation}\label{eq:cond_inter_final}
\frac{\tau h}{\delta T^2}\leq \epsilon_6 \quad\text{and}\quad \frac{\tau^4\dt}{\delta ^4T^6}\leq \epsilon_6,
\end{equation}
the following estimate holds
\begin{align}\notag 
&\dbint_{Q}\taubm{s}^{-1}\left[(\Dtbar z)^2+|\delh \taubm{z}|^2\right] + \dbint_{Q} \taubm{s} |\ov{\difh\taubm{z}}|^2 \\ \notag
&\quad + \dbint_{Q} \taubm{s} |\difh\taubm{z}|^2 + \dbint_{Q} \taubm{s}^3 \taubm{z}^2   \\ \label{eq:car_B0_st5}
& \leq C_{\lambda_1}\left(\dbint_{Q}\taubm{r}^2 |L_{\sdmT}q|^2 + \tau \dbint_{{\mathcal B_0}\times(0,T)} \taubm{\theta} |\difh\taubm{z}|^2 + \tau^3 \dbint_{{\mathcal B_0}\times(0,T)} \taubm{\theta}^3 \taubm{z}^2+BT\right),
\end{align}
where we have defined
\begin{align}\notag
BT&:=\int_{\Omega}(s^{M+\frac{1}{2}})^2 (z^{M+\frac{1}{2}})^2+\int_{\Omega}(s^{\frac{1}{2}})^2 (z^{\frac{1}{2}})^2+\int_{\Omega}|(\difh z)^{M+\frac{1}{2}}|^2 \\ \label{eq:BT_def}
&\quad + \int_{\Omega} (s^{M+\frac{1}{2}}h)^2|(\difh z)^{M+\frac{1}{2}}|^2+\int_{\Omega} (s^{\frac{1}{2}}h)^2|(\difh z)^{\frac{1}{2}}|^2.
\end{align}

As in other discrete Carleman works, the terms collected in $BT$ in the above equation cannot be absorbed, but only estimated. The result reads as follows.

\begin{lem}\label{lem:BT_terms}
Assume that \eqref{eq:cond_inter_final} holds. Then, there exists some $C>0$ uniform with respect to $h$ and $\dt$ such that
\begin{equation}\label{eq:bound_BT}
BT\leq C(1+\epsilon_6^2) h^{-2}\left(\int_{\Omega}(z^{\frac{1}{2}})^2+\int_{\Omega}(z^{M+\frac{1}{2}})^2\right).
\end{equation}
\end{lem}
\begin{proof} Under the hypothesis of the lemma and recalling that $\delta\leq 1/2$, we deduce that $\dt\leq \delta T/2$, therefore 
\begin{equation}\label{eq:est_theta_max}
\max_{t\in [0,T+\dt]}\theta(t)\leq 2/(\delta T^2).
\end{equation}
With this estimate, we readily see that the first term in $BT$ can be bounded as
\begin{align*}
\int_{\Omega}(s^{M+\frac{1}{2}})^2(z^{M+\frac{1}{2}})^2&=\int_{\Omega} \left(\tau \theta^{M+\frac{1}{2}}\right)^2 (z^{M+\frac{1}{2}})^2 \\
&\leq 4h^{-2}\int_{\Omega} \left(\frac{\tau h}{\delta T^2}\right)^2 (z^{M+\frac{1}{2}})^2 \leq 4\epsilon_6^2 h^{-2}\int_{\Omega} (z^{M+\frac{1}{2}})^2,
\end{align*}
where we have used the first condition in \eqref{eq:cond_inter_final}. The same is true for the second term in $BT$. For the third term, we notice that $(\difh z)^2\leq Ch^{-2}\left\{\shpx{z}^2+\shmx{z}^2\right\}$. Thus,
\begin{equation*}
\int_{\Omega}|(\partial_h z)^{M+\frac{1}{2}}|^2 \leq C h^{-2}\int_{\Omega}(z^{M+\frac{1}{2}})^2.
\end{equation*}

For the fourth term, we argue as in the previous cases to deduce
\begin{equation*}
\int_{\Omega}(s^{M+\frac{1}{2}}h)^2|(\difh z)^{M+\frac{1}{2}}|^2 \leq Ch^{-2}\left(\frac{\tau h}{\delta T^2}\right)^2 \int_{\Omega} (z^{M+\frac{1}{2}})^2  \leq C\epsilon_6^2 h^{-2}\int_{\Omega} (z^{M+\frac{1}{2}})^2.
\end{equation*}
Note that the same is true for the last term in $BT$. Collecting the above estimates gives the desired result. 
\end{proof}


Combining estimates \eqref{eq:car_B0_st5} and \eqref{eq:bound_BT} gives
\begin{align}\notag 
&\dbint_{Q}\taubm{s}^{-1}\left[(\Dtbar z)^2+|\delh \taubm{z}|^2\right] + \dbint_{Q} \taubm{s} |\ov{\difh\taubm{z}}|^2 \\ \notag
&\quad + \dbint_{Q} \taubm{s} |\difh\taubm{z}|^2 + \dbint_{Q} \taubm{s}^3 \taubm{z}^2   \\ \notag
& \leq C_{\lambda_1}\left(\dbint_{Q}\taubm{r}^2 |L_{\sdmT}q|^2 + \dbint_{{\mathcal B_0}\times(0,T)} \taubm{s} |\difh\taubm{z}|^2 + \dbint_{{\mathcal B_0}\times(0,T)} \taubm{s}^3 \taubm{z}^2 \right) \\ \label{eq:est_car_B0_stf}
&\quad + C_{\lambda_1} h^{-2}\left(\int_{\Omega}(z^{\frac{1}{2}})^2+\int_{\Omega}(z^{M+\frac{1}{2}})^2\right)
\end{align}
for all $\tau\geq \tau_1(T+T^2)$, $h\leq h_0$, and 
\begin{equation}\label{eq:cond_ff}
\frac{\tau h}{\delta T^2}\leq \epsilon_6 \quad\text{and}\quad \frac{\tau^4\dt}{\delta ^4T^6}\leq \epsilon_6.
\end{equation}

\subsection{Returning to the original variable and conclusion}
To conclude our proof, we will remove the local term containing $\difh$ in the right-hand side of \eqref{eq:est_car_B0_stf} and then comeback to the original variable. We argue slightly different to \cite{BLR14} and \cite{BHS20} since the time and space variables are both discrete, but the overall result is the same. 

We present the following.

\begin{lem}\label{lem:local_remove}
For any $\gamma>0$, there exists $C>0$ uniform with respect to $h$ and $\dt$ such that
\begin{align}\notag 
\dbint_{Q_{\mathcal B_0}}\taubm{s}|\difh\taubm{z}|^2 &\leq C\left(1+\frac{1}{\gamma}\right)\dbint_{{\mathcal B}\times(0,T)}\taubm{s}^3 \taubm{z}^2+C\dbint_{Q}|\ov{\difh\taubm{z}}|^2 \\\notag
&\quad + \dbint_{Q} \taubm{s}\taubm{z}^2 +\gamma \dbint_{Q}\taubm{s}^{-1}|\delh\taubm{z}|^2  \\ \label{eq:est_local_grad}
&\quad +\dbint_{Q}\tbm(s^{-1}[sh]^2)|\delh z|^2.
\end{align}
\end{lem}

The proof of this result can be found in \Cref{sec:aux_lemmas}. This result mimics the classical procedure used in the continuous setting and tells that we can remove the local term of $\difh\taubm{z}$ by paying the cost of increasing a little bit the observation set of the local term of $\taubm{z}$. Notice also that some lower order terms appear, nonetheless, all of them can be absorbed in the left-hand side.

With estimate \eqref{eq:est_local_grad} in hand, we can choose $\gamma=\frac{1}{2C_{\lambda_1}}$ with $C_{\lambda_1}$ the constant appearing in \eqref{eq:est_car_B0_stf} and set $\tau_2\geq \tau_1\geq 1$ sufficiently large to obtain
\begin{align}\notag 
&\dbint_{Q}\taubm{s}^{-1}\left[(\Dtbar z)^2+|\delh \taubm{z}|^2\right] + \dbint_{Q} \taubm{s} |\ov{\difh\taubm{z}}|^2 \\ \notag
&\quad + \dbint_{Q} \taubm{s} |\difh\taubm{z}|^2 + \dbint_{Q} \taubm{s}^3 \taubm{z}^2    \\ \notag
& \leq C_{\lambda_1}\left(\dbint_{Q}\taubm{r}^2 |L_{\sdmT}q|^2 + \dbint_{{\mathcal B}\times(0,T)} \taubm{s}^3 \taubm{z}^2 \right)  \\ \label{eq:est_car_B0_sin_grad}
&\quad + C_{\lambda_1} h^{-2}\left(\int_{\Omega}(z^{\frac{1}{2}})^2+\int_{\Omega}(z^{M+\frac{1}{2}})^2\right)
\end{align}
for all $\tau\geq \tau_2(T+T^2)$ and verifying \eqref{eq:cond_ff}. In this step, we can decrease if necessary the value of $\epsilon_6$ in \eqref{eq:cond_ff} to absorb the last term in \eqref{eq:est_local_grad} since we recall that $(sh)\leq \tau h(\delta T^2)^{-1}$. For convinience, we still call it $\epsilon_6$.

Now, we will come back to the original variable. We recall that we have set the change of variables $z=rq$. To give appropriate bounds on the gradient terms in \eqref{eq:est_car_B0_sin_grad}, we proceed as follows. A direct computation yields 
\begin{equation}\label{eq:iden_grad_change}
r\difh q=r\difh(\rho z)=r\av{\rho}\difh z+\av{z}r \difh\rho.
\end{equation}
Multiplying by $s^{1/2}$ and taking the $L^2_{\sdmT}(Q)$-norm (since $z$ and $q$ are naturally defined on the dual time-grid) gives
\begin{equation*}
\ddbint_{Q}r^2 s |\difh q|^2 \leq C_{\lambda_1}\left(\ddbint_{Q}s |\difh z|^2 +\ddbint_{Q} s^1 |\tilde z|^2(r\difh \rho)^2 \right),
\end{equation*}
where we have used that $r\av{\rho}=\mathcal O_{\lambda}(1)$. From Proposition \ref{eq:r_mh_dr_rho}, we have $r\difh \rho=s\mathcal O_{\lambda}(1)$, hence
\begin{align*}
\ddbint_{Q}r^2 s |\difh q|^2 &\leq C_{\lambda_1}\left(\ddbint_{Q}s |\difh z|^2 +\ddbint_{Q} s^3|\tilde z|^2 \right)  \\
&\leq C_{\lambda_1}\left(\ddbint_{Q}s |\difh z|^2 +\ddbint_{Q} s^3\avl{|z|^2} \right) \\
&=C_{\lambda_1}\left(\ddbint_{Q}s |\difh z|^2 +\ddbint_{Q} s^3|z|^2 \right),
\end{align*}
where we have used convexity in the second line and formula \eqref{eq:shift_av_space} together with zero boundary conditions  in the third line. Shifting the time integral (see \cref{trans_doub}) yields
\begin{equation}\label{eq:est_grad_q}
\dbint_{Q} \tbm(r^2 s)|\difh\taubm{q}|^2\leq  C_{\lambda_1}\left(\dbint_{Q}\taubm{s}|\difh\taubm{z}|^2+\dbint_{Q}\taubm{s}^3\taubm{z}^2\right).
\end{equation}

The same ideas can be used to obtain a term containing $\ov{\difh q}$. Indeed, with \Cref{lem:average_product,lem:double_average}, a straightforward  computation gives
\begin{align*}
\ov{\difh q}&=\ov{\difh(\rho z)}=\ov{\av{\rho}\difh z}+\ov{\av{z}\difh \rho} \\
&= \ov{\av{\rho}}\,\ov{\difh z}+\ov{\av{z}}\,\ov{\difh \rho}+\frac{h^2}{4}\left(\delh z\difhb(\av\rho)+\difhb(\av{z})\delh \rho \right) \\
&= \ov{\av{\rho}}\,\ov{\difh z}+z \ov{\difh \rho}+\frac{h^2}{4}\left(2 \delh z \ov{\difh\rho}+\ov{\difh{z}}\delh \rho \right),
\end{align*}
where we have used the useful identity $\difhb{\av{p}}=\ov{\difh p}$ for $p\in \R^{\smesh}$. Arguing as above, we multiply by $r s^{1/2}$ the above identity and take the $L^2_{\sdmT}(Q)$-norm. \Cref{prop:low_order_derivs} provides the useful estimates $r\ov{\av\rho}=\mathcal O_{\lambda_1}(1)$, $r\ov{\difh \rho}=s\mathcal O_{\lambda_1}(1)$ and $r\delh\rho=s^2\mathcal O_{\lambda_1}(1)$. Thus
\begin{align*}
\ddbint_{Q}r^2 s|\ov{\difh q}|^2 &\leq C_{\lambda_1}\left(\ddbint_{Q}s^3|z|^2+\ddbint_{Q}s |\ov{\difh z}|^2\right) \\
&\quad + C_{\lambda_1}\left(\ddbint_{Q}s^{-1}(sh)^4|\delh z|^2+\ddbint_{Q}s(sh^4)|\ov{\difh z}|^2\right).
\end{align*}
Note that the last two terms have the corresponding good power of $s$ and a small factor $(sh)$. By shifting the time integral (see \cref{trans_doub}) and recalling condition \eqref{eq:cond_ff}, we get
\begin{align}\notag
\dbint_{Q}\tbm(r^2 s) |\ov{\difh\taubm{q}}|^2 &\leq C_{\lambda_1}\left(\dbint_{Q}\taubm{s}^3\taubm{z}^2+\ddbint_{Q} \taubm{s} |\ov{\difh\taubm{z}}|^2\right) \\ \label{eq:est_dual_grad_q}
&\quad + C_{\lambda_1}\epsilon_6^4\left(\ddbint_{Q}\taubm{s}^{-1}|\delh \taubm{z}|^2+\ddbint_{Q}\taubm{s}|\ov{\difh \taubm{z}}|^2\right).
\end{align}

Using similar ideas, we can give an estimate for a term containing  $\delh\taubm q$. To do so, we recall identity \eqref{eq:iden_lap} and see that
\begin{equation*}
rs^{-1/2}\delh q=s^{-1/2}r\delh \rho \left(z+\frac{h^2}{4}\delh z\right)+s^{-1/2}r\ov{\av{\rho}} \delh z +2 s^{-1/2}r\ov{\difh \rho}\,\ov{\difh z}.
\end{equation*}
Arguing as we did above, we readily obtain
\begin{align}\notag 
\dbint_{Q}& \tbm(r^2s^{-1})|\delh\taubm{q}|^2 \\ \notag
&\leq C_{\lambda_1}\left(\dbint_{Q}\taubm{s}^3\taubm{z}^2+\dbint_{Q}\taubm{s}^{-1}|\delh\taubm{z}|^2+\dbint_{Q}\taubm{s}|\ov{\difh\taubm{z}}|^2\right) \\ \label{eq:estimate_deltah_q}
& \quad + C_{\lambda_1}\epsilon_6^4 \dbint_{Q}\taubm{s}^{-1}|\delh \taubm{z}|^2.
\end{align}

Using estimates \eqref{eq:est_grad_q}--\eqref{eq:estimate_deltah_q} in \eqref{eq:est_car_B0_sin_grad} and decreasing (if necessary) the value of $\epsilon_6$, we obtain
\begin{align}\notag 
&\dbint_{Q}\tbm(r^2s^{-1}) |\delh \taubm{q}|^2+\dbint_{Q} \tbm(r^2s) |\ov{\difh\taubm{q}}|^2  \\ \notag
&\quad +\dbint_{Q} \tbm(r^2s) |\difh\taubm{q}|^2 +  \dbint_{Q} \tbm(r^2 s^3) \taubm{q}^2  \\  \notag
& \leq C_{\lambda_1}\left(\dbint_{Q}\taubm{r}^2 |L_{\sdmT}q|^2 + \dbint_{{\mathcal B}\times(0,T)} \tbm(r^2s^3) \taubm{q}^2 \right) \\ \label{eq:est_car_B0_change}
&\quad + C_{\lambda_1} h^{-2}\left(\int_{\Omega}\left|(e^{s\varphi}q)^{\frac{1}{2}}\right|^2+\int_{\Omega}\left|(e^{s\varphi}q)^{M+\frac{1}{2}}\right|^2\right),
\end{align}
where we have used that $z=rq$ to change variables in the terms containing $\taubm z$. Notice also that we have dropped the positive term containing $\Dtbar z$. 

To add a term of $\Dtbar q$, we simply use the equation verified by $q$. Indeed, we have $-\Dtbar q= L_{\sdmT} q+\delh \taubm{q}$, hence
\begin{align}\notag 
\dbint_{Q}\tbm(r^2 s^{-1})|\Dtbar q|^{2} &\leq 2\dbint_{Q} \tbm(r^2 s^{-1})|L_{\sdmT}q|^2+2\dbint_{Q}\tbm(r^2s^{-1})|\delh \taubm{q}|^2 \\ \label{eq:dtbar_q}
&\leq 2\dbint_{Q} \taubm{r}^2 |L_{\sdmT}q|^2+2\dbint_{Q}\tbm(r^2s^{-1})|\delh \taubm{q}|^2,
\end{align}
where we have used that $(s(t))^{-1}\leq 1$ for $\tau_2\geq 1$ large enough. Combining inequalities \eqref{eq:est_car_B0_change} and  \eqref{eq:dtbar_q} we obtain
\begin{align}\notag 
&\dbint_{Q}\tbm(r^2s^{-1})\left[|\Dtbar q|^2+|\delh \taubm{q}|^2\right] + \dbint_{Q} \tbm(r^2s) |\ov{\difh\taubm{q}}|^2  \\ \notag
& + \dbint_{Q} \tbm(r^2s) |\difh\taubm{q}|^2 +  \dbint_{Q} \tbm(r^2 s^3) \taubm{q}^2    \\ \notag
& \leq C_{\lambda_1}\left(\dbint_{Q}\taubm{r}^2 |L_{\sdmT}q|^2 + \dbint_{{\mathcal B}\times(0,T)} \tbm(r^2s^3) \taubm{q}^2 \right) \\ \label{eq:est_car_B0_change_final}
&\quad + C_{\lambda_1} h^{-2}\left(\int_{\Omega}\left|(e^{s\varphi}q)^{\frac{1}{2}}\right|^2+\int_{\Omega}\left|(e^{s\varphi}q)^{M+\frac{1}{2}}\right|^2\right)
\end{align}
for all $\tau\geq \tau_2(T+T^2)$, $h\leq h_0$, and 
\begin{equation*}
\frac{\tau h}{\delta T^2}\leq \epsilon_6 \quad\text{and}\quad \frac{\tau^4\dt}{\delta ^4T^6}\leq \epsilon_6.
\end{equation*}
We conclude the proof by setting $\epsilon_0=\epsilon_6$ and $\tau_0=\tau_2$ and recalling that $r=e^{\tau \theta \varphi}$ and $s=\tau\theta$.

\section{$\phi(h)$-null controllability}\label{sec:control_results}
In this section, we use the Carleman estimate \eqref{eq:car_fully_discrete} to deduce control properties for linear and semilinear fully-discrete parabolic systems. 

\subsection{A fully-discrete observability inequality}
Let us consider the following fully-discrete problem with potential $a\in L^\infty_{\smT}(Q)$
\begin{equation}\label{eq:sys_fully_control_sec}
\begin{cases}
(\Dt y)^{n+\frac{1}{2}}-\delh\taup{y}^{n+\frac{1}{2}}+\taup{ay}^{n+\frac{1}{2}}=\mathbf{1}_{\omega} v^{n+\frac{1}{2}}&n\in\inter{0,M-1}, \\
y^0=g.
\end{cases}
\end{equation}
To achieve a $\phi(h)$-controllability result for \eqref{eq:sys_fully_control_sec}, we begin by proving a relaxed observability estimate for the solutions associated to the adjoint system given by
\begin{equation}\label{eq:adj_heat_sec}
\begin{cases}
-(\Dtbar q)^{n}-\delh\taubm{q}^{n}+a^n\taubm{q}^{n}=0 &n\in\inter{1,M}, \\
q^{M+\frac{1}{2}}=q_T.
\end{cases}
\end{equation}

We have the following result.
\begin{prop}\label{prop:obs_ineq}
For any $\vartheta\in\inter{1,4}$, there exist positive constants $h_0$, $C_0$, $C_1$, $C_2$ and $C_T$, such that for all $T\in(0,1)$, all potentials $a\in L^\infty_{\smT}(Q)$, under the conditions 
\begin{equation}\label{eq:def_h1}
h\leq \min\{h_0,h_1\} \quad\textnormal{with}\quad  h_1=C_0\left(1+\frac{1}{T}+\norme{a}_{L^\infty_{\smT}(Q)}^{2/3}\right)^{-\vartheta}
\end{equation}
and
\begin{equation}\dt\leq \min\{T^{-2} h^{4/\vartheta},(4\norme{a}_{L^\infty_{\smT}(Q)})^{-1}\},
\end{equation}
any solution to \eqref{eq:adj_heat_sec} with $q_T\in \R^{\smesh}$ satisfies
\begin{equation}\label{eq:obs_ineq_fully}
|q^{\frac{1}{2}}|_{L^2(\Omega)}\leq C_{obs}\left(\ddbint_{\omega\times(0,T)}|q|^2+e^{-\frac{C_2}{h^{1/\vartheta}}}|q_T|_{L^2(\Omega)}^2\right)^{1/2},
\end{equation}
where 
\begin{equation}\label{eq:cobs_form}
C_{obs}=e^{C_1(1+\frac{1}{T}+\|a\|_{L^\infty_{\smT}(Q)}^{2/3}+T\norme{a}_{L^\infty_{\smT}(Q)})}.
\end{equation}
\end{prop}

The proof of this result follows as close as possible the continuous case (see e.g. \cite{FCG06}) in a first stage. Nonetheless, in a second part, a careful connection on the different discrete parameters (in this case, $\dt$, $h$, and $\delta$) should be done to obtain the uniform constants involved in \eqref{eq:obs_ineq_fully} (see \Cref{sec:further_h_dt} for further discussion on this). Similar estimates in the discrete setting have been obtained in \cite[Proposition 4.1]{BLR14} for the space-discrete case and \cite[Proposition 3.1]{BHS20} in the time-discrete case.

\begin{proof}[Proof of \Cref{prop:obs_ineq}] For clarity, we have divided the proof in two steps. In what follows, $C$ denotes a positive constant uniform with respect to $h$, $\dt$ and $\delta$ which may change from line to line.

\textbf{Step 1. Cleaning up the Carleman estimate.} Applying \eqref{eq:car_fully_discrete} to the solutions \eqref{eq:adj_heat_sec} with $\mathcal B=\omega$, we readily obtain
\begin{align*}\notag 
\tau^3&\dbint_{Q} \tbm(e^{2\tau\theta\varphi}\theta^3)\taubm{q}^2 \leq C\left(\dbint_{Q}\tbm(e^{2\tau\theta\varphi})|a\taubm{q}|^2+\tau^3\dbint_{{\omega}\times(0,T)}\tbm(e^{2\tau\theta\varphi}\theta^3)\taubm{q}^2\right) \\
&\quad   +C h^{-2}\left(\int_{\Omega}\abs{(e^{\tau\theta\varphi}q)^{\frac{1}{2}}}^2+\int_{\Omega}\abs{(e^{\tau\theta\varphi}q)^{M+\frac{1}{2}}}^2\right) 
\end{align*}
for all $\tau\geq \tau_0(T+T^2)$, $0<h\leq h_0$, $\dt>0$ and $0<\delta\leq 1/2$ satisfying the condition 
\begin{equation}\label{eq:cond_delta}
\frac{\tau^4\dt}{\delta ^4 T^6} \leq \epsilon_0 \quad\textnormal{and}\quad \frac{\tau h}{\delta T^2}\leq \epsilon_0.
\end{equation}
The first term in the right-hand side can be controlled by the term on the left-hand side by choosing $\tau$ large enough. Indeed, using that ${a}\in{L^\infty_{\smT}(Q)}$ and $\theta^{-1}\leq CT^2$ it is standard to see that by choosing 
\begin{equation}\label{eq:tau_big}
\tau\geq CT^2\norme{a}_{L^\infty_{\smT}(Q)}^{2/3} 
\end{equation}
we have
\begin{align}\notag 
\tau^3\dbint_{Q} \tbm(e^{2\tau\theta\varphi}\theta^3)\taubm{q}^2 &\leq C\left(\tau^3\dbint_{Q_{\omega}}\tbm(e^{2\tau\theta\varphi}\theta^3)\taubm{q}^2\right) \\ \label{eq:car_clean}
&\quad   +C h^{-2}\left(\int_{\Omega}\abs{(e^{\tau\theta\varphi}q)^{\frac{1}{2}}}^2+\int_{\Omega}\abs{(e^{\tau\theta\varphi}q)^{M+\frac{1}{2}}}^2\right).
\end{align}
Notice that we can combine \eqref{eq:tau_big} with our initial hypothesis for $\tau$ by choosing some $\tau_1\geq \tau_0$ large enough and setting
\begin{equation*}
\tau \geq \tau_1\left(T+T^2+T^2\norme{a}^{2/3}_{L^\infty_{\smT}(Q)}\right).
\end{equation*}

From \eqref{eq:adj_heat_sec}, we see that $q^{n-\frac{1}{2}}$ solves the equation
\begin{equation}\label{eq:iden_qnmed}
q^{n-\frac{1}{2}}-q^{n+\frac{1}{2}}-\dt \delh q^{n-\frac{1}{2}}+\dt a^n q^{n-\frac{1}{2}}=0, \quad n\in\inter{1,M}.
\end{equation}
From this expression, we can take the $L^2$-inner product on $\R^{\smesh}$ (see \cref{eq:def_inner_L2}) with $q^{n-\frac{1}{2}}$ and use the identity $(a-b)a=\frac{1}{2}a^2-\frac{1}{2}b^2+\frac{1}{2}(a-b)^2$ to deduce
\begin{align}\notag 
\frac{1}{2}&\left(|{q^{n-\frac12}}|_{L^2(\Omega)}^2-|{q^{n+\frac{1}{2}}}|_{L^2(\Omega)}^2\right)+\frac{1}{2}|{q^{n-\frac12}-q^{n+\frac{1}{2}}}|^2_{L^2(\Omega)}+\dt\left(-\delh q^{n-\frac{1}{2}},q^{n-\frac{1}{2}}\right)_{L^2(\Omega)}\\ \notag
&=-\dt \left(a^n q^{n-\frac{1}{2}},q^{n-\frac{1}{2}}\right)_{L^2(\Omega)} \\ \label{eq:est_ener_discr}
&\leq \dt |a^{n}|_{L^\infty(\Omega)} |{q^{n-\frac{1}{2}}}|^2_{L^2(\Omega)} \leq \dt \norme{a}_{L^\infty_{\smT}(Q)}|{q^{n-\frac{1}{2}}}|_{L^2(\Omega)}^2.
\end{align}
Using the following uniform discrete Poincar\'e-type inequality (which is valid even for non-uniform meshes)
\begin{equation*}
|y|^2_{L^2(\Omega)}\leq C\left(-\delh y,y\right)_{L^2(\Omega)}, \quad \forall y\in R^{\sfullmesh}, \quad y=0 \text{ on }\partial\mesh,
\end{equation*}
we obtain from \eqref{eq:est_ener_discr}, as soon as $2\dt\norme{a}_{L^\infty_{\smT}(Q)}<1$, that
\begin{equation}\label{eq:est_disip}
|{q^{n-\frac{1}{2}}}|_{L^2(\Omega)}^2\leq \frac{1}{1-2\dt\norme{a}_{L^\infty_{\smT}(Q)}}|q^{n+\frac{1}{2}}|_{L^2(\Omega)}, \quad n\in\inter{1,M}.
\end{equation}
From estimate \eqref{eq:est_disip} and the useful inequality $e^{2x}>1/(1-x)$ for $0<x<1/2$, we get
\begin{equation}\label{eq:est_disip_final}
|q^{\frac{1}{2}}|^2_{L^2(\Omega)}\leq e^{CT\norme{a}_{L^\infty_{\smT}(Q)}}|q^{n+\frac{1}{2}}|_{L^2(\Omega)}^2, \quad n\in\inter{1,M},
\end{equation}
for some $C>0$ uniform with respect to $\dt$ provided 
\begin{equation}\label{eq:est_cond_scheme}
\dt\norme{a}_{L^\infty_{\smT}(Q)}<1/4.
\end{equation}

Now, we turn our attention to \eqref{eq:car_clean}. Shifting the integral in time with formula \eqref{trans_doub} and since we are adding positive terms, we see that the term in the left-hand side can be bounded as
\begin{equation*}
\tau^3\ddbint_{Q}e^{2\tau\theta\varphi}\theta^3q^2\geq \sum_{n\in\inter{M/4,3M/4}}\dt \,\tau^3\int_{\Omega} (e^{2\tau\theta\varphi})^{n+\frac{1}{2}}(\theta^3)^{n+\frac{1}{2}} |q^{n+\frac{1}{2}}|^2.
\end{equation*}
Recalling that $\varphi$ is negative and independent of time, we deduce that $(e^{2\tau\theta\varphi})^{n+\frac{1}{2}}\geq e^{-\frac{2^5\tau K_0}{3T^2}}$, $n\in\inter{M/4,3M/4}$, where $K_0:=\max_{x\in\overline{\tilde \Omega}}\{-\varphi(x)\}$. Moreover, since $\theta\geq 1/T^2$ for all $t\in[0,T]$, we get
\begin{equation*}
\tau^3\ddbint_{Q}e^{2\tau\theta\varphi}\theta^3q^2\geq \sum_{n\in\inter{M/4,3M/4}}\dt \,\tau^3 e^{-\frac{2^5 \tau K_0}{3T^2}}T^{-6} |q^{n+\frac{1}{2}}|_{L^2(\Omega)}^2.
\end{equation*}
Combining the above estimate with \eqref{eq:est_disip_final} and adding up, we get
\begin{align}\notag
\tau^3\ddbint_{Q}e^{2\tau\theta\varphi}\theta^3q^2 &\geq (T/2-\dt)\,\tau^3 e^{-\frac{C\tau}{T^2}-CT\|a\|_{L^\infty_{\smT}(Q)}}T^{-6} |q^{\frac{1}{2}}|_{L^2(\Omega)}^2 \\ \label{eq:bound_below_car}
&\geq CT e^{-\frac{C\tau}{T^2}-CT\|a\|_{L^\infty_{\smT}(Q)}} |q^{\frac{1}{2}}|^2,
\end{align}
for some $C>0$ uniform with respect to $\dt$. 

From \eqref{eq:est_theta_max} and estimate \eqref{eq:est_disip_final}, we have that
\begin{equation}\label{eq:est_terminal}
\int_{\Omega}\abs{(e^{\tau\theta\varphi}q)^{\frac{1}{2}}}^2+\int_{\Omega}\abs{(e^{\tau\theta\varphi}q)^{M+\frac{1}{2}}}^2 \leq e^{-\frac{4k_0\tau}{\delta T^2}+CT\norme{a}_{L^\infty_{\smT}(Q)}}|q^{M+\frac{1}{2}}|^2_{L^2(\Omega)},
\end{equation}
where we have denoted $k_0:=\min_{x\in\ov{\av{\Omega}}}\{-\varphi(x)\}$. On the other hand, observe that the following estimate holds
\begin{equation*}
e^{2\tau\theta\varphi s^3}\leq \tau^3 2^6T^{-6}\exp\left(-\frac{2^3k_0\tau }{T^2}\right)\leq C, \quad \forall (t,x)\in (0,T)\times \Omega,
\end{equation*}
uniformly for $\tau\geq \frac{3}{8k_0}T^2$. This, together with estimates \eqref{eq:bound_below_car}--\eqref{eq:est_terminal}, can be used in \eqref{eq:car_clean} to obtain
\begin{align*}
|q^{\frac{1}{2}}|^2_{L^2(\Omega)} &\leq CT^{-1}e^{\frac{C\tau}{T^2}+CT\norme{a}_{L^\infty_{\smT}(Q)}}\ddbint_{\omega\times(0,T)}|q|^2 \\
&\quad +Ch^{-2} e^{\frac{\tau}{T^2}(C-\frac{C^\prime}{\delta})+CT\norme{a}_{L^\infty_{\smT}(Q)}}|q^{M+\frac{1}{2}}|^2_{L^2(\Omega)}.
\end{align*}
for any $\tau\geq \tau_2(T+T^2+T^2\norme{a}_{L^\infty_{\smT}(Q)})$ with $\tau_2=\max\{\tau_1,3/8k_0\}$. Observe that for $0<\delta\leq \delta_1<\frac{1}{2}$ small enough, we obtain
\begin{align}\notag
|q^{\frac{1}{2}}|^2_{L^2(\Omega)} &\leq CT^{-1}e^{\frac{C\tau}{T^2}+CT\norme{a}_{L^\infty_{\smT}(Q)}}\ddbint_{Q_{\omega}}|q|^2 \\ \label{eq:obs_almost}
&\quad +Ch^{-2} e^{-C\frac{\tau}{\delta T^2}+CT\norme{a}_{L^\infty_{\smT}(Q)}}|q^{M+\frac{1}{2}}|^2_{L^2(\Omega)}.
\end{align}

\textbf{Step 2. Connection of the discrete parameters.} To conclude the proof, we will connect the parameters $h$, $\dt$, and $\delta$. We recall that the following conditions should be met
\begin{equation}\label{eq:cond_carleman_small}
\frac{\tau h}{\delta T^2}\leq \epsilon_0 \quad\textnormal{and}\quad \frac{\tau^4\dt}{\delta ^4 T^6} \leq \epsilon_0
\end{equation}
along with $h\leq h_0$, $\delta\leq \delta_1$ and $\dt\norme{a}_{L^\infty_{\smT}(Q)}<1/4$. We fix $\tau=\tau_2(T+T^2+T^2\norme{a}^{2/3}_{L^\infty_{\smT}(Q)})$ and define
\begin{equation*}
h_1:=\epsilon_0\left\{\frac{\delta_1}{\tau_2}\left(1+\frac{1}{T}+\norme{a}^{2/3}_{L^\infty_{\smT}(Q)}\right)^{-1}\right\}^{\vartheta}
\end{equation*}
and
\begin{equation}\label{eq:dt_tilde}
\avl{\dt}:= h_1^{4/\vartheta}.
\end{equation}
Notice that such definitions imply that
\begin{equation}\label{eq:h1_dt1_rel}
\frac{h_1\tau^\vartheta}{\delta_1^\vartheta T^{2\vartheta}}=\epsilon_0 \quad\text{and}\quad \frac{\avl{\dt} \tau^4}{\delta_1^4T^8}= \epsilon_0^{4/\vartheta}.
\end{equation}

We choose $h\leq \min\{h_0,h_1\}$ and set
\begin{equation}\label{eq:choose_delta}
\delta=\left(\frac{h}{h_1}\right)^{1/\vartheta}\delta_1 \leq \delta_1,
\end{equation}
whence, we get
\begin{equation}\label{eq:rel_eps_h}
\epsilon_0=\frac{\tau^\vartheta h}{\delta^\vartheta T^{2\vartheta}}\geq \frac{\tau h}{\delta T^2}
\end{equation}
for any $\vartheta\geq 1$, since $T$, $\delta$ are small and $\tau$ is large. This verifies the first condition of \eqref{eq:cond_carleman_small}. On the other hand, from \eqref{eq:dt_tilde} and \eqref{eq:choose_delta}, we have
\begin{equation*}
\epsilon_0^{4/\vartheta}=\frac{\avl{\dt}\tau^4}{\delta_1^4 T^8}=\frac{h^{4/\vartheta}\tau^4}{\delta^4 T^8}.
\end{equation*}
So, by choosing $\dt \leq \min\{T^{-2}h^{4/\vartheta},(4\norme{a}_{L^\infty_{\smT}(Q)})^{-1}\}$ we obtain from the above expression
\begin{equation*}
\frac{\dt \tau^4}{\delta^4 T^6}\leq \epsilon_0^{4/\vartheta}\leq \epsilon_0
\end{equation*}
for any $\vartheta\leq 4$. This verifies the second condition in \eqref{eq:cond_carleman_small} and allow us to conclude that $\vartheta$ should belong to $\inter{1,4}$ to fulfill also \eqref{eq:rel_eps_h}. Additionally, we notice that the particular selection of $\dt$ implies the stability condition \eqref{eq:est_cond_scheme}.

As $\frac{\tau}{\delta T^2}=\left(\frac{\epsilon_0}{h}\right)^{1/\vartheta}$, thanks to \eqref{eq:rel_eps_h}, we have from \eqref{eq:obs_almost}
\begin{align*}
|q^{\frac{1}{2}}|^2_{L^2(\Omega)}&\leq CT^{-1}e^{C(1+\frac{1}{T}+\norme{a}_{L^\infty_{\smT}(Q)}
+T\norme{a}_{L^\infty_{\smT}(Q)})}\ddbint_{\omega\times(0,T)}|q|^2 \\
&\quad + Ch^{-2}e^{-C \frac{\epsilon_0^{1/\vartheta}}{h^{1/\vartheta}}+CT\norme{a}_{L^\infty_{\smT}(Q)}}|q^{M+\frac{1}{2}}|^2_{L^2(\Omega)},
\end{align*}
which yields
\begin{equation*}
|q^{\frac{1}{2}}|^2_{L^2(\Omega)}\leq e^{C_1(1+\frac{1}{T}+\norme{a}_{L^\infty_{\smT}(Q)}
+T\norme{a}_{L^\infty_{\smT}(Q)})}\left(\ddbint_{\omega\times(0,T)}|q|^2+ e^{- \frac{C_2}{h^{1/\vartheta}}}|q^{M+\frac{1}{2}}|^2_{L^2(\Omega)}\right).
\end{equation*}
The proof finishes by recalling the initial condition of system \eqref{eq:adj_heat_sec}.
\end{proof}

\subsection{The linear case: proof of \Cref{thm:main_contr}}

With the result of \Cref{prop:obs_ineq}, we are in position to prove our main $\phi(h)$-null controllability result. 

Let us fix $T\in(0,1)$, $\vartheta\in\inter{1,4}$ and choose $h>0$ small enough according to \Cref{prop:obs_ineq}. Decreasing if necessary the value of $h$, we can always assume that $\dt\leq T^{-2}h^{4/\vartheta}$. Under these conditions, \Cref{prop:obs_ineq} gives the relaxed observability inequality 
\begin{equation}\label{eq:obs_prop_control}
|q^{\frac{1}{2}}|_{L^2(\Omega)}\leq C_{obs}\left(\ddbint_{\omega\times(0,T)}|q|^2+ \phi(h)|q_T|^2_{L^2(\Omega)}\right)^{1/2},
\end{equation}
valid for all the solutions to \eqref{eq:adj_heat_sec}, with $\phi(h)=e^{-\frac{C_2}{h^{1/\vartheta}}}$. 

We introduce the fully-discrete penalized functional 
\begin{equation}\label{eq:func_fully}
J_{\ssh,\sdt}(q_T)=\frac{1}{2}\ddbint_{\omega\times(0,T)}|q|^2+\frac{\phi(h)}{2}|q_T|_{L^2(\Omega)}^2+(g,q^{\frac{1}{2}})_{L^2(\Omega)}, \quad\forall q_T\in\R^{\smesh},
\end{equation}
defined for the solutions to \eqref{eq:adj_heat_sec} and where we recall that $g\in\R^\smesh$ is the initial datum of \eqref{eq:sys_fully_control_sec}. It is not difficult to see that $J$ is continuous and strictly convex. Moreover, from Cauchy-Schwarz and Young inequalities, we have 
\begin{align*}
J_{\ssh,\sdt}(q_T) & \geq \frac{1}{2}\ddbint_{\omega\times(0,T)}|q|^2+\frac{\phi(h)}{2}|q_T|^2_{L^2(\Omega)}-\frac{1}{4C_{obs}^2}|q^{1/2}|^2_{L^2(\Omega)}-C_{obs}^2|y_0|^2_{L^2(\Omega)} \\
&\geq  \frac{1}{4}\ddbint_{\omega\times(0,T)}|q|^2+\frac{\phi(h)}{4}|q_T|^2_{L^2(\Omega)}-C_{obs}^2|y_0|^2_{L^2(\Omega)},
\end{align*}
where we have used inequality \eqref{eq:obs_prop_control} in the second line. This allows us to conclude that $J$ is coercive and thus the existence of a unique minimizer, that we denote by $\avcl{q_T}$, is guaranteed. 

We consider $\avcl{q}$ the solution to \eqref{eq:adj_heat_sec} with initial datum $\avcl{q_T}$. From the Euler-Lagrange equation associated to the minimization of \eqref{eq:func_fully}, we have
\begin{equation}\label{eq:euler_lag}
\ddbint_{\omega\times(0,T)}\avcl{q}q+\phi(h)(\avcl{q_T},q_T)_{L^2(\Omega)}=-(g,q^{\frac{1}{2}})_{L^2(\Omega)}
\end{equation}
for any $q_T\in\R^{\smesh}$. We set the control $v=L_{T;a}^{\ssh,\sdt}(g)=(\mathbf{1}_{\omega_h}\avcl{q}^{\,n+\frac{1}{2}})_{n\in\inter{0,M-1}}$ and consider the solution $y$ to the controlled problem 
\begin{equation}\label{eq:sys_controlled_qopt}
\begin{cases}
(\Dt y)^{n+\frac{1}{2}}-\delh\taup{y}^{n+\frac{1}{2}}+\taup{ay}^{n+\frac{1}{2}}=\mathbf{1}_{\omega_h} \avcl{q}^{\,n+\frac{1}{2}} &n\in\inter{0,M-1}, \\
y^0=g.
\end{cases}
\end{equation}
By duality, we deduce from \eqref{eq:sys_controlled_qopt} and \eqref{eq:adj_heat_sec} that
\begin{equation}\label{eq:duality_opt}
\ddbint_{\omega\times(0,T)}\avcl{q}q=(y^{M},q_T)_{L^2(\Omega)}-(g,q^{\frac{1}{2}})_{L^2(\Omega)}
\end{equation}
for any $q_T\in\R^{\smesh}$, whence, from \eqref{eq:euler_lag} and \eqref{eq:duality_opt} we conclude
\begin{equation}\label{eq:target}
y^{M}=-\phi(h)\avcl{q_T}.
\end{equation}

By taking $q_T=\avcl{q_T}$ in \eqref{eq:euler_lag}, we readily obtain
\begin{equation*}
{\norme{\avcl{q}}^2_{L^2_{\sdmT}(\omega\times(0,T))}}+\phi(h)\abs{\avcl{q_T}}^2_{L^2(\Omega)}=-(g,\avcl{q}^{\frac{1}{2}}) \leq \abs{g}_{L^2(\Omega)}|\avcl{q}^{\frac{1}{2}}|_{L^2(\Omega)}.
\end{equation*}
Hence, with the observability inequality \eqref{eq:obs_prop_control} applied to $\avcl{q}$, we get
\begin{equation*}
\norme{v}_{L^2_{\sdmT}(\omega\times(0,T))}=\norme{\avcl{q}}_{L^2_{\sdmT}(\omega\times(0,T))}\leq C_{obs}|g|_{L^2(\Omega)}
\end{equation*}
and 
\begin{equation}\label{eq:size_qopt}
\sqrt{\phi(h)}|\avcl{q_T}|_{L^2(\Omega)}\leq C_{obs}|g|_{L^2(\Omega)}.
\end{equation}
In this way, the linear map 
\begin{align*}
L_{T;a}^{\ssh,\sdt}:\R^{\smesh}&\to L^2_{\sdmT}(\omega\times(0,T)) \\
g&\mapsto v
\end{align*}
is well-defined and continuous. From the expressions \eqref{eq:target} and \eqref{eq:size_qopt}, we finally get
\begin{equation}\label{eq:conv_target}
|y^{M}|_{L^2(\Omega)}\leq C_{obs}\sqrt{\phi(h)}|y_0|_{L^2(\Omega)}.
\end{equation}
This ends the proof for $\phi(h)=e^{-C_2/h^{1/\vartheta}}$. 

The case of a general function $\phi$ follows similarly and only a minor adjustment is required. For fixed $\vartheta\in\inter{1,4}$ and any given $\phi(h)$ verifying \eqref{eq:inf_intro}, we see that there exists some $h_2>0$ such that 
\begin{equation*}
e^{-C_2/h^{1/\vartheta}}\leq \phi(h)
\end{equation*}
for all $0<h\leq h_2$. Decreasing, if necessary, the value of $h$ and setting $\dt\leq T^{-2} h^{4/\vartheta}$ we can see that inequality \eqref{eq:obs_prop_control} holds for any function $\phi(h)$ verifying \eqref{eq:inf_intro}. In this way, the rest of the proof follows by the same arguments.

\begin{rmk}
To prove the general case $T\geq 1$, it is enough to divide the interval $[0,T]$ in two parts. First, we choose some $T_0<1\leq T$ and set $M_0=\lfloor \tfrac{T_0}{\dt} \rfloor$. From the previous result, we know that there exists a fully-discrete control $v_0=(v_0^{n+\frac{1}{2}})_{n\in\inter{0,M_0-1}}$ with $\norme{v_0}_{L^2_{\sdmT}(0,T;L^2(\R^\smesh))}\leq C_{obs}^{T_0}|g|^{L^2(\Omega)}$ such that $y$ solution to 
\begin{equation}\label{eq:sys_contr_M0}
\begin{cases}
\D \frac{y^{n+1}-y^{n}}{\dt}-\delh y^{n+1}+a^{n+1}y^{n+1}=\mathbf{1}_{\omega_h} v_0^{n+\frac{1}{2}} &n\in\inter{0,M_0-1} \\
y^{n+1}_0=y^{n+1}_{N+1}=0 & n\in \inter{0,M_0-1}, \\
y^0=g,
\end{cases}
\end{equation}
verifies 
\begin{equation}\label{eq:contr_M0}
|y^{M_0}|_{L^2(\Omega)}\leq C_{obs}^{T_0}\sqrt{\phi(h)}|g|_{L^2(\Omega)},
\end{equation}
where $C_{obs}^{T_0}$ is the observability constant corresponding to $T_0$. This defines the state $(y^n)$ for all $n\in\inter{0,M_0}$. 

Now, we set $v^{n+\frac{1}{2}}=0$ for $n\in\inter{M_0,M-1}$ and consider the uncontrolled system
\begin{equation}\label{eq:uncontr_M0}
\begin{cases}
\D \frac{y^{n+1}-y^{n}}{\dt}-\delh y^{n+1}+a^{n+1}y^{n+1}=0 &n\in\inter{M_0,M-1}, \\
y^{n+1}_0=y^{n+1}_{N+1}=0 & n\in \inter{M_0,M-1}, \\
\end{cases}
\end{equation}
with initial data $y^{M_0}$ coming from the sequence \eqref{eq:sys_contr_M0}. Arguing as we did in Step 1 of the proof of \Cref{prop:obs_ineq}, we can obtain an estimate of the form
\begin{equation}\label{eq:ener_M}
|y^{M}|_{L^2(\Omega)}^2\leq \kappa |y^{n}|_{L^2(\Omega)}^2, \quad n\in\inter{M_0,M-1}
\end{equation}
for some $\kappa>0$ only depending on $\|a\|_{L^\infty_{\smT}(Q)}$, $T_0$, and $T$. In this way, combining \eqref{eq:contr_M0} and \eqref{eq:ener_M}, we have constructed a sequence $y=(y^n)_{n\in\inter{0,M}}$ by means of the auxiliary problems \eqref{eq:sys_contr_M0} and \eqref{eq:uncontr_M0} such that 
\begin{equation*}
|y^{M}|_{L^2(\Omega)}\leq C\sqrt{\phi(h)}|g|_{L^2(\Omega)},
\end{equation*}
which is in fact a $\phi(h)$-controllability constraint. 

\end{rmk}

\subsection{The semi-linear case: proof of \Cref{thm:semilinear}}

The proof of this result follows some classical arguments. For this reason, we only give a brief sketch. We define 
\begin{equation*}
g(s):=
\begin{cases}
\frac{f(s)}{s} &\text{if } s\neq 0, \\
f^\prime(0) &\text{if } s=0.
\end{cases}
\end{equation*}
The assumptions on $f$ guarantee that $g$ and $f^\prime$ are well defined, continuous and bounded functions. For $\zeta\in L^2_{\smT}(Q)$, we consider the linear system
\begin{equation}\label{sys_linearized}
\begin{cases}
\D\frac{y^{n+1}-y^{n}}{\dt}-\delh y^{n+1}+g(\zeta^{n+1})y^{n+1}=\mathbf{1}_\omega v^{n+\frac12}, &\quad  n\in\inter{0,M-1}, \\
y^{n+1}_{|\partial \Omega}=0, &\quad n\in\inter{0,M-1}, \\
y^0=y_0. &
\end{cases}
\end{equation}
We set $a_\zeta^{n}=g(\zeta^n)$, so that we have
\begin{equation}\label{unif_coeff}
\|a_\zeta\|_{L^\infty_{\smT}(Q)}\leq K, \quad\forall \zeta\in L^2_{\smT}(Q), 
\end{equation}
where $K$ is the Lipschitz constant of $f$. In view of \Cref{prop:obs_ineq} and \Cref{thm:main_contr}, for $h$ and $\dt$ chosen sufficiently small, i.e., $h\leq \min\{h_0,h_1\}$
\begin{equation*}
h_1= C_1(1+\frac{1}{T}+K^{2/3})^{-\vartheta}
\end{equation*}
and 
\begin{equation}\label{dt_nonlin}
\dt\leq \min \{T^{-2}h^{4/\vartheta},(4K)^{-1}\},
\end{equation}
we can build a control $v_\zeta=L_{T;a_\zeta}^{\ssh,\sdt}(g)$ and the associated controlled solution to \eqref{sys_linearized} such that
\begin{equation}\label{control_zeta}
|y_{\zeta}^M|_{L^2(\Omega)}\leq Ce^{-\frac{C_1}{h^{1/\vartheta}}}|g|_{L^2(\Omega)}, \quad \|v_\zeta\|_{L^2_{\dmT}(0,T;L^2(\omega))}\leq C|y_0|_{L^2(\Omega)}
\end{equation}
where $C_1>0$ and $C=\exp\left[C(1+\frac{1}{T}+K^{2/3}+TK)\right]$ are uniform with respect to $\zeta$ and the discretization parameters $h$ and $\dt$. Notice that by selecting the parameter $\dt$ as in \eqref{dt_nonlin} guarantees the existence of a solution to \eqref{sys_linearized} and also the stability of the discrete scheme.

Define the map
\begin{align*}
\Lambda: L^2_{\smT}(Q)&\to L^2_{\smT}(Q) \\
		\zeta&\mapsto y_\zeta,
\end{align*}
where $y_\zeta$ is the solution to \eqref{sys_linearized} associated to $a_\zeta^n=g(\zeta^n)$, $n\in\inter{1,M}$, and control as in \eqref{control_zeta}. Arguing as in the proof of \Cref{prop:obs_ineq}, we can readily deduce the energy estimate
\begin{equation}\label{ener_unif_K}
\|y_\zeta\|_{L^2_{\mT}(Q)}\leq e^{CT\|a_\zeta\|_\infty}\|v_\zeta\|_{L^2_{\dmT}(Q)}.
\end{equation}
Taking into account \eqref{unif_coeff} and \eqref{ener_unif_K}, we deduce that the image of $\Lambda$ is bounded, implying that there exists a closed convex set in $L_{\smT}^2(Q)$ which is fixed by $\Lambda$. Moreover, it can be easily verified that $\Lambda$ is a continuous map from $L^2_{\smT}(Q)$ into itself (which follows from an adaptation of \cite[Lemma 5.3]{BLR14}), while the uniform estimate
\begin{equation*}
\|y_\zeta\|_{L^2_{\mT}(Q)}\leq C^\prime |y_0|_{L^2(\Omega)}
\end{equation*}
for the solutions to \eqref{sys_linearized} allows to conclude that $\Lambda$ is a compact map since $L^2_{\smT}(Q)=L^2_{\smT}(0,T;\R^\smesh)$ is in fact the finite dimensional space $(\R^{\smesh})^{\smT}=\R^{\smesh\times \smT}$.

All of the previous properties allow us to to apply Schauder fixed point theorem to deduce the existence of $y\in L^2_{\smT}(Q)$ such that $\Lambda(y)=y$. Setting $v=L_{T;a_y}^{\ssh,\sdt}(g)$ we obtain
\begin{equation*}
\begin{cases}
\D\frac{y^{n+1}-y^{n}}{\dt}-\Delta y^{n+1}+f(y^{n+1})=\mathbf{1}_\omega v^{n+\frac12}, &\quad  n\in\inter{0,M-1}, \\
y^{n+1}_{|\partial \Omega}=0, &\quad n\in\inter{0,M-1}, \\
y^0=g, &
\end{cases}
\end{equation*}
which concludes the proof as we have found a control $v$ that drives the solution of the semilinear parabolic system to a final state $y^M$ satisfying estimates \eqref{control_zeta}. 

\section{Further results and remarks}\label{sec:further}

We devote this section to present some concluding remarks regarding the controllability of fully-discrete systems.

\subsection{Some variants on the Carleman estimate}

With the notation introduced in \Cref{sec:not_space,sec:not_time,sec:not_full}, we can readily identify and prove some additional estimates. We enumerate them below.

\begin{enumerate}
\item If one considers instead of ${L}_{\sdmT}$ the following forward-in-time operator
\begin{equation*}
(\tilde{L}_{\sdmT}q)^n:= (\Dtbar q)^n-\delh \taubp{q}^n, \quad n\in \inter{1,M},
\end{equation*}
for any $q\in (\R^{\sfullmesh})^{\overline\sdmT}$ with $(q_{|\partial \Omega})^{n+\frac12}, n\in\inter{0,M-1}$, then \Cref{thm:fully_discrete_carleman} also holds by replacing all the $\tbm$ operators by $\tbp$.
\item 
  With the tools presented in \Cref{app:discrete_things}, \Cref{thm:fully_discrete_carleman} can adapted without major changes to fully-discrete parabolic operators acting on  primal variables in time. For instance, we can consider the parabolic operator 
\begin{equation*}
(L_{\smT}y)^{n+\frac12}=(\Dt y)^{n+\frac12}-\delh\taup{y}^{n+\frac12}, \quad n\in \inter{0,M-1},
\end{equation*}
for all $y\in (\R^{\sfullmesh})^{\overline\smT}$ with $(y_{|\partial \Omega})^{n+1}=0$, $n\in\inter{0,M}$.
Then, under similar conditions to \Cref{thm:fully_discrete_carleman}, we can prove the following estimate
\begin{align}\notag 
&\tau^{-1}\ddbint_{Q}\tcp(e^{2\tau\theta\varphi}\theta^{-1})\left(|\Dt y|^2+|\delh\taup{y}|^2\right)+\tau\ddbint_{Q}\tcp(e^{2\tau\theta\varphi}\theta)|\difh\taup{y}|^2 \\ \notag
&\quad +\tau\ddbint_{Q}\tcp(e^{2\tau\theta\varphi}\theta)|\ov{\difh\taup{y}}|^2+\tau^3\ddbint_{Q} \tcp(e^{2\tau\theta\varphi}\theta^3)\taup{y}^2 \\ \notag
&\leq C\left(\ddbint_{Q}\tcp(e^{2\tau\theta\varphi})|L_{\smT}y|^2+\tau^3\ddbint_{\mathcal{B}\times(0,T)}\tcp(e^{2\tau\theta\varphi}\theta^3)\taup{y}^2\right) \\ \label{eq:carleman_ineq_forw}
&\quad   +C h^{-2}\left(\int_{\Omega}\abs{(e^{\tau\theta\varphi}y)^{0}}^2+\int_{\Omega}\abs{(e^{\tau\theta\varphi}y)^{M}}^2\right).
\end{align}

As in the previous remark, we can adapt the result for the backward-in-time operator
\begin{equation*}
(\tilde{L}_{\mathcal P}y)^{n+\frac12}= - (\Dt y)^{n+\frac12}-\delh\taum{y}^{n+\frac12}, \quad n\in \inter{0,M-1}.
\end{equation*}
\item All of the methodology presented here can be extended to deal with more general problems. On the one hand, estimate \eqref{eq:car_fully_discrete} and the variants presented above hold in a more general multi-dimensional setting. Using the notation introduced in \cite[Section 1.1]{BLR14}, we can work in any dimension as soon as we restrict to finite difference schemes on Cartesian grids. On the other, by mixing the tools in \cite[Section 2]{BLR14} for the space-discrete case and the tools for the time- and fully-discrete settings presented in \Cref{app:discrete_things}, it is possible to prove Carleman estimates for parabolic operators where the Laplacian is replaced by a general operator of the form $-\textnormal{div}(\gamma \nabla \bullet )$ where $x\mapsto \gamma(x)$ is a nicely behaved function. Notice that in the 1-D setting presented here, this amounts to consider a discrete operator of the form $-\difhb(\gamma \difh \bullet)$. However, to focus on the fully-discrete nature of our problem, we avoided these approaches since the notation and computations increase heavily.

\item Lastly, we shall mention that more general non-uniform meshes for the space discretization can be taken into account by following the discussion presented in \cite[Section 5]{BHL10}. These non-uniform meshes are obtained as the smooth image of a uniform grid and then used to deduce controllability results in a more general setting. Thanks to the procedure presented in \cite{BHL10}, we expect that this can be extended also to consider smooth non-uniform meshes for the discretization of the time variable but the details remain to be given.

\end{enumerate}

\subsection{On the connection between $h$ and $\dt$} \label{sec:further_h_dt}

Looking at the proof of the relaxed observability inequality of \Cref{prop:obs_ineq}, we see that a natural condition connecting $h$ and $\dt$ appears and roughly states that we have to choose $\dt \sim h^{4/\vartheta}$ for $\vartheta\in\inter{1,4}$. In turn, this condition comes from \eqref{eq:cond_delta} which allows to control some remainder terms during the proof of our fully-discrete Carleman estimate. 

The introduction of the parameter $\vartheta$ gives an extra degree of freedom while fixing $h$ and $\dt$ without altering \textit{too much} the best decay rate we can achieve for the target, that is, \eqref{eq:conv_target} with $\phi(h)=e^{-C/h^{1/\vartheta}}$ for some uniform $C>0$.

By selecting $\vartheta=1$, we impose a very restrictive condition on $\dt$, i.e. $\dt \sim h^4$, but allows to recover the best decay $|y^{M}|_{L^2(\Omega)}=\O(e^{-C/h})$. This is comparable to the result obtained in \cite{BLR14} in the semi-discrete case. On the other hand, by choosing $\vartheta=4$, we relax the condition to $\dt\sim h$ but at the price of modifying the size of the target and obtaining $|y^{M}|_{L^2(\Omega)}=\mathcal O(e^{-C/h^{1/4}})$ which in turn resembles the convergence obtained for the time-discrete case in \cite{BHS20}. We point out that in any case, the size of the target remains exponentially small and allows to prove a more practical result for any function $\phi(h)$ verifying \eqref{eq:inf_intro}.

In the fully-discrete case, conditions connecting $h$ and $\dt$ for controllability purposes have also appeared in \cite[Theorem 3.5]{BHLR11}. Indeed, by employing a Lebeau-Robbiano type strategy, the authors prove a controllability result for \eqref{eq:sys_fully_control} with $a=\{a^n\}_{n\in\inter{1,M}}\equiv 0$ in which
\begin{equation*}
|y^{M}|_{L^2(\Omega)}\leq Ce^{-C/h^{\gamma}}|g|_{L^2(\Omega)}
\end{equation*}
for some $\gamma>0$ provided $\dt\leq C_T h^{\gamma}$ with $C_T>0$. This obviously yields a better result in the linear case (without potential) and impose less restrictive conditions between $h$ and $\dt$. 

Nonetheless, due to the spectral nature of the Lebeau-Robbiano technique (see \cite{LR95}), the case of a space-and-time-dependent potential $a$ or more general results for the semilinear case or coupled systems discussed below are out of reach. 

\subsection{Some perspectives}

The approach presented in this paper can be used to deal with other less standard control problems for coupled systems in which Carleman estimates are at the heart of the proofs. We present a brief discussion on a couple of open problems that can be addressed with the tools presented here.

\textbf{Insensitizing controls}. Let $\mathcal O\subset \Omega$ be an observation subset and consider the functional
\begin{equation*}
\Psi(y)=\frac{1}{2}\sum_{n=1}^{M}\dt\,|y^n|_{L^2(\mathcal O)}^2
\end{equation*}
and the control system
\begin{equation}\label{heat_discr_ins}
\begin{cases}
\D\frac{y^{n+1}-y^{n}}{\dt}-\delh y^{n+1}=\mathbf{1}_{\omega} v^{n+\frac12}+\xi^{n+\frac12}, &\quad n\in\inter{0,M-1}, \\
y^{n+1}_0=y^{n+1}_{N+1}=0, &\quad  n\in\inter{0,M-1}, \\
y^0=g+\sigma w, &
\end{cases}
\end{equation}
where $g\in L^2(\Omega)$ and $\xi\in L_{\sdmT}^2(Q)$ are given functions and the data of equation \eqref{heat_discr_ins} is incomplete in the following sense: $w\in L^2(\Omega)$ is unknown while $|w|_{L^2(\Omega)}=1$ and $\sigma\in \mathbb R$ is unknown and small enough. The idea is to look for a control $v=(v^{n+\frac12})_{n\in \inter{0, M-1}}$ such that 
\begin{equation}\label{ins_cond}
\left.\frac{\partial \Psi(y)}{\partial \sigma}\right|_{\sigma=0}=0, \quad \forall w\in L^2(\Omega).
\end{equation}
This is the so-called insensitizing problem (see the seminal work \cite{Lio90}) and has been thoroughly studied in different contexts.

The insensitizing control problem is equivalent to study the null-controllability of a cascade system of parabolic PDEs (see, e.g., \cite[Theorem 1]{deT00}).  At the discrete level, \eqref{ins_cond} translates into finding a control $v$ such that 
\begin{equation}\label{ins_0}
q^{\frac{1}{2}}=0,
\end{equation}
where $q=(q^{n+\frac12})_{n\in \inter{0,M-1}}$ can be found from the following forward-backward cascade system
\begin{equation}\label{eq:coupled_ins}
\begin{split}
&\begin{cases}
\D\frac{y^{n+1}-y^{n}}{\dt}-\delh y^{n+1}=\mathbf{1}_\omega v^{n+\frac12}+\xi^{n+\frac12}, & n\in\inter{0,M-1}, \\
y^{n+1}_{|\partial \Omega}=0, & n\in\inter{0,M-1}, \\
y^0=y_0, &
\end{cases}
\\
&\begin{cases}
\D\frac{q^{n-\frac12}-q^{n+\frac12}}{\dt}-\delh q^{n-\frac12}=\mathbf{1}_{\mathcal O} y^{n}, & n\in\inter{1,M}, \\
q^{n-\frac12}_{|\partial \Omega}=0, & n\in\inter{1,M}, \\
q^{M+\frac12}=0. &
\end{cases}
\end{split}
\end{equation}

Of course, we cannot expect to obtain such kind of result for \eqref{heat_discr_ins} but rather a relaxed condition.  
In view of previous results for space-discrete insensitizing problems (see \cite[Theorem 1.4]{BHSdT19}), we think that it is possible to obtain, by means of the {Carleman inequalities \eqref{eq:car_fully_discrete} and \eqref{eq:carleman_ineq_forw}}, a fully-discrete observability for the coupled system \eqref{eq:coupled_ins} (see \cite[Section 4]{BHSdT19}), thus giving a relaxed notion of $\phi(h)$-insensitizing control for \eqref{heat_discr_ins}. Nevertheless, details remain to be given.

\smallskip
\textbf{Hierarchic control}. The notion of hierarchic control, introduced in \cite{Lio94}, looks for a systematic way to combine the notions of optimal control and controllability. To fix ideas, let us consider the system 
\begin{equation}\label{heat_discr_hier}
\begin{cases}
\D\frac{y^{n+1}-y^{n}}{\dt}-\Delta_h y^{n+1}=\mathbf{1}_\omega u^{n+\frac12}+\mathbf{1}_{\mathcal O}v^{n+\frac12}, &\quad n\in\inter{0,M-1}, \\
y^{n+1}_0=y^{n+1}_{N+1}=0, &\quad  n\in\inter{0,M-1}, \\
y^0=g,&
\end{cases}
\end{equation}
where $u$ and $v$ are controls exerted on the sets $\omega,\mathcal O\subset \Omega$ with $\omega\cap\mathcal O=\emptyset$. The original idea in the work \cite{Lio94} is to find controls satisfying simultaneously the following control objectives
\begin{itemize}
\item[(P1)] Find $v$ such that
\begin{equation}\label{eq:p2}
\min_{v\in L^2_{\sdmT}(0,T;\R^\smesh)}\left\{\frac{1}{2}\sum_{n=1}^{M}\dt |y^n-Y^{n}|_{L^2(\Omega)}^2+\frac{\beta}{2}\sum_{n=0}^{M-1}\dt |v^{n+\frac{1}{2}}|^2_{L^2(\mathcal O)}\right\}.
\end{equation}
where $Y=(Y^{n})_{n\in\inter{1,M}}$ is a given function and $\beta>0$. Intuitively, solving \eqref{eq:p2} amounts to find the less-energy control such that $y$ remains close to the target $Y$ and is analogous to a classical optimal control problem.
\item[(P2)] Find $u$ solving the null-controllability problem
\begin{equation}\label{eq:P1}
\begin{cases}
\D \min_{u\in L^2_{\sdmT}{(0,T;\R^{\smesh}})} \frac{1}{2}\sum_{n=0}^{M-1}\dt |u^{n+\frac{1}{2}}|_{L^2(\omega)}^2, \\
\textnormal{subject to } y^{M}=0.
\end{cases}
\end{equation}
\end{itemize}

Following the methodology proposed in \cite{Lio94}, (P1)--(P2) can be solved by employing a Stackelberg strategy and a hierarchy of controls. The first step amounts to look for $v$ (the \textit{follower} control) assuming that $u$ (the \textit{leader} control) is fixed. Using classical optimal control tools (see e.g. \cite{Lio71}) adapted to the discrete setting, we can prove that there exists $v$ solving (P1) and can be characterized by the optimality system
\begin{equation}\label{eq:coupled_hier}
\begin{split}
&\begin{cases}
\D\frac{y^{n+1}-y^{n}}{\dt}-\delh y^{n+1}=\mathbf{1}_\omega u^{n+\frac12}-\frac{1}{\beta} q^{n+\frac12}\mathbf{1}_{\mathcal O}, & n\in\inter{0,M-1}, \\
y^{n+1}_{|\partial \Omega}=0, & n\in\inter{0,M-1}, \\
y^0=y_0, &
\end{cases}
\\
&\begin{cases}
\D\frac{q^{n-\frac12}-q^{n+\frac12}}{\dt}-\delh q^{n-\frac12}=y^{n}-Y^{n}, & n\in\inter{1,M}, \\
q^{n-\frac12}_{|\partial \Omega}=0, & n\in\inter{1,M}, \\
q^{M+\frac12}=0. &
\end{cases}
\end{split}
\end{equation}
Notice that solving (P1) amounts to add an extra equation to the original system \eqref{heat_discr_hier}. Once this is done, according to (P2) it remains to obtain a control $h$ such that $y^{M}=0$. However, as in the insensitizing case, it is not reasonable to expect this but rather the weaker notion of $\phi(h)$-controllability. 

Following the spirit of the continuous case (see, for instance, \cite{AFCS15} for a similar problem in a slightly more general framework), by duality, we should obtain a relaxed observability inequality like
\begin{equation}\label{eq:obs_hier}
|z^{\frac{1}{2}}|^2_{L^2(\Omega)}+\sum_{n=1}^{M}\dt |\Xi^{n}p^{n}|^2_{L^2(\Omega)}\leq C\left(\sum_{n=0}^{M-1}\dt |\mathbf{1}_\omega z^{n+\frac12}|^2_{L^2(\Omega)}+e^{-\frac{C}{h}}|z_T|^2_{L^2(\Omega)}\right)
\end{equation}
for the the solutions to the adjoint system 
\begin{equation*}\label{eq:adj_hier}
\begin{split}
&\begin{cases}
\D\frac{z^{n-\frac12}-z^{n+\frac12}}{\dt}-\delh z^{n-\frac12}=p^{n}, & n\in\inter{1,M}, \\
z^{n-\frac12}_{|\partial \Omega}=0, & n\in\inter{1,M}, \\
z^{M+\frac12}=z_T, &
\end{cases}
\\
&\begin{cases}
\D\frac{p^{n+1}-p^{n}}{\dt}-\delh p^{n+1}=-\frac{1}{\beta}z^{n+\frac12}\mathbf{1}_{\mathcal O}, & n\in\inter{0,M-1}, \\
p^{n+1}_{|\partial \Omega}=0, & n\in\inter{0,M-1}, \\
p^{M+\frac12}=0, &
\end{cases}
\end{split}
\end{equation*}
where $\Xi=(\Xi^n)_{n\in\inter{1,M}}$ is a suitable exponential weight function, typically related to the Carleman weights. As in the insensitizing controls, we believe that this is doable by using the {Carleman estimates \eqref{eq:car_fully_discrete} and \eqref{eq:carleman_ineq_forw}} and some of the procedures introduced here. 

Of course, \eqref{eq:obs_hier} will only yield a $\phi(h)$-controllability result instead of the null-controllability constraint in \eqref{eq:P1}. Nonetheless, obtaining the relaxed inequality \eqref{eq:obs_hier} and performing an analysis for the fully-discrete case similar to \cite[Section 4]{BHSdT19} will complement nicely the very recent results on numerical hierarchic control developed in \cite{dCFC20}.


\appendix

\section{Discrete calculus results}\label{app:discrete_things}

\subsection{Results for space-discrete variables}

The goal of this first section is to provide a self-contained summary of calculus rules for space-discrete operators like $\difh$, $\difhb$, and to provide estimates for the successive applications of such operators on the weight functions. We present the results without a proof and refer the reader to \cite{BHL10} (see also \cite{BLR14}) for a complete discussion. 

To avoid introducing cumbersome notation, we introduce the following continuous difference and averaging operators
\begin{align*}
&\shc^+ f(x):=f(x+\tfrac{h}{2}), \quad \shc^{-}f(x)=f(x-\tfrac{h}{2}), \\
&\dhc f:=\frac{1}{h}\left(\shc^+-\shc^-\right)f, \quad \mhc{f}=\avcl{f}:=\frac{1}{2}\left(\shc^{+}+\shc^{-}\right)f.
\end{align*}
With this notation, the results given below will be naturally translated to discrete versions. More precisely, for a function $f$ continuously defined on $\R$, the discrete function $\difh f$ is in fact $\dhc f$ sampled on the dual mesh $\dmesh$, and $\difhb f$ is $\dhc f$ sampled on the primal mesh $\mesh$. Similar meanings will be used for the averaging symbols $\av{f}$  and $\ov{f}$ (see \eqref{eq:average_dual} and \eqref{eq:average_primal}, respectively) and for more intricate combinations: for instance, $\avl{\delh f}=\avl{\difhb\difh f}$ is the function $\avcl{\dhc\dhc f}$ sampled on $\dmesh$.

\subsubsection{Discrete calculus formulas}

\begin{lem}\label{lem:chain_rule_space}
Let the functions $f_1,f_2$ be continuously defined over $\R$. We have
\begin{equation*}
\dhc(f_1f_2)=\dhc(f_1)\avcl{f_2}+\avcl{f_1}\dhc(f_2).
\end{equation*}
\end{lem}
The translation of the result to discrete functions $f_1,f_2\in\R^{\smesh}$ (resp. $g_1,g_2\in\R^{\sdmesh}$) is
\begin{equation}\label{eq:difh_prod}
\difh(f_1f_2)=\difh(f_1) \avl{f_2} + \avl{f_1}\difh(f_2), \quad \Big(\textnormal{resp. } \difhb(g_1g_2)=\difhb(g_1)\,\ov{g_2}+\ov{g_1}\,\difhb(g_2)\Big).
\end{equation}

\begin{lem}\label{lem:average_product}
Let the functions $f_1,f_2$ be continuously defined over $\R$. We have
\begin{equation*}
\avcl{f_1f_2}=\avcl{f_1}\avcl{f_2}+\frac{h^2}{4}\dhc(f_1)\dhc(f_2).
\end{equation*}
\end{lem}
The translation of the result to discrete functions $f_1,f_2\in\R^{\smesh}$ (resp. $g_1,g_2\in\R^{\sdmesh}$) is
\begin{equation*}
\avl{f_1f_2}=\avl{f_1}\avl{f_2} + \frac{h^2}{4}\difh(f_1)\difh(f_2), \quad \Big(\textnormal{resp. } \ov{f_1f_2}=\ov{f_1}\,\ov{f_2} + \frac{h^2}{4}\difhb(f_1)\difhb(f_2))\Big).
\end{equation*}

\begin{lem}\label{lem:double_average}
Let the function $f$ be continuously defined over $\R$. We have
\begin{equation*}
\mhc^2 f:= \avc{\avc{f}}=f+\frac{h^2}{f}\dhc\dhc f.
\end{equation*}
\end{lem}

The following result provides a discrete integration-by-parts formula and a related identity for averaged functions.
\begin{prop}
Let $f\in \R^{\sfullmesh}$ and $g\in\R^{\sdmesh}$. Then
\begin{align}
 \label{eq:int_by_parts_space}  
 &\D \int_{\Omega} f (\difhb g)=-\int_{\Omega}(\difh f)g+f_{N+1}g_{N+\frac{1}{2}}-f_0 g_{\frac{1}{2}}, \\
 \label{eq:shift_av_space} 
 &\D \int_{\Omega} f\ov{g}=\int_{\Omega}\av{f} g-\frac{h}{2}f_{N+1}g_{N+\frac{1}{2}}-\frac{h}{2}f_0g_{\frac{1}{2}}.
\end{align}
\end{prop}
\begin{lem}\label{lem:oper_discr_sp_cont}
Let $f$ be a sufficiently smooth function defined in a neighborhood of $\overline{\Omega}$. We have
\begin{enumerate}[label=\upshape(\roman*),ref=\thelem(\roman*)]
 \item \label{it:shift_over_cont} $\D \shc^{\pm} f = f \pm \frac{h}{2}\int_{0}^{1} \partial_x f(\cdot\pm \sigma h/2)\d{\sigma}$,
 \inlineitem \label{it:av_over_cont} $\D \mhc^{j} f = f + C_j h^2 \int_{-1}^{1} (1-|\sigma|) \partial_x^2 f(\cdot + l_j\sigma h)\d{\sigma}$, 
 \item \label{it:deriv_over_cont} $\D \dhc^jf =\partial_x^j f +C_j h^2\int_{-1}^{1}(1-|\sigma|)^{j+1}\partial_x^{j+2}f(\cdot+l_j \sigma h)\d{\sigma}$, \quad $j=1,2$, $l_1=\frac{1}{2}$, $l_2=1$.
\end{enumerate}
\end{lem}

\subsubsection{Space-discrete computations related to Carleman weights}

We present here a summary of results related to space-discrete operations performed on the Carleman weigh functions. The estimates presented below are at the heart of our fully-discrete Carleman estimate. The proof of such results can be found on \cite{BHL10}.

We set $r=e^{s\varphi}$ and $\rho=r^{-1}$. The positive parameters $s$ and $h$ will be large and small, respectively. We highligh the dependence on $s,h$ and $\lambda$ in the following results. We always assume that $s\geq 1$ and $\lambda\geq 1$.

\begin{lem}\label{lem:deriv_cont_weights} Let $\alpha,\beta\in\mathbb N$. We have
\begin{align*}\notag
\partial_x^\beta(r\partial_x^\alpha \rho )&=\alpha^\beta(-s\phi)^\alpha\lambda^{\alpha+\beta}(\partial_x \psi)^{\alpha+\beta} \\
&\quad + \alpha\beta (s\phi)^\alpha \lambda^{\alpha+\beta-1}\O(1)+\alpha(\alpha-1)s^{\alpha-1}\mathcal O_{\lambda}(1)= s^\alpha\O_\lambda(1).
\end{align*}
Let $\sigma\in[-1,1]$. We have
\begin{equation*}
\partial_x^\beta\left(r(t,\cdot)(\partial_x^\alpha \rho)(t,\cdot+\sigma h)\right)=\mathcal O_{\lambda}(s^\alpha(1+(sh)^\beta))e^{\mathcal O_{\lambda}(sh)}.
\end{equation*}
Provided $\frac{\tau h}{\delta T^2}\leq 1$, we have $\partial_x^\beta\left(r(t,\cdot)(\partial_x^\alpha \rho)(t,\cdot+\sigma h)\right)=s^\alpha \mathcal O_{\lambda}(1)$. The same expressions hold with $r$ and $\rho$ interchanged if we replace $s$ by $-s$ everywhere. 
\end{lem}

\begin{prop}\label{prop:low_order_derivs} Let $\alpha\in\mathbb N$. Provided $\frac{\tau h}{\delta T^2}\leq 1$, we have
\begin{enumerate}[label=\upshape(\roman*),ref=\thelem(\roman*)]
\item \label{it:r_mhj_dx_rho} $r\mhc^j \partial_x^\alpha \rho=r\partial_x^\alpha \rho+s^\alpha(sh)^2\mathcal O_{\lambda}(1)=s^\alpha\mathcal O_{\lambda}(1)$, $j=1,2$, 
\item \label{eq:r_mh_dr_rho} $r \mhc^j \dhc \rho=r\partial_x \rho+s(sh)^2\mathcal O_{\lambda}(1)$, $j=1,2$,
\item \label{eq:r_dh2_rho} $r\dhc^2\rho=r\partial_x^2\rho+s^2(sh)^2\mathcal O_{\lambda}(1)=s^2\mathcal O_{\lambda}(1)$.
\end{enumerate}
The same estimates hold with $\rho$ and $r$ interchanged.
\end{prop}

\begin{prop} Let $\alpha\in\mathbb N$. For $k=0,1,2$, $j=0,1,2$, and provided $\frac{\tau h}{\delta T^2}\leq 1$ we have
\begin{enumerate}[label=\upshape(\roman*),ref=\thelem(\roman*)]
\item \label{eq:dh_r_dh2rho}$\dhc^k(r\dhc^2 \rho)=\partial_x^k(r\partial_x^2\rho)+s^2(sh)^2\mathcal O_{\lambda}(1)=s^2\mathcal O_{\lambda}(1)$,
\item \label{eq:dh_rmh2rho} $\dhc^k(r\mhc^2 \rho)=(sh)^2\mathcal O_{\lambda}(1)$.
\end{enumerate}
The same estimates hold with $\rho$ and $r$ interchanged.
\end{prop}

\begin{prop} Let $\alpha,\beta\in\mathbb N$ and $k,j=0,1,2$. Provided $\frac{\tau h}{\delta T^2}\leq 1$, we have
\begin{enumerate}[label=\upshape(\roman*),ref=\thelem(\roman*)]
\item \label{eq:high_deriv_r2_rho_dh2_rho} $\mhc^j\dhc^k \partial_x^\alpha\left(r^2\avcl{\dhc\rho}\,\dhc^2\rho\right)=\partial_x^k\partial_x^\alpha (r^2(\partial_x\rho)\partial_x^2\rho)+s^3(sh)^2\mathcal O_{\lambda}(1)=s^3\mathcal O_{\lambda}(1)$,
\item \label{eq:high_deriv_r2_rho_mh2_rho} $\mhc^j \dhc^k \partial_x^\alpha\left(r^2\avcl{\dhc \rho}\,\mhc^2\rho\right)=\partial_x^k\partial_x^\alpha(r\partial_x \rho)+s(sh)^2\mathcal O_{\lambda}(1)=s\mathcal O_{\lambda}(1)$.
\end{enumerate}
The same estimates hold with $\rho$ and $r$ interchanged.
\end{prop}

\begin{rmk}\label{rmk:dh2}
We set $\dhctwo:=((\shc^+)^2-(\shc^-)^2)/2h=\mhc\dhc$ and $\mhctwo:=((\shc^+)^2+(\shc^-)^2)/2$. The estimates presented in the previous results are then preserved when we replace some of the $\dhc$ by $\dhctwo$ and some of the $\mhc$ by $\mhctwo$.
\end{rmk}

%
%
%

\subsection{Results for time-discrete variables}\label{app:time_discrete}

Following the spirit of the previous section, here we devote to provide a summary of calculus rules manipulating the time-discrete operators $\Dt$ and $\Dtbar$, and also to provide estimates for the application of such operators on the weight functions. 

As we did before, to avoid introducing additional notation, we present the following continuous difference operator. For a function $f$ defined on $\R$, we set
\begin{align*}
&\tcp f(t):=f(t+\tfrac{\dt}{2}), \quad \tcm f(t):=f(t-\tfrac{\dt}{2}), \\
&\Dtc f:=\frac{1}{\dt}(\tcp-\tcm)f.
\end{align*}
As for the the space variable, we can obtain discrete versions of the results presented below quite naturally. Indeed, for a function $f$ continuously defined on $\R$, the discrete function $\Dt f$ amounts to evaluate $\Dtc f$ at the mesh points $\dmT$ and $\Dtbar f$ is $\Dtc f$ sampled at the mesh points $\mT$.

\subsubsection{Time-discrete calculus formulas}

\begin{lem}
Let the functions $f_1$ and $f_2$ be continuously defined over $\mathbb R$. We have
\begin{equation*}
\Dtc(f_1f_2)=\tcp{f_1}\,\Dtc f_2+\Dtc f_1\,\tcm{f_2}.
\end{equation*}
The same holds for
\begin{equation*}\label{deriv_prod_2}
\Dtc(f_1f_2)=\tcm{f_1}\,\Dtc f_2+\Dtc f_1\,\tcp{f_2}.
\end{equation*}
From the above formulas, if $f_1=f_2=f$, we have the useful identities
\begin{align*}
&\tcp{f}\,\Dtc f=\frac{1}{2}\Dtc \left(f^2\right)+\frac{1}{2}\dt(\Dtc f)^2,\qquad \tcm{f}\,\Dtc f=\frac{1}{2}\Dtc\left(f^2\right)-\frac{1}{2}\dt(\Dtc f)^2.
\end{align*}
\end{lem}

The translation of the result to discrete functions $f,g_1,g_2\in H^{\overline{\dmT}}$ is 
\begin{equation}\label{deriv_prod}
\begin{split}
&\Dtbar(g_1g_2)=\taubp{g_1}\Dtbar g_2+\Dtbar g_1\taubm{g_2}, \\
&\Dtbar(g_1g_2)=\taubm{g_1}\Dtbar g_2+\Dtbar g_1\taubp{g_2},
\end{split}
\end{equation}
and
\begin{align}\label{f_Dt_2}
&\taubp{f}\Dtbar f=\frac{1}{2}\Dtbar \left(f^2\right)+\frac{1}{2}\dt(\Dtbar f)^2, \\ \label{f_Dt_3}
&\taubm{f}\Dtbar f=\frac{1}{2}\Dtbar\left(f^2\right)-\frac{1}{2}\dt(\Dtbar f)^2.
\end{align}
Of course, the above identities also hold for functions $f,g_1,g_2\in H^{\overline{\mT}}$ and their respective translation operators and difference operator $\mathtt{t^{\pm}}$ and $\Dt$. 

The following result covers discrete integration by parts and some useful related formulas.
\begin{prop}
Let $\{H,(\cdot,\cdot)_{H}\}$ be a real Hilbert space and consider $u\in H^{\overline{\mT}}$ and $v\in H^{\overline{\dmT}}$. We have the following:
\begin{align}\label{trans_doub}
&\dint_{0}^{T}\left(\tcp{u},v\right)_{H}=\int_{0}^{T}\left(u,\bar{\mathtt{t}}^{-}v\right)_{H}, \\
& \dint_{0}^T \left(\tcm u,v\right)_{H}=\dt(u^0,v^{\frac12})_H -\dt(u^M,v^{M+\frac12})_H+\int_{0}^{T}\left(u,\bar{\mathtt{t}}^{+}v\right)_H.\label{trans_doub2}
\end{align}
Moreover, combining the above identities, we have the following discrete integration by parts formula
\begin{equation}\label{int_by_parts}
\dint_{0}^{T} \left(D_t u,v\right)_{H}=-(u^0,v^{\frac12})_{H}+(u^M,v^{M+\frac12})_{H}-\int_{0}^{T}\left(\Dtbar v,u\right)_{H}.
\end{equation}
\end{prop}

\begin{rmk}

 If we consider two functions $f,g\in H^{\overline{\dmT}}$, we can combine \eqref{trans_doub} and \eqref{int_by_parts} to obtain the formula
 \begin{equation}\label{by_parts_same}
 \int_{0}^{T}\left(\Dtbar f,\bar{\mathtt{t}}^{-}g\right)_{H}=-(f^{\frac12},g^{\frac12})_{H}+(f^{M+\frac12},g^{M+\frac12})_{H}-\int_{0}^{T}\left(\bar{\mathtt{t}}^{+} f,\Dtbar g\right)_{H}.
 \end{equation}
 Analogously, for $f,g\in H^{\overline{\mT}}$, the following holds
  \begin{equation}\label{by_parts_same_dual}
 \dint_{0}^{T}\left(D_t f,\tcp{g}\right)_{H}=-(f^{0},g^{0})_{H}+(f^{M},g^{M})_{H}-\dint_{0}^{T}\left(\tcm f,D_t g\right)_{H}.
 \end{equation}
Observe that in these formulas, the integrals are taken over the same discrete points. These will be particularly useful during the derivation of the Carleman estimates \eqref{eq:car_fully_discrete} and \eqref{eq:carleman_ineq_forw}.
\end{rmk}

\subsubsection{Time-discrete computations related to Carleman weights}

We present some lemmas related to time-discrete operations applied to the Carleman weights. The proof of these results can be found in \cite[Appendix B]{BHS20}. We recall that $r=e^{s\varphi}$ and $\rho=r^{-1}$. We highlight the dependence on $\tau$, $\delta$, $\dt$ and $\lambda$ in the following estimates.

\begin{lem}[Time-discrete derivative of the Carleman weight]\label{lem:deriv_lemma_time}
Provided $\dt \tau (T^3 \delta^2 )^{-1} \leq 1$, we have
\begin{equation*}
\tcm{(r)} \Dtc \rho=-\tau\,\tcm{(\theta^\prime)}\varphi +\dt\left(\frac{\tau}{\delta^3T^4}+\frac{\tau^2}{\delta ^4T^6}\right)\mathcal O_{\lambda}(1).
\end{equation*}
\end{lem}

\begin{lem}[Discrete operations on the weight $\theta$]\label{lemma_deriv_theta}
 There exists a universal constant $C>0$ uniform with respect to $\dt$, $\delta$ and $T$ such that
\begin{enumerate}[label=\upshape(\roman*),ref=\thelem(\roman*)]
\item \label{est_dt_square}
$ |\Dtc (\theta^\ell)| \leq \ell T\tcm{(\theta^{\ell+1})} + C \frac{\dt}{\delta^{\ell+2}T^{2\ell+2}}  , \quad \ell=1,2,\ldots$
%
\item \label{est_dt_thp}
$\Dtc(\theta^\prime) \leq C T^2\tcm{(\theta^3)}+C\frac{\dt}{\delta^4T^5}$.

\end{enumerate}
\end{lem}

In addition to the results presented above, since we are dealing with a fully-discrete case, we need to give an additional lemma concerning the effect of the time-discrete operator $\Dtc$ over some discrete operations in the space variable. This is a new result as compared to \cite{BLR14} and \cite{BHS20} but follows the arguments there. The result reads as follows. 

\begin{lem}[Mixed derivatives] \label{lem:mixed_derivatives}Provided $\tau h (\delta T^2)\leq 1$ and $\sigma$ is bounded, we have
\begin{enumerate}[label=\upshape(\roman*),ref=\thelem(\roman*)]
\item\label{der_mixed_cont} $\partial_t^\beta\left(r(t,x)\partial_x^{\alpha}(t,x+\sigma h)\right)=T^\beta s^\alpha(t) \theta^\beta(t)\mathcal O_{\lambda}(1)$, $\beta=1,2$ and $\alpha\in\mathbb N$.
\end{enumerate}
If in addition $\frac{\dt \tau}{T^3\delta^2}\leq \frac{1}{2}$, the following estimates hold
\begin{enumerate}[resume,label=\upshape(\roman*),ref=\thelem(\roman*)]
\item \label{eq:cross_d_mh} $\Dtc(r\, \mhc^2 \rho )= T \tcm(\theta[sh]^2)\mathcal O_{\lambda}(1)+\left(\frac{\tau \dt}{\delta ^3 T^4}\right)\left(\frac{\tau h}{\delta^{} T^2}\right)\mathcal O_{\lambda}(1)$.
\item \label{eq:cross_dt_dh} $\Dtc(r\, \dhc^2 \rho )= T \tcm(s^2\theta )\mathcal O_{\lambda}(1)+ \left(\frac{\tau^2 \dt}{\delta^4 T^6}\right)\mathcal O_{\lambda}(1)+\left(\frac{\tau \dt}{\delta ^3 T^4}\right)\left(\frac{\tau h}{\delta^{} T^2}\right)^3\mathcal O_{\lambda}(1)$.
\end{enumerate}
\end{lem}

\begin{proof}
For $\beta=1$, the proof of item (i) can be found on \cite[Proof of Proposition 2.14]{BLR14}. For $\beta=2$, the proof follows exactly as in that work just by noting that $\partial_t^2 (r\partial_x^\alpha \rho)= T^2 s^\alpha \theta^2 \mathcal O_\lambda(1)$.

To prove item (ii), we shall exploit that the weights $\rho$ and $r$ can be written in separated variables. Indeed, by Lemma \ref{it:av_over_cont}, we have
\begin{equation}\label{eq:r_mh_rho}
r(t,x)\mhc^2\rho(t,x)=1+C h^2 \int_{-1}^{1}(1-|\sigma|)r(t,x)\partial_x^2 \rho(t,x+\sigma h)d\sigma.
\end{equation}
hence
\begin{align}\notag 
&\frac{r(t+\dt,x)\mhc^2\rho(t+\dt,x)-r(t,x)\mhc^2\rho(t,x)}{\dt}\\ \notag
&\quad =\partial_t(r(t,x)\mhc^2 \rho(t,x))+\dt\int_{0}^{1}(1-\gamma)\partial_t^2\left(r(t+\gamma\dt,x)\mhc^2 \rho(t+\gamma \dt,x)\right)d \gamma \\ \label{eq:A1_A2}
& \quad = A_1+A_2.
\end{align}
by a first order Taylor formula in the time variable. Differentiating with respect to $t$ in \eqref{eq:r_mh_rho} and item (i) yield
\begin{equation}\label{eq:est_A1}
A_1=T[s(t)h]^2\theta(t) \mathcal O_{\lambda}(1).
\end{equation}
For estimating $A_2$, we use the change of variable $t\mapsto t+\gamma \dt$ for $\gamma\in[0,1]$ in item (i) and observe that provided $\frac{\dt \tau}{T^3\delta^2} \leq \frac{1}{2}$, we have $\max_{t\in [0,T+\dt]} \theta(t)\leq \frac{2}{\delta T^2}$. Therefore,
\begin{equation}\label{eq:est_A2}
A_2=\dt T^2 s^2(t+ \dt)\theta^2(t+ \dt) h^2 \mathcal O_{\lambda}(1) =  \left(\frac{\tau \dt}{\delta ^3 T^4}\right)\left(\frac{\tau h}{\delta^{} T^2}\right)\mathcal O_{\lambda}(1),
\end{equation}
where we have used that $h\ll 1$ to remove one power of $h$. Putting together \eqref{eq:A1_A2}--\eqref{eq:est_A2} and performing the change of variable $t\mapsto t-\frac{\dt}{2}$yield the desired result. 

Finally, to prove item (iii), we have from Lemma \ref{it:deriv_over_cont} that 
\begin{equation}\label{eq:r_dh_rho}
r(t,x)\dhc^2\rho(t,x)=r(t,x)\partial_x^2 \rho(t,x)+C h^4 \int_{-1}^{1}(1-|\sigma|)^3 r(t,x)\partial_x^4 \rho(t,x+\sigma h)d\sigma.
\end{equation}
Therefore, arguing as above
\begin{align}\notag 
&\frac{r(t+\dt,x)\dhc^2\rho(t+\dt,x)-r(t,x)\dhc^2\rho(t,x)}{\dt}\\ \notag
&\quad =\partial_t(r(t,x)\dhc^2 \rho(t,x))+\dt\int_{0}^{1}(1-\gamma)\partial_t^2\left(r(t+\gamma\dt,x)\dhc^2 \rho(t+\gamma \dt,x)\right)d \gamma \\ \label{eq:B1_B2}
& \quad = B_1+B_2.
\end{align}
Differentiating with respect to $t$ in \eqref{eq:r_dh_rho} and using item (i) yield
\begin{equation*}
B_1=Ts(t)^2\theta(t) \mathcal O_{\lambda}(1)+ h^4\left[Ts(t)^4\theta(t)\mathcal O_{\lambda}(1)\right].
\end{equation*}
Using that $\frac{\tau h}{\delta T^2}\leq 1$ and $h\ll 1$, we can adjust some powers in the above expression and obtain
\begin{equation}\label{eq:est_B1}
B_1= Ts(t)^2\theta(t)\mathcal O_{\lambda}(1).
\end{equation}
For the term $B_2$, we argue as for $A_2$. We use the change of variable $t\mapsto t+\gamma \dt$, $\gamma\in[0,1]$, in item (i) and observe that provided $\frac{\dt \tau}{T^3\delta^2}\leq \frac{1}{2}$, we have $\max_{t\in[0,t+\dt]}\theta(t)\leq \frac{2}{\delta T^2}$. In this way, we get
\begin{align}\notag
B_2&=\dt T^2 s^2(t+\dt)\mathcal\theta^2(t+\dt) \O_{\lambda}(1)+\dt T^2h^4 s^4(t+\dt)\theta^2(t+\dt)\O_{\lambda}(1) \\ \label{eq:est_B2}
&= \frac{\tau^2\dt}{\delta^4 T^6}\mathcal O_{\lambda}(1)+\left(\frac{\tau\dt}{\delta^3 T^4}\right)\left(\frac{\tau h}{\delta T^2}\right)^3\mathcal O_{\lambda}(1).
\end{align}
Collecting the estimates \eqref{eq:est_B1}--\eqref{eq:est_B2} in \eqref{eq:B1_B2} and setting the change of variable $t\mapsto t-\frac{\dt}{2}$ gives the desired result.
\end{proof}

\section{Estimates of the cross product in the Carleman estimate}\label{app:proofs}

\subsection{Estimates that only require space discrete computations}

\subsubsection*{Proof of \Cref{lem:est_I11}}

The proof is straightforward. First, using \eqref{trans_doub}, we can relax a little bit the notation by hiding the operator $\tbm$. More precisely we have
\begin{equation*}
I_{11}=2\dbint_{Q}\tbm(r^2\ov{\av{\rho}}\delh z\ov{\difh{\rho}}\,\ov{\difh{z}})=2\ddbint_{Q} r^2\ov{\av{\rho}}\delh z \ov{\difh \rho}\,\ov{\difh z}.
\end{equation*}
Now, we can focus on the space variable. Noting that $\delh=\difhb\difh$ and $\difhb([\difh z]^2)=2\difhb(\difh z)\ov{\difh z}$ thanks to \Cref{lem:chain_rule_space}, we can integrate by parts (see formula \eqref{eq:int_by_parts_space}) and obtain
\begin{align*}
I_{11}&=2\ddbint_{Q}r^2\ov{\av{\rho}}\delh z\ov{\difh{\rho}}\,\ov{\difh{z}} \\
&=-\ddbint_{Q} {\difh\left(r^2\,\ov{\av{\rho}}\,\ov{\difh\rho}\right)} |\difh{{z}}|^2+\dint_0^{T} \left(r^2\ov{\av{\rho}}\,\ov{\difh{\rho}}\right)_{N+1}\left|\difh{z}\right|^2_{N+\frac{1}{2}} \\
&\quad - \dint_0^{T} \left(r^2\ov{\av{\rho}}\,\ov{\difh{\rho}}\right)_{0}\left|\difh{z}\right|^2_{\frac{1}{2}}.
\end{align*}
Using that $\difh(r^2\,\ov{\av{\rho}}\,\ov{\difh\rho})=-s\lambda^2(\partial_x\psi)^2\phi+s\lambda \phi \mathcal{O}(1)+s(sh)^2\mathcal O_{\lambda}(1)$ (obtained from Lemma \ref{eq:high_deriv_r2_rho_mh2_rho} and \Cref{lem:deriv_cont_weights}) and $r\ov{\av{\rho}}=1+(sh)^2\mathcal O_{\lambda}(1)$ (see Lemma \ref{it:r_mhj_dx_rho}), we get
\begin{align*}
I_{11}&=\ddbint_{Q} s\lambda^2 (\partial_x\psi)^2\phi |\difh{{z}}|^2-\ddbint_{Q} (s\lambda\phi \mathcal O(1)+s(sh)^2\mathcal O_{\lambda}(1))\left|\difh{z}\right|^2 \\
&\quad +\dint_{0}^{T} (1+(sh)^2\mathcal O_{\lambda}(1)) \left[ \left(r \ov{\difh{\rho}}\right)_{N+1} \left|\difh{{z}}\right|^2_{N+\frac{1}{2}} - \left(r\ov{\difh \rho}\right)_0 \left|\difh{{z}}\right|_0^2 \right].
\end{align*}
The result follows by shifting back the time integral with formula \eqref{trans_doub}.

\subsubsection*{Proof of \Cref{lem:est_I12}}

We proceed as in the previous proof. First, we shift the time integral and then using integration by parts in space, we obtain
\begin{align}\notag 
I_{12}&=-2\ddbint_{Q}{s \partial_{xx}\phi r \ov{\av{\rho}} z \delh z} \\ \label{eq:I12_int}
&=2\ddbint_{Q} s\,\difh(r\ov{\av{\rho}}\partial_{xx}\phi)\av{z}\difh z +2 \ddbint_{Q}{s \avl{(r\ov{\av{\rho}}\partial_{xx}\phi)}|\difh z|^2}.
\end{align}
Here, we have used that $(z_{|\partial\Omega})^{n-\frac{1}{2}}=0$ for $n\in\inter{1,M}$ so no boundary conditions appear. 

Using Lemma  \ref{it:deriv_over_cont}, we see that $\avl{\partial_{xx}\phi}=\partial_{xx}\phi+h\mathcal O_{\lambda}(1)$, thus from Proposition \ref{it:r_mhj_dx_rho}, we get
\begin{equation}\label{eq:est_av_q}
\avl{(r\ov{\av{\rho}}\partial_{xx}\phi)}=\partial_{xx}\phi+\left[h+(sh)^2\right]\mathcal O_{\lambda}(1).
\end{equation}
On the other hand, by Propositions \ref{it:r_mhj_dx_rho} and \ref{eq:dh_rmh2rho}, we get
\begin{equation}\label{eq:est_dif_q}
\difh(r\ov{\av{\rho}}\partial_{xx}\phi)= \left[1+(sh)^2\right]\mathcal O_{\lambda}(1).
\end{equation}

Using estimates \eqref{eq:est_av_q}--\eqref{eq:est_dif_q} in \eqref{eq:I12_int}, we see that $I_{12}$ can be written as
\begin{equation}\label{eq:I12_mas_R}
I_{12}=2\ddbint_{Q}s\partial_{xx}\phi|\difh z|^2+ R_{12},
\end{equation}
where
\begin{equation*}
R_{12}:=\ddbint_{Q} \left[(sh)+s(sh)^2\right]\mathcal{O}_{\lambda}(1)\left|\difh z\right|^2+\ddbint_{Q} [s+s(sh)^2]\mathcal O_{\lambda}(1)z\difh{z}.
\end{equation*}
Using that 
\begin{equation}\label{eq:parxx_phi}
\partial_{xx}\phi=\lambda^2|\partial_x\psi|^2\phi+\lambda\phi \mathcal O(1)
\end{equation}
in expression \eqref{eq:I12_mas_R}, the result follows by Cauchy-Schwarz and Young inequalities and shifting the time integral. Notice that we have adjusted some powers of the product $(sh)$ by using that $s\geq 1$ and $(sh)\leq \tau h(\delta T^2)^{-1}\leq 1$.

\subsubsection*{Proof of \Cref{lem:est_I21}}
Define $p:=r^2\delh \rho \ov{\difh \rho}$. From \eqref{trans_doub} and noting that $\ov{\difh z}=\difhb(\av{z})$ we see that
\begin{equation*}
I_{21}=2\ddbint_{Q}p\, \ov{\difh z}\, \ov{\av{z}}=\ddbint_{Q} p\, \difhb(|\av{z}|^2),
\end{equation*}
where we have used \Cref{lem:chain_rule_space}. Integrating by parts in space, we get
\begin{equation*}
I_{21}=\underbrace{-\ddbint_{Q}(\difh p)|\av{z}|^2}_{=:I_{21}^{(1)}}+\underbrace{\dint_{0}^{T}p_{N+1}|\av{z}_{N+\frac{1}{2}}|^2-\dint_{0}^{T}p_0 |\av{z}_{\frac{1}{2}}|^2}_{=: I_{21}^{(2)}}.
\end{equation*}
We observe that $\av{z}_{\frac{1}{2}}=\frac{h}{2}(\difh z)_{\frac{1}{2}}$ and $\av{z}_{N+\frac{1}{2}}=-\frac{h}{2}(\difh z)_{N+\frac{1}{2}}$. Thanks to  \Cref{lem:deriv_cont_weights} and Proposition \ref{eq:high_deriv_r2_rho_dh2_rho}  notice that $p=\left[s^2\mathcal O_{\lambda}(1)+s^2(sh)^2\mathcal O_{\lambda}(1)\right]r\ov{\difh\rho}$ . Thus, 
\begin{align}\notag 
I_{21}^{(2)}&=\dint_{0}^{T}\left[(sh)^2+(sh)^4\right]\mathcal O_\lambda(1)\left\{\left(r\ov{\difh\rho}\right)_0 \left(\difh z\right)_{N+\frac{1}{2}}^2+\left(r\ov{\difh\rho}\right)_{N+1} (\difh z)_{N+\frac{1}{2}}^2\right\} \\ \label{eq:I21_2}
&= \dint_{0}^{T} (sh)^2\mathcal O_\lambda(1)\left\{\left(r\ov{\difh\rho}\right)_0 \left(\difh z\right)_{N+\frac{1}{2}}^2+\left(r\ov{\difh\rho}\right)_{N+1} (\difh z)_{N+\frac{1}{2}}^2\right\}
\end{align}
by using that $sh\leq \tau h(\delta T^2)^{-1}\leq 1$.

From \Cref{lem:average_product}, formula \eqref{eq:shift_av_space} and recalling that $(z_{|\partial \Omega})^{n-\frac{1}{2}}$ for $n\in\inter{1,M}$, we have that
\begin{align}\notag
I_{21}^{(1)}&=-\ddbint_{Q} (\difh p)\avl{|z|^2}+\frac{h^2}{4}\ddbint_{Q}(\difh p)|\difh z|^2 \\ \label{eq:I21_1}
&=-\ddbint_{Q} \ov{\difh p}\,{|z|^2}+\frac{h^2}{4}\ddbint_{Q}(\difh p)|\difh z|^2.
\end{align}

We claim the following.
\begin{claim}\label{cl:der_p}
Provided $\tau h(\delta T^2)^{-1}\leq 1$, we have 
\begin{align*}
&\difh p= s^3\O_{\lambda}(1)+s^3(sh)^2\O_{\lambda}(1), \\
&\ov{\difh p}=-s^3\lambda^4\phi^3(\partial_x\psi)^4+(s\lambda\phi)^3\mathcal O_\lambda(1)+s^2\O_{\lambda}(1)+s^3(sh)^2\O_{\lambda}(1). 
\end{align*}
\end{claim}
\begin{proof}
From Proposition \ref{eq:high_deriv_r2_rho_dh2_rho}, we have $\difh p=\partial_x(r^2(\partial_x \rho)\partial_{xx}\rho)+s^3(sh)^2\mathcal O_{\lambda}(1)$. On the other hand, a straightforward computation gives 
\begin{equation}\label{eq:iden_r2_rx_rxx}
\partial_{x}(r^2(\partial_x\rho)\partial_{xx}\rho)=-3s^3\lambda^4(\partial_x \psi)^4\phi^3+(s\lambda\phi)^3\mathcal O(1)+s^2\mathcal O_{\lambda}(1). 
\end{equation}
Thus, the first result follows immediately. 

For the second one, we note that $\ov{\difh p}={\difh}_{\scriptscriptstyle 2} p$ (see \Cref{rmk:dh2}), whence $\ov{\difh p}= {\difh}_{\scriptscriptstyle 2}(r^2\delh \rho \ov{\difh \rho})$ and using Proposition \ref{eq:high_deriv_r2_rho_dh2_rho} the results follows from \eqref{eq:iden_r2_rx_rxx}.
\end{proof}

Using \Cref{cl:der_p} to estimate in \eqref{eq:I21_1} and recalling \eqref{eq:I21_2} we readily get
\begin{align*}
I_{21}&=3 \ddbint_{Q} s^3\lambda^4\phi^3(\partial_x\psi)^4|z|^2 - \ddbint_{Q} \mu_{21} |z|^2 - \ddbint_{Q}\nu_{21}|\difh z|^2 \\
&\quad + \dint_{0}^{T} (sh)^2\mathcal O_\lambda(1)\left\{\left(r\ov{\difh\rho}\right)_0 \left(\difh z\right)_{N+\frac{1}{2}}^2+\left(r\ov{\difh\rho}\right)_{N+1} (\difh z)_{N+\frac{1}{2}}^2\right\},
\end{align*}
where 
\begin{equation*}
\mu_{21}=(s\lambda\phi)^3\O(1)+s^2\O_{\lambda}(1)+s^3(sh)^2\O_{\lambda}(1), \quad \nu_{21}=s(sh)^2\O_{\lambda}(1).
\end{equation*}
As before, we have adjusted some powers of $(sh)$ in $\nu_{21}$ by recalling that $(sh)\leq \tau h(\delta T^2)^{-1}\leq 1$ and $h\ll1$. The desired result then follows by shifting the integral in time. 

\subsubsection*{Proof of \Cref{lem:est_I22}}
By \Cref{lem:double_average}, we have that $\ov{\av{z}}=z+\frac{h^2}{4}\delh  z$. Thus
\begin{equation}\label{eq:I22_init}
I_{22}=-2\ddbint_{Q}r \delh \rho s\partial_{xx}\phi z^2-\frac{h^2}{2}\ddbint_{Q} r\delh \rho s\partial_{xx}\phi z \delh z =: I_{22}^{(1)}+I_{22}^{(2)}.
\end{equation}

From Proposition \ref{eq:r_dh2_rho} and \Cref{lem:deriv_cont_weights}, we readily have 
\begin{equation}\label{eq:r_delhrho}
r\delh \rho= r\partial_{xx}\rho+s^2(sh)^2\O_\lambda(1)=s^2\lambda^2(\partial_x\psi)^2\phi^2+s\O_{\lambda}(1)+s^2(sh)^2\O_{\lambda}(1)
\end{equation}
whence, combining with \eqref{eq:parxx_phi}, we obtain that
\begin{equation}\label{eq:I22_1}
I_{22}^{(1)}=-2\ddbint_{Q}s^3\lambda^4\phi^3(\partial_{x}\psi)^4 |z|^2-\ddbint_{Q}\mu |z|^2,
\end{equation}
where $\mu=s^3\lambda^3\phi^3\O(1)+s^2\O_{\lambda}(1)+s^3(sh)^2\O_{\lambda}(1)$.

For the term $I_{22}^{(2)}$, we proceed as follows. We set $p_{\phi}:=r\delh \rho\partial_{xx}\phi$ and by using that $(z_{|\partial \Omega})^{n-\frac{1}{2}}=0$ for $n\in\inter{1,M}$, we get after integration by parts in space
\begin{equation*}
I_{22}^{(2)}=\frac{h^2}{2}\ddbint_{Q} s\avl{p_\phi}|\difh z|^2+\frac{h^2}{2}\ddbint_{Q}s\difh p_{\phi}\difh z \av{z}.
\end{equation*}
Noting that $\difh(z^2)=2\av{z}\difh z$, we can integrate once again in the last term of the above to obtain
\begin{equation}\label{eq:I22_2}
I_{22}^{(2)}=\frac{h^2}{2}\ddbint_{Q} s\avl{p_\phi}|\difh z|^2-\frac{h^2}{4}\ddbint_{Q} s \delh p_{\phi} z^2=: J_1+J_2.
\end{equation}

Let us estimate $J_1$ and $J_2$. For the first term, we have from \eqref{eq:parxx_phi} that $\partial_{xx}\phi=\O_{\lambda}(1)$ and from estimate \eqref{eq:r_delhrho}, we have $p_{\phi}=s^2\O_{\lambda}(1)+s^2(sh)^2\O_{\lambda}(1)$. The same estimate holds for $\avl{p_\phi}$. From this fact and adjusting some powers of $(sh)$ yields
\begin{equation}\label{eq:J1}
J_1=\ddbint_{Q}s(sh)^2\O_{\lambda}(1)|\difh z|^2.
\end{equation}

For estimating $J_2$, we claim the following
\begin{claim}\label{claim_J2} Provided $\tau h(\delta T^2)^{-1}\leq 1$ we have $h^2\delh p_{\phi}=s^2(h+h^2)\mathcal O_{\lambda}(1)+(sh)^4\mathcal O_{\lambda}(1)$.
\end{claim}
\begin{proof}
From \eqref{eq:parxx_phi}, notice that 
\begin{equation}\label{eq:prop_norms}
\norme{\partial_{xx}\phi}_\infty=\mathcal O_{\lambda}(1), \quad \|\difh(\partial_{xx}\phi)\|_\infty=\O_{\lambda}(1), \quad  \norme{h\delh(\partial_{xx}\phi)}_\infty=\O_{\lambda}(1).
\end{equation}
On the other hand, a direct computation gives
\begin{equation*}
h^2\delh p_{\phi}= h^2 \delh(\partial_{xx}\phi) \ov{\avl{r\delh \rho}}+2 h^2 \ov{\difh{\partial_{xx}\phi}}\, \ov{\difh(r\delh \rho)}+h \ov{\avl{\partial_{xx}\phi}} \delh(r\delh \rho).
\end{equation*}
From Propositions \ref{it:r_mhj_dx_rho} and \ref{eq:dh_r_dh2rho}, estimates \eqref{eq:prop_norms}, and the fact that $(r\partial_{xx}\phi)=\partial_x(r\partial_{xx}\phi)=\partial_{xx}(r\partial_{xx}\phi)=s^2 \O_{\lambda}(1)$, we obtain the desired result. 
\end{proof}

Using \Cref{claim_J2}, we readily see that
\begin{equation}\label{eq:J2}
J_2=\ddbint_{Q}\left[s^3(h+h^2)+s(sh)^4\right]\mathcal O_{\lambda}(1)|z|^2=\ddbint_{Q} s^2(sh)\mathcal O_{\lambda}(1)|z|^2.
\end{equation}
Putting together \eqref{eq:I22_init}, \eqref{eq:I22_1}--\eqref{eq:I22_2}, \eqref{eq:J1}, and \eqref{eq:J2}, the desired result follows by adjusting some powers of $(sh)$. 

\subsection{Estimates involving time-discrete computations}\label{app:proof_time_easy}

\subsubsection*{Proof of \Cref{lem:est_I31}}
We begin by shifting the integral in time, hence
\begin{equation}\label{eq:I31_init}
I_{31}=-2\ddbint_{Q}\tau\varphi\theta^\prime z r \ov{\difh\rho}\,\ov{\difh{z}} = -2\ddbint_{Q}\tau\theta^\prime \avl{(\varphi r\ov{\difh \rho}z)} \difh z,
\end{equation}
where we have used formula \eqref{eq:shift_av_space} in the second equality. Observe that no boundary conditions appear since $z_{|\partial \Omega}^{n-\frac{1}{2}}=0$ for $n\in\inter{1,M}$. Noting that
\begin{equation*}
\avl{(\varphi r \ov{\difh \rho}z)}= \avl{(\varphi r\ov{\difh \rho})}\, \avl{z} +\frac{h^2}{4} \difh(\varphi r\ov{\difh \rho})\difh z
\end{equation*}
thanks to \Cref{lem:average_product}, we rewrite \eqref{eq:I31_init} as
\begin{align*}
I_{31}&=-2\ddbint_{Q} \tau\theta^\prime \avl{(\varphi r\ov{\difh \rho})} \avl{z} (\difh z)-\frac{h^2}{2}\ddbint_{Q} \tau\theta^\prime \difh(\varphi r\difh \rho)(\difh z)^2 \\
&=-\ddbint_{Q} \tau\theta^\prime \avl{(\varphi r\ov{\difh \rho})}\difh(z^2)-\frac{h^2}{2}\ddbint_{Q} \tau\theta^\prime \difh(\varphi r\difh \rho)(\difh z)^2.
\end{align*}
Integrating by parts in space the first term in the above expression and using that $\difhb \av{p}=\ov{\difh p}$ for $p\in\R^{\smesh}$, we have
\begin{equation*}
I_{31}=\ddbint_{Q} \tau\theta^\prime \ov{\difh(\varphi r\ov{\difh \rho})} z^2-\frac{h^2}{2}\ddbint_{Q} \tau\theta^\prime \difh(\varphi r\difh \rho)(\difh z)^2.
\end{equation*}
Using that $\|\difh \varphi\|_\infty=\O_{\lambda}(1)$ and $\partial_{x}(r\partial_x \rho)=s\O_{\lambda}(1)$, we can prove as in \Cref{cl:der_p} that
\begin{equation*}
\difh(\varphi r\ov{\difh \rho})= \ov{\difh(\varphi r\ov{\difh \rho})} =s\mathcal O_{\lambda}(1)+s(sh)^2\O_{\lambda}(1).
\end{equation*}
Notice that since $(sh)\leq \tau h(\delta T^2)^{-1}\leq 1$, we can further simplify the above estimate and obtain that both derivatives are of order $s\mathcal O_{\lambda}(1)$. Thus, from this remark, we have
\begin{equation*}
I_{31}=\dbint_{Q} \tau\taubm{s \theta^\prime} \mathcal O_{\lambda}(1) \taubm{z}^2+h^2\dbint_{Q} \tau\taubm{s\theta^\prime} \mathcal{O}_{\lambda}(1) |\difh \taubm{z}|^2.
\end{equation*}

Using that $\theta^\prime=(2t-T)\theta^2$ for $t\in[0,T]$, we obtain
\begin{align*}
I_{31}&=\dbint_{Q} T\tbm(s^2\theta)\mathcal O_{\lambda}(1)\taubm{z}^2  + \dbint_{Q} T\tbm(\theta[sh]^2)\mathcal O_{\lambda}(1)|\difh \taubm{z}|^2.
\end{align*}
This ends the proof. 

\subsubsection*{Proof of \Cref{lem:est_I33}}
The proof of this term can be carried out exactly as in \cite{BHS20}, since the space variable does not play any major role. For completeness, we sketch it briefly. 

Using formula \eqref{f_Dt_3}, we write
\begin{equation*}
I_{33}=-\dbint_{Q}\tau\varphi \taubm{\theta^\prime}\taubm{z}\Dtbar z =-\frac{1}{2}\dbint_{Q}\tau \varphi \taubm{\theta^\prime}\Dtbar(z^2)+\frac{\dt}{2} \dbint_{Q}\tau\varphi\taubm{\theta^\prime}(\Dtbar z)^2,
\end{equation*}
and integrating by parts in time using \eqref{by_parts_same}, we get
\begin{align}\notag 
I_{33}&=\frac{1}{2}\int_{\Omega}  \tau \varphi (\theta^\prime)^{\frac{1}{2}}(z^{\frac{1}{2}})^2-\frac{1}{2}\int_{\Omega}\tau\varphi (\theta^\prime)^{M+\frac{1}{2}}(z^{M+\frac{1}{2}})^2 \\ \label{eq:I33_inter}
&\quad +\frac{1}{2}\dbint_{Q}\tau\varphi \Dtbar(\theta^\prime)\taubm{z}^2+\frac{\dt}{2} \dbint_{Q}\tau\varphi\taubm{\theta^\prime}(\Dtbar z)^2.
\end{align}

By definition, we have that 
\begin{equation}\label{eq:deriv_theta_cont}
\theta^\prime=(2t-T)\theta^2,
\end{equation}
thus $(\theta^\prime)^{\frac{1}{2}}<0$ and $(\theta^\prime)^{M+\frac{1}{2}}>0$. Therefore, recalling that $\varphi<0$ for all $x\in \Omega$, we see that the first two terms in \eqref{eq:I33_inter} are positive and therefore can be dropped. A further computation using \eqref{eq:deriv_theta_cont} and Lemma \ref{est_dt_thp} and  yields

\begin{equation*}
I_{33}\geq -\dbint_{Q} \taubm{\mu_{33}} \taubm{z}^2- \dbint_{Q} \taubm{\gamma_{33}}(\Dtbar z)^2
\end{equation*}
with $\mu_{33}=(\tau T^2\theta^3+\frac{\tau \dt}{\delta^4T^5})\mathcal O_{\lambda}(1)$ and $\gamma_{33}=\dt \tau T \theta^2\mathcal O_{\lambda}(1)$, where we have used that $\varphi=\O_{\lambda}(1)$.

\subsection{A new estimate}\label{sec:new_estimate}

\subsubsection*{Proof of \Cref{lem:est_I13}}
This is the most delicate and cumbersome estimate since it combines the action of both space and time discrete results. For clarity, we have divided the proof in three steps.

\textbf{Step 1. An estimate for $I_{13}$}. Using integration by parts in space, we get
\begin{align}\notag 
I_{13}&=\dbint_{Q}\tbm\left(r\ov{\av{\rho}}\delh z\right)\Dtbar{z} \\ \notag
&= -\dbint_{Q} \Dtbar(\difh z) \difh\taubm{z}\tbm{\avl{(r\ov{\av{\rho}})}}-\dbint_{Q} \tbm\left(\difh\left[r\ov{\av{\rho}}\,\right]\right)\difh\taubm{ z} \Dtbar\left(\avl{z}\right)\\ \label{eq:def_I13_1_I13_2}
&=: I_{13}^{(1)}+I_{13}^{(2)},
\end{align}
where we have used that $\Dtbar$ commutes with $\difh$ and $\mh$ in the first and second terms, respectively. Using formula \eqref{f_Dt_3}, we have that the $I_{13}^{(1)}$ can be written as
\begin{equation*}
I_{13}^{(1)}=-\frac{1}{2}\dbint_{Q}\tbm\avl{(r\ov{\av{\rho}})}\Dtbar\left(|\difh z|^2\right)+\frac{\dt}{2}\dbint_{Q} \tbm\avl{(r\ov{\av{\rho}})}\left(\Dtbar(\difh z)\right)^2
\end{equation*}
and integrating by parts in time in the first integral, we get
\begin{align*}
I_{13}^{(1)}&=\dbint_{Q}|\difh\taubp{z}|^2 \Dtbar \avl{(r\ov{\av{\rho}})}-\int_{\Omega} \avl{(r\ov{\av{\rho}})}^{M+\frac{1}{2}} \left|(\difh z)^{M+\frac{1}{2}}\right|^2 \\
&\quad + \int_{\Omega} \avl{(r\ov{\av{\rho}})}^{\frac{1}{2}} \left|(\difh z)^{\frac{1}{2}}\right|^2+\frac{\dt}{2}\dbint_{Q}  \tbm\avl{(r\ov{\av{\rho}})}\left(\Dtbar(\difh z)\right)^2.
\end{align*}
From Proposition \ref{it:r_mhj_dx_rho}, we observe that for $0<(sh)\leq \frac{\tau h}{\delta T^2}<\epsilon_1(\lambda)$ with $\epsilon_1(\lambda)$ small enough, we have that $\left(r\ov{\av{\rho}}\right)\geq c_{\lambda}>0$, thus the last three terms of the above equation have prescribed signs (we remark that the extra average does not affect the form of that estimate). That is
\begin{align}\notag 
I_{13}^{(1)}&\geq K - c_{\lambda} \int_{\Omega} \left|(\difh z)^{M+\frac{1}{2}}\right|^2 \\ \label{est:I13_inter}
&\quad +c_{\lambda} \int_{\Omega} \left|(\difh z)^{\frac{1}{2}}\right|^2+c_{\lambda}\dt \dbint_{Q} \left(\Dtbar(\difh z)\right)^2,
\end{align}
and where $K:=\dbint_{Q}|\difh\taubp{z}|^2 \Dtbar\avl{(r\ov{\av{\rho}})}$. We claim the following.

\begin{claim}\label{claim:deriv_prom_cross_prom}
Provided $\frac{\dt \tau }{\delta^2 T^3}\leq \frac{1}{2}$, we have 
\begin{equation}\label{eq:bound_dt_av}
\Dtbar\avl{(r\ov{\av{\rho}})}=T \tbm(\theta[sh]^2)\mathcal O_{\lambda}(1)+\left(\frac{\tau \dt}{\delta ^3 T^4}\right)\left(\frac{\tau h}{\delta^{} T^2}\right)\mathcal O_{\lambda}(1).
\end{equation}
\end{claim}
\begin{proof}
Notice that we can write $2 \Dtbar\avl{(r\ov{\av{\rho}})}=(\sh^{+}+\sh^{-}) \Dtbar (r\ov{\av{\rho}})$. The result follows by applying Lemma \ref{eq:cross_d_mh}. 
\end{proof}

Observe that the condition on $\dt$ of \Cref{claim:deriv_prom_cross_prom} is in agreement with our initial hypothesis. Actually, the assumption of \Cref{lem:est_I13} is stronger than this condition. Using that $|\difh\taubp{z}|^2\leq C |\difh\taubm{z}|^2+C(\dt)^2(\Dtbar(\difh z))^2$ and \eqref{eq:bound_dt_av} we get
\begin{align} \label{eq:K_bound}
K&\geq - \dbint_{Q}|\difh \taubm{z}|^2 \taubm{\nu_K} - C_{\lambda} \dt \dbint_{Q}  \left(\frac{\tau \dt}{\delta ^3 T^4}\right)\left(\frac{\tau h}{\delta^{} T^2} \right) \left(\Dtbar(\difh z)\right)^2,
\end{align}
for some $C_\lambda>0$ uniform with respect to $\dt$ and where $\nu_K:=T \theta (sh)^2 \mathcal O_{\lambda}(1)+\left(\frac{\tau \dt}{\delta ^3 T^4}\right)\left(\frac{\tau h}{\delta^{} T^2}\right)\mathcal O_{\lambda}(1)$. Notice that taking $\epsilon_1(\lambda)$ small enough in our initial hypothesis the last term of \eqref{est:I13_inter} controls the last term of \eqref{eq:K_bound}. So, overall, $I_{13}^{(1)}$ can be bounded as
\begin{equation}\label{eq:est_I13_1_final}
I_{13}^{(1)}\geq - c_{\lambda} \int_{\Omega} \left|(\difh z)^{M+\frac{1}{2}}\right|^2 - \dbint_{Q}|\difh \taubm{z}|^2 \taubm{\nu_K}+ \tilde{c}_{\lambda}\dt \dbint_{Q} \left(\Dtbar(\difh z)\right)^2,
\end{equation}
for a constant $\tilde c_{\lambda}>0$ uniform with respect to $\dt$.

Let us comeback to the term $I_{13}^{(2)}$. From Proposition \ref{eq:dh_rmh2rho}, we have that $\difh[r\ov{\av{\rho}}]=(sh)^2\mathcal O_\lambda(1)$ and using Cauchy-Schwarz and Young inequalities, we get
\begin{equation*}
|I_{13}^{(2)}| \leq C \left(\dbint_{Q} \tbm\left(s^{-1}[sh]^2\right) |\Dtbar(\av{z})|^2+ \dbint_{Q} \tbm(s[sh]^2)|\difh\taubm{z}|^2\right),
\end{equation*}
for some $C>0$ only depending on $\lambda$. A direct computation shows that $|\Dtbar(\avl{z})|^2=|\avl{(\Dtbar z)}|^2\leq \avl{(\Dtbar z)^2}$ by convexity. This, together the fact that $z_{|\partial\Omega}^{n-\frac{1}{2}}=0$ for $n\in\inter{1,M}$, we can use \eqref{eq:shift_av_space} to deduce
\begin{equation}\label{eq:est_I13_2_final}
|I_{13}^{(2)}| \leq C \left(\dbint_{Q} \tbm\left(s^{-1}[sh]^2\right) |\Dtbar(z)|^2+ \dbint_{Q} \tbm(s[sh]^2)|\difh\taubm{z}|^2\right).
\end{equation}
Combining \eqref{eq:est_I13_1_final} and \eqref{eq:est_I13_2_final} will provide a lower bound for $I_{13}$. This concludes Step 1.

\smallskip
\textbf{Step 2. An estimate for $I_{23}$}. Using that $\ov{\av{z}}=z+\frac{h^2}{4}\ov{\difh}\difh z$, we can rewrite $I_{23}$ as
\begin{align*}
I_{23}&= \dbint_{Q}\tbm\left(r\delh \rho \ov{\av{z}}\right) \Dtbar{z} \\
&=\dbint_{Q}\tbm\left(r\delh \rho\, z\right)\Dtbar z+\frac{h^2}{4}\dbint_{Q}\tbm\left(r\delh \rho \right)\delh\taubm{z}\Dtbar z =: I_{32}^{(1)}+I_{32}^{(2)}.
\end{align*}

Let us estimate $I_{23}^{(1)}$. Using identity \eqref{f_Dt_3} and the integration-by parts formula in \eqref{by_parts_same}, we see that
\begin{align*}
I_{23}^{(1)}&=\frac{1}{2}\dbint_{Q} \tbm\left(r\delh \rho\right)\Dtbar(|z|^2)-\frac{\dt}{2}\dbint_{Q}\tbm\left(r\delh \rho\right) (\Dtbar z)^2 \\
&= -\frac{1}{2}\dbint_{Q}\Dtbar(r\delh \rho)\taubm{z}^2+\frac{1}{2}\int_{\Omega} (r\delh \rho)^{M+\frac{1}{2}} |z^{M+\frac{1}{2}}|^2-\frac{1}{2}\int_{\Omega} (r\delh \rho)^{\frac{1}{2}} |z^{\frac{1}{2}}|^2- \\
&\quad -\frac{\dt}{2}\dbint_{Q}\tbm\left(r\delh \rho\right) (\Dtbar z)^2.
\end{align*}
To estimate the last three terms of the above expression we can use that 
\begin{equation}\label{eq:est_rdelhrho}
r\delh \rho=s^2\mathcal O_{\lambda}(1)+s\lambda\phi \mathcal O(1)+s^2(sh)^2\mathcal O_{\lambda}(1)=s^2\mathcal O_{\lambda}(1). 
\end{equation}
For the first one we can use directly Lemma \ref{eq:cross_dt_dh} sampled on the primal mesh $\mT$. Thus, 
\begin{align}\notag 
I_{23}^{(1)}&= \dbint_{Q}\taubm{\mu_P}\taubm{z}^2+\int_{\Omega} (s^{M+\frac{1}{2}})^2\mathcal O_{\lambda}(1)|z^{M+\frac{1}{2}}|^2+\int_{\Omega} (s^{\frac{1}{2}})^2\mathcal O_{\lambda}(1)|z^{\frac{1}{2}}|^2 \\ \label{eq:est_I23_1_final}
&\quad + \dt \dbint_{Q} \taubm{s^2}\mathcal O_{\lambda}(1)(\Dtbar z)^2,
\end{align}
where $\mu_P=T s^2\theta \mathcal O_{\lambda}(1)+ \left(\frac{\tau^2 \dt}{\delta^4 T^6}\right)\mathcal O_{\lambda}(1)+\left(\frac{\tau \dt}{\delta ^3 T^4}\right)\left(\frac{\tau h}{\delta^{} T^2}\right)^3\mathcal O_{\lambda}(1)$.

We focus now on the term $I_{23}^{(2)}$. Using that $\delh z=\difhb(\difh z)$, we can integrate by parts in space and get
\begin{align}\notag 
I_{23}^{(2)}&=\frac{h^2}{4}\dbint_{Q}\tbm\left(r \delh \rho \right) \delh\taubm{z}\Dtbar z \\ \notag
&=-\frac{h^2}{4}\dbint_{Q} \tbm\avl{\left(r \delh \rho\right)}\difh\taubm{z}\Dtbar(\difh z) -\frac{h^2}{4}\dbint_{Q} \tbm\left(\difh(r\delh \rho)\right)\difh\taubm{z} \Dtbar \left(\avl{z}\right)\\ \label{eq:est_I23_init}
&=: \mathcal J_1+\mathcal J_2,
\end{align}
where we have used once again that $\Dtbar$ commutes with the space-discrete operations $\difh$ and $\mh$. 

Arguing as we did for $I_{12}^{(1)}$, we see that 
\begin{align*}
\mathcal J_1 &= -\frac{h^2}{8}\dbint_{Q}\tbm\avl{\left(r\delh \rho\right)}\Dtbar(|\difh z|^2)+\dt\frac{h^2}{8}\dbint_{Q}\tbm\avl{\left(r\delh \rho \right)}|\Dtbar(\difh z)|^2 \\
&=-\frac{h^2}{8}\int_{\Omega}\avl{(r\delh \rho)}^{M+\frac{1}{2}} |(\difh z)^{M+\frac{1}{2}}|^2+\frac{h^2}{8}\int_{\Omega}\avl{(r\delh \rho)}^{\frac{1}{2}} |(\difh z)^{\frac{1}{2}}|^2 \\
&\quad + \frac{h^2}{8}\dbint_{Q}\Dtbar\avl{(r\delh \rho)}|\difh\taubp{z}|^2+ \dt \frac{h^2}{8}\dbint_{Q}\tbm\avl{\left(r\delh \rho \right)}|\Dtbar(\difh z)|^2.
\end{align*}

We have the following.

\begin{claim}\label{claim:deriv_prom_cross_deriv}
Provided $\frac{\dt \tau }{\delta^2 T^3}\leq \frac{1}{2}$, we have 
\begin{equation}\label{eq:bound_dt_av}
\Dtbar\avl{(r\delh \rho)}=T \tcm(s^2\theta )\mathcal O_{\lambda}(1)+ \left(\frac{\tau^2 \dt}{\delta^4 T^6}\right)\mathcal O_{\lambda}(1)+\left(\frac{\tau \dt}{\delta ^3 T^4}\right)\left(\frac{\tau h}{\delta^{} T^2}\right)^3\mathcal O_{\lambda}(1).
\end{equation}
\end{claim}

The proof is exactly as in \Cref{claim:deriv_prom_cross_prom} but taking into account the estimate in Lemma \ref{eq:cross_dt_dh}.

Using \eqref{eq:est_rdelhrho} (and noting that the estimate does not change with the extra average in space), \Cref{claim:deriv_prom_cross_deriv} and the fact that $|\difh\taubp{z}|^2\leq C |\difh\taubm{z}|^2+C(\dt)^2(\Dtbar(\difh z))^2$, we compute after a long but straightforward computation that
\begin{align} \notag
|\mathcal J_1| & \leq C\int_{\Omega} (s^{M+\frac{1}{2}}h)^2|(\difh z)^{M+\frac{1}{2}}|^2+C\int_{\Omega} (s^{\frac{1}{2}}h)^2|(\difh z)^{\frac{1}{2}}|^2 \\ \notag
&\quad + C\dbint_{Q} T \tbm{\theta[sh]^2}|\difh \taubm{z}|^2+ C\dbint_{Q}\left[ \left(\frac{\tau^2 \dt}{\delta^4 T^6}\right)+\left(\frac{\tau \dt}{\delta ^3 T^4}\right)\left(\frac{\tau h}{\delta^{} T^2}\right)^3 \right]  |\difh \taubm{z}|^2 \\ \notag
&\quad +C\dt \dbint_{Q} \left(\frac{\tau\dt}{\delta^2T^3}\right)\left(\frac{\tau h}{\delta T^2}\right)|\Dtbar(\difh z)|^2+C\dt \dbint_{Q}\left(\frac{\tau h}{\delta T^2}\right)^2|\Dtbar(\difh z)|^2 \\ \label{eq:est_J1}
&\quad + C \dt \dbint_{Q}\left[ \left(\frac{\tau^2 \dt}{\delta^4 T^6}\right)+\left(\frac{\tau \dt}{\delta ^3 T^4}\right)\left(\frac{\tau h}{\delta^{} T^2}\right)^3 \right]  |\Dtbar(\difh{z})|^2,
\end{align}
for some positive $C$ only depending on $\lambda$. Above we used that $h\ll 1$ to adjust some powers of $h$. We also notice that in the above equation, all of the terms containing $|\Dtbar(\difh z)|^2$ can be made small enough by the initial hypothesis of the lemma. This will be important in the next step since those terms will be absorbed by the corresponding one in \eqref{eq:est_I13_1_final}.

We turn our attention to the term $\mathcal J_2$. Estimate \eqref{eq:est_rdelhrho} and a quick computation yield
\begin{align}\notag 
|\mathcal J_2| &\leq C\left(\dbint_{Q}\tbm\left(s[sh]^2\right)|\difh\taubm{z}|^2+\dbint_{Q}\tbm\left(s^{-1}[sh]^2\right)|\Dtbar(\av{z})|^2\right) \\ \label{eq:est_J2}
&\leq C\left(\dbint_{Q}\tbm\left(s[sh]^2\right)|\difh\taubm{z}|^2+\dbint_{Q}\tbm\left(s^{-1}[sh]^2\right)|\Dtbar z|^2\right),
\end{align}
where the second inequality is obtained by arguing exactly as for the term $I_{13}^{(2)}$.

\smallskip
\textbf{Step 3. Conclusion.} Now, we are in position to find a lower bound for $I_{13}+I_{23}$. First, we shall notice that the last three terms of \eqref{eq:est_J1} can be controlled by the last term of \eqref{eq:est_I13_1_final} by decreasing (if necessary) the parameter $\epsilon_1(\lambda)$ in our initial hypothesis. Then, we just need to combine \eqref{eq:est_I13_1_final}, \eqref{eq:est_I13_2_final}, \eqref{eq:est_I23_1_final}, \eqref{eq:est_I23_init}, \eqref{eq:est_J1}, and \eqref{eq:est_J2} to obtain
\begin{align*} 
I_{13}+I_{23}&\geq - c_{\lambda} \int_{\Omega} \left|(\difh z)^{M+\frac{1}{2}}\right|^2 - \dbint_{Q} \taubm{\mu_+}\taubm{z}^2 - \dbint_{Q}\taubm{\nu_{+}} |\difh \taubm{z}|^2  \\
& \quad + c^\prime_{\lambda}\dt \dbint_{Q} \left(\Dtbar(\difh z)\right)^2 - \dbint_{Q} \taubm{ \gamma_{+}}|\Dtbar z|^2 \\
& \quad - \int_{\Omega} (s^{M+\frac{1}{2}})^2\mathcal O_{\lambda}(1)|z^{M+\frac{1}{2}}|^2-\int_{\Omega} (s^{\frac{1}{2}})^2\mathcal O_{\lambda}(1)|z^{\frac{1}{2}}|^2 \\
& \quad - \int_{\Omega} (s^{M+\frac{1}{2}}h)^2\mathcal O_{\lambda}(1)|(\difh z)^{M+\frac{1}{2}}|^2-\int_{\Omega} (s^{\frac{1}{2}}h)^2\mathcal O_{\lambda}(1)|(\difh z)^{\frac{1}{2}}|^2,
\end{align*}
for some $c_\lambda,c^\prime_{\lambda}>0$ uniform with respect to $h$ and $\dt$ and where
\begin{align*}
\nu_{+}&:= \left\{T \theta (sh)^2 +s(sh)^2+\left(\frac{\tau \dt}{\delta ^3 T^4}\right)\left(\frac{\tau h}{\delta^{} T^2}\right) + \left(\frac{\tau^2 \dt}{\delta^4 T^6}\right)+\left(\frac{\tau \dt}{\delta ^3 T^4}\right)\left(\frac{\tau h}{\delta^{} T^2}\right)^3\right\}\mathcal O_{\lambda}(1),  \\
\mu_{+}&:=\left\{T s^2\theta + \left(\frac{\tau^2 \dt}{\delta^4 T^6}\right)+\left(\frac{\tau \dt}{\delta ^3 T^4}\right)\left(\frac{\tau h}{\delta^{} T^2}\right)^3\right\}\mathcal O_{\lambda}(1), \\
\gamma_{+}&:= \left\{s^{-1}(sh)^2+\dt s^2 \right\}\O_{\lambda}(1).
\end{align*}
This ends the proof. 

\section{Some intermediate lemmas}\label{sec:aux_lemmas}

\subsection{Proof of \Cref{lem:gradient_adjoint}}

By shifting the integral in time (see \cref{trans_doub}) and then using \eqref{eq:shift_av_space}, we have
\begin{equation}\label{eq:grad_init}
\tau \lambda^2\dbint_{Q}\taubm{\theta}|\difh\taubm{z}|^2=\underbrace{\lambda^2 \ddbint_{Q} s \ov{\phi |\difh z|^2}}_{=:\mathcal D}+\frac{h}{2}\dint_{0}^{T}s\lambda^2(\phi |\difh z|^2)_{\frac{1}{2}}+\frac{h}{2}\dint_{0}^{T}s\lambda^2(\phi |\difh z|^2)_{N+\frac{1}{2}},
\end{equation}
where we recall that $s=\tau\theta$. Since $\phi$ is a positive function, notice that the last two terms of the above expression are positive. 

Let us focus on the term $\mathcal D$. From \Cref{lem:average_product}, we get
\begin{equation*}
\mathcal D=\lambda^2\ddbint_{Q}s \ov{\phi}\,\ov{|\difh z|^2}+\frac{h^2}{4}\ddbint_{Q}s\lambda^2 \difhb(\phi)\difhb\left(|\difh z|^2\right)=:\mathcal D_1+\mathcal D_2.
\end{equation*}
For the term $\mathcal D_1$, using once again  \Cref{lem:average_product} gives
\begin{equation*}
\mathcal D_1=\lambda^2\ddbint_{Q} s \ov{\phi}|\ov{\difh z}|^2+\frac{h^2\lambda^2}{4}\ddbint_{Q}s\ov{\phi}(\delh z)^2,
\end{equation*}
where we have used that $\difhb\difh = \delh $. On the other hand, integrating by parts in the space variable, from $\mathcal D_2$, we get
\begin{equation}\label{eq:D2_iden}
\mathcal D_2=-\frac{h^2}{4}\ddbint_{Q}s\lambda^2 \difh(\difhb \phi)(\difh z)^2+\frac{h^2}{4}\ddbint_{Q}s\lambda^2(\difhb \phi)_{N+1}(\difh z)^2_{N+\frac{1}{2}}-\frac{h^2}{4}\ddbint_{Q}s\lambda^2(\difhb \phi)_{0}(\difh z)^2_{\frac{1}{2}}.
\end{equation}

Notice that $\difhb (\phi)=\partial_x\phi+h^2\mathcal O_{\lambda}(1)=\mathcal O_{\lambda}(1)$ thanks to Lemma \ref{it:deriv_over_cont}. Thus, once the parameter $\lambda$ is fixed, we can choose $h$ sufficiently small such that the last two terms of \eqref{eq:grad_init} control the last two terms of \eqref{eq:D2_iden}. Moreover, using items (ii) and (iii) of \Cref{lem:oper_discr_sp_cont}, we see that $\difh(\difhb \phi)=\partial_x^2\phi+h^2\mathcal O_{\lambda}(1)=\O_{\lambda}(1)$ and $\ov{\phi}=\phi+h\mathcal O_{\lambda}(1)$. From these and putting together \eqref{eq:grad_init}--\eqref{eq:D2_iden} we obtain
\begin{align*}
\tau\lambda^2\dbint_{Q}\taubm{\theta}|\difh\taubm{z}|^2&\geq \lambda^2\ddbint_{Q}s\phi |\ov{\difh z}|^2+\ddbint_{Q}(sh)\mathcal O_{\lambda}(1)|\ov{\difh z}|^2 \\
&\quad + \frac{\lambda^2 h^2}{4}\ddbint_{Q} s\phi (\delh z)^2+h^4\ddbint_{Q} s \mathcal O_{\lambda}(1)(\delh z)^2 \\
&\quad + h^2\ddbint_{Q} s\mathcal O_{\lambda}(1)(\difh z)^2.
\end{align*}
Shifting the integrals in time in the right-hand side of the above expression yields the desired result. 

\subsection{Proof of \Cref{lem:dt_absorb}}
We begin by increasing, if necessary, the parameter $\tau_1$ such that $\tau_1\geq 1$ and $\tau\geq 1$. Notice that
\begin{equation}\label{eq:prop_s}
1\leq \tau_1\left(\tfrac{1}{T}+1\right)\leq \tau\theta(t)=s(t), \quad t\in[0,T].
\end{equation}
We repeat the definition of $\underline{X_2}$ for convenience. We have
\begin{align}\notag
\underline{X_2}&= \left\{\left(\frac{\tau \dt}{\delta^4T^5}\right) + \left(\frac{\tau^2 \dt}{\delta^4 T^6}\right) + \left(\frac{\dt \tau}{\delta^3T^4}\right)^2+\left(\frac{\dt \tau^2}{\delta^4 T^6}\right)^2 + \left(\frac{\tau \dt}{\delta ^3 T^4}\right)\left(\frac{\tau h}{\delta^{} T^2}\right)^3\right\} \dbint_{Q}\taubm{z} \\ \label{eq:X2_def_app}
&\quad + \left\{\left(\frac{\tau^2 \dt}{\delta^4 T^6}\right)+ \left(\frac{\tau \dt}{\delta ^3 T^4}\right)\left(\frac{\tau h}{\delta^{} T^2}\right)+\left(\frac{\tau \dt}{\delta ^3 T^4}\right)\left(\frac{\tau h}{\delta^{} T^2}\right)^3 \right\}\dbint_{Q}|\difh\taubm{z}|^2.
\end{align}
We remark that at this point, we have the condtion $\frac{\tau h}{\delta T^2}\leq \epsilon_3$ for some $\epsilon_3=\epsilon_3(\lambda)$ small enough (see  \cref{eq:cond_inter}). Recalling that $\delta\leq 1/2$, $0<T<1$ and \eqref{eq:prop_s} allows us to see that provided
\begin{equation}\label{eq:cond_kappa_1}
\frac{\tau^2\dt}{\delta^4T^6}\leq \kappa_1
\end{equation}
for some $\kappa_1>0$ small enough, the term $\underline{X_2}$ can be bounded as
\begin{equation}\label{eq:est_X2_under}
\underline{X_2}\leq \left(2\kappa_1+2\kappa_1^2+\kappa_1\epsilon_3^2\right)\dbint_{Q}\taubm{s}^3\taubm{z}^2 \leq 5 \kappa_1\dbint_{Q}\taubm{s}^3\taubm{z}^2.
\end{equation}
Notice that in the first inequality we have the product of the small parameters $\kappa_1$ and $\epsilon_3$. Since $\epsilon_3$ has been already chosen small we can bound it uniformly by 1.

On the other hand, from definition \eqref{def:under_W} and rewriting it as $\underline{W}=\sum_{i=1}^{3}W^{(i)}$, we proceed  to bound each of the terms. For the first one, using that $\max_{t\in[0,T]}\theta\leq  (\delta T^2)^{-1}$ we get
\begin{align}\notag 
W^{(1)}&=\dbint_{Q} \dt\left(\tau T \taubm{\theta}^2 + \frac{\tau \dt}{\delta^3 T^4}\right)(\Dtbar z)^2 \\ \notag
&\leq  \left\{\frac{\dt \tau^2}{\delta^3 T^5}+ \left(\frac{\tau \dt }{\delta^2T^3}\right)^2\right\}\dbint_{Q} \taubm{s}^{-1}(\Dtbar z)^2 \\ \label{est_w33_app}
& \leq (\kappa_1+\kappa_1^2)\dbint_{Q}\taubm{s}^{-1}(\Dtbar z)^2.
\end{align}

For  $W^{(2)}$, we have
\begin{align}\notag 
W^{(2)}&=\dbint_{Q} \dt \taubm{s}^2  (\Dtbar z)^2=\dbint_{Q}\dt \taubm{s}^3 \taubm{s}^{-1} (\Dtbar z)^2 \\ \label{eq:est_Wm_app}
& \leq \kappa_{2} \dbint_{Q} \taubm{s}^{-1}(\Dtbar z)^2,
\end{align} 
where the condition 
\begin{equation}\label{eq:cond_kappa_2}
\frac{\tau^3\dt}{\delta^3 T^6}\leq \kappa_2
\end{equation}
holds for some $\kappa_2>0$ small enough. Finally, using \eqref{eq:prop_s}, we see that
\begin{align}\notag 
W^{(3)}&=\dbint_{Q}\left(T^2(\tau\dt)^2\theta^4+\frac{\tau^2(\dt)^4}{\delta^6 T^8}\right) (\Dtbar z)^2 \\ \notag
&\leq \tau^4(\dt^2)T^2\dbint_{Q}\taubm{\theta}^6\taubm{s}^{-1}(\Dtbar z)^2 + \left(\frac{\tau \dt}{\delta^3 T^4}\right)^4 \dbint_{Q}\taubm{s}^{-1}(\Dtbar z)^2 \\ \notag
& \leq \left\{\left(\frac{\tau^2 \dt}{\delta^3 T^5}\right)^2+\left(\frac{\tau \dt}{\delta^3 T^4}\right)^4\right\} \dbint_{Q}\taubm{s}^{-1}(\Dtbar z)^2\\ \label{eq:est_Wg_app}
& \leq \left(\kappa_1^2+\kappa_1^4\right)\dbint_{Q} \taubm{s}^{-1}(\Dtbar z)^2.
\end{align}
Since $\delta\leq 1/2$, $\tau\geq 1$ and $0<T<1$, we can combine conditions \eqref{eq:cond_kappa_1} and \eqref{eq:cond_kappa_2} into a single one verifying 
\begin{equation*}
\frac{\tau^4\dt }{\delta^4 T^6}\leq \epsilon_5
\end{equation*}
for some $\epsilon_5=\epsilon_5(\lambda)$ small enough. Collecting estimates \eqref{est_w33_app}, \eqref{eq:est_Wm_app}, and \eqref{eq:est_Wg_app} gives the desired result. 

\subsection{Proof of \Cref{lem:local_remove}}

We adapt the procedure used in the continuous setting. In particular, we follow \cite[pp. 1409]{FCG06}. Let us consider $\eta\in C_c^\infty(\Omega)$ such that
\begin{gather}\label{eq:prop_eta}
0\leq \eta\leq 1\text{ in } \Omega, \quad \eta=1 \text{ in a neighborhood of } \mathcal B_0, \quad \textnormal{supp}\, \eta \subset\subset \mathcal B.
\end{gather}
By the properties of discretization, we can ensure the additional property
\begin{equation}\label{eq:prop_diff_eta}
\frac{\difhb(\eta)}{\eta^{1/2}}\in L^\infty(\Omega)
\end{equation}
uniformly with respect to $h$. Obviously, the above function is supported within $\mathcal B$. 

By shifting the time integral and the definition of the function $\eta$, we have
\begin{align}\notag
\ddbint_{Q_{\mathcal B_0}}s|\difh z|^2 &\leq \ddbint_{Q}s\eta |\difh z|^2=-\ddbint_{Q}s \difhb(\eta \difh z)z \\ \label{eq:est_local}
&=-\ddbint_{Q} s\difhb(\eta) \ov{\difh z} z-\ddbint_{Q}s\ov{\eta}\delh z z =: \mathcal L_1+\mathcal L_2.
\end{align}

For the first term in the above expression, we have by using \eqref{eq:prop_diff_eta} and Cauchy-Schwarz and Young inequalities 
\begin{align}\notag
|\mathcal L_1| &\leq \frac{1}{2} \ddbint_{Q} \left|\frac{\difhb(\eta)}{\eta^{1/2}}\right|^2 |\ov{\difh z}|^2+\frac{1}{2}\dbint_{Q} s^2 \eta |z|^2 \\ \label{eq:est_local_1}
& \leq C \ddbint_{Q} |\ov{\difh z}|^2+C\ddbint_{Q}s^3\eta|z|^2 ,
\end{align}
where we have used that $s(t)\geq 1$ for all $t$ to adjust the power of $s$ in the last term of the above equation. 

For the term $\mathcal L_2$, using that $\ov{\eta}=\eta+h\mathcal O(1)$ together with Cauchy-Schwarz and Young inequalities 
\begin{align}\notag
|\mathcal L_2| &\leq \gamma \ddbint_{Q} s^{-1}\eta |\delh z|^2+ C_{\gamma}\ddbint_{Q}s^3\eta |z|^2 \\ \label{est:local_2}
& \quad + \frac{1}{2} \ddbint_{Q}s^{-1}(sh)^2|\delh z|^2+\frac{1}{2} \ddbint_{Q}s|z|^2
\end{align}
for any $\gamma>0$.

Using estimates \eqref{eq:est_local_1}--\eqref{est:local_2} in \eqref{eq:est_local} and recalling the properties of $\eta$ in \eqref{eq:prop_eta} gives the desired result after shifting the integral in time.

\renewcommand{\abstractname}{Acknowledgements}
\begin{abstract}
\end{abstract}
\vspace{-0.8cm}
The authors would like to thank Prof. Franck Boyer (Institut de Math\'ematiques de Toulouse) for some clarifying discussions about the works \cite{BHL10,BHLR11}. 

This work was partially supported by the programme ``Estancias posdoctorales por M\'exico'' of CONACyT, Mexico and by the National Autonomous University of Mexico (grant: PAPIIT, IN102116).

\bibliographystyle{alpha}
\bibliography{bib_fully}

\end{document}

%% file: macros_fully.tex
\def\mesh{\mathfrak{M}}
\def\dmesh{\overline\mesh}
\def\smesh{{\scriptscriptstyle \mesh}}
\def\sdmesh{{\scriptscriptstyle \overline{\mesh}}}
\def\sfullmesh{{\scriptscriptstyle \mesh\cup\partial\mesh}}

\def\Yint#1{\mathchoice
    {\YYint\displaystyle\textstyle{#1}}%
    {\YYint\textstyle\scriptstyle{#1}}%
    {\YYint\scriptstyle\scriptscriptstyle{#1}}%
    {\YYint\scriptscriptstyle\scriptscriptstyle{#1}}%
      \!\int}
\def\YYint#1#2#3{{\setbox0=\hbox{$#1{#2#3}{\int}$}
    \vcenter{\hbox{$#2#3$}}\kern-.52\wd0}}

\def\dint{\Yint{\textrm{---}}}

\newcommand\inter[1]{\llbracket #1\rrbracket}

 \usepackage{tikz}
 \usepackage{pgfplots}
 \usepackage{subfig}
 \pgfplotsset{surface/.style={ %
               xmax=1.1,%
               axis z line=center,%
               axis x line=center,%
               axis y line=center,%
               zmin=0,%
               clip=false,%
               extra x ticks={1},%
               extra x tick label={$T=1$},%
               xtick={0},%
               ytick=\empty}}

 \pgfplotsset{erreurs/.style={scale=1,
     legend cell align=left,
     legend pos=outer north east,
     legend plot pos=right,
     legend style={cells={anchor=east},draw=none},
     xlabel=$h$,
    xmin=0.001,xmax=0.03}}

\tikzset{pente/.style={opacity=0.6}}

\pgfplotsset{cout/.style={black,mark=diamond*,mark size=2.5,mark options={fill=gray}}}
\pgfplotsset{cible/.style={black,mark=square*,mark size=2.5,mark options={fill=gray}}}
\pgfplotsset{cibleyT/.style={black,mark=*,mark size=2.5,mark options={fill=gray}}}
\pgfplotsset{CG/.style={black,mark=otimes*,mark size=2.5,mark options={fill=gray}}}
\pgfplotsset{solex/.style={black,mark=*,mark size=2.5,mark options={fill=gray}}}
\pgfplotsset{energie/.style={black,mark=triangle*,mark size=2.5,mark options={fill=gray}}}
